\pdfoutput=1
\documentclass[a4paper]{article}
\usepackage{lmodern}

\usepackage{fullpage}

\usepackage[normalem]{ulem}
\usepackage[utf8]{inputenc}
\usepackage[T1]{fontenc}
\usepackage[english]{babel}

\usepackage{graphicx,amsfonts,amsmath,amssymb,amsthm,bm}
\graphicspath{{./}{figs/}}
\usepackage[authoryear,sort&compress]{natbib}

\PassOptionsToPackage{hypertexnames=false}{hyperref}
\usepackage{main-commands}
\usepackage{longtable}
\usepackage{booktabs}
\usepackage{array}

\usepackage{nicefrac}
\usepackage{enumitem} 
\usepackage{cleveref}
\numberwithin{equation}{section}

\usepackage{tikz}
\usetikzlibrary{shapes.geometric,calc,positioning}

\newcommand{\objmeas}{J_{\calB^\star}}
\newcommand{\objdual}{J_{\calH}}

\renewcommand{\inlPC}[1]{\begin{color}{red}\texttt{PC:#1}\end{color}}
\renewcommand{\comPC}[1]{\todo[color=red!20]{PC:#1}}

\renewcommand{\comPC}[1]{}
\renewcommand{\comYDC}[1]{}
\renewcommand{\inlPC}[1]{}

\usepackage{authblk}

\title{Fenchel--Young Duality Gaps:\\ Certified Early Stopping for Regularized Inverse Problems}

\date{September 2026}

\author[1]{{\small Pierre-Cyril} Aubin-Frankowski} 
\author[1,2,3]{{\small Yohann} De Castro}

\affil[1]{\it \small CERMICS, CNRS, ENPC, Institut Polytechnique de Paris, Marne-la-Vallée, France.}
\affil[2]{\it \small Institut Camille Jordan, École Centrale Lyon, CNRS UMR 5208, France.}
\affil[3]{\it \small Institut Universitaire de France (IUF).}

\begin{document}
\pagestyle{plain}
\setcounter{page}{1}

\maketitle

\begin{abstract}
    We study computable error bounds and certified early stopping for regularized inverse problems, where a data-fidelity term is traded against a regularizer. The analysis relies on an exact duality-gap identity that splits the total gap of $F(\Phi\mu)+\lambda R(\mu)$ into a data-fidelity Fenchel--Young loss and a regularizer Fenchel--Young loss,
    \[
        \Delta(\mu,h)=L_F(\Phi\mu\parallel h)+\lambda L_R(\mu\parallel\eta),\qquad \eta=-\Phi^\star h/\lambda,
    \]
    valid for any primal point $\mu$ and any dual point $h$. The data-fidelity term $F$ is strictly convex, so wherever $F^\star$ is differentiable the loss $L_F(\Phi\mu\parallel h)$ is the Bregman divergence of~$F$ between the prediction $\Phi\mu$ and $\nabla F^\star(h)$, and it vanishes exactly at \emph{Mirror Alignment} $h=\nabla F(\Phi\mu)$. Evaluated at a dual-feasible point $\tilde h$, the gap~$\Delta(\mu,\tilde h)$ is computable and \emph{oracle-free}, meaning that it uses no knowledge of the solution, and it bounds the suboptimality of $\mu$. Under the source condition, the same Fenchel--Young losses give \emph{a priori} bounds on the estimation and prediction errors. Their scale is the irreducible model and noise error $L_F(\Phi\mu^\star\parallel h^\star)$, which vanishes exactly when Mirror Alignment holds at the certificate. This gives an early-stopping rule: run the algorithm until the regularizer Fenchel--Young loss falls below a tolerance $\epsilon$. A constructive version of the Br\o ndsted--Rockafellar theorem then turns the current pair into an exact \emph{dual-feasible} one, and this proxy lifts to an exact primal certificate. We build the proxy by a proximal step in the geometry of the fidelity, with Bregman kernel $F^\star$ and tilted by the prediction $\Phi\mu$: it recovers the Euclidean step of Carlier when $F$ is the squared error, and it reduces the duality gap by the regularizer Fenchel--Young loss, up to a second-order remainder that vanishes in the quadratic case. Our running example is the Generalized Beurling--Lasso (GBL), where $R$ is the total-variation norm on signed measures. It contains the classical Beurling--Lasso, obtained with the squared error, and also covers robust, logistic, entropic and inverse-optimal-transport losses. The same duality gap certifies deep-learning optimizers such as Lion-K and Muon, in their proximal form, as solvers of the regularized program. A companion paper by the same authors builds on these error bounds to establish exact support recovery for the GBL under a non-degenerate source condition.
\end{abstract}



\section{Introduction}
Recovering structured signals from indirect, noisy measurements is a classical problem in inverse problems and signal processing. We focus on convex programs that combine a strictly convex \emph{data-fidelity} term $F:\calH\to\R$, acting on the prediction $\Phi\mu$ in a separable Hilbert space $\calH$, with a lower semi-continuous proper convex \emph{regularizer} $R$, acting on a Banach-dual variable $\mu$,
\begin{equation}
    \addtocounter{equation}{1}
    \tag{\theequation,\,$\mathrm{F\!+\!R}$}
    \label{eq:f_plus_r_intro}
    \min_\mu \Big\{\, F(\Phi\mu) + \lambda R(\mu) \,\Big\}\,,
\end{equation}
where $\Phi$ is a bounded weak-$*$ continuous forward operator and $\lambda>0$ trades fidelity against regularization. All the results of this paper (duality-gap decomposition, Fenchel--Young error bounds, Mirror Alignment, Br\o ndsted--Rockafellar early stopping, and its mirror-prox refinement) are developed at this level of generality; specific choices of $R$ (the total-variation norm of measures, group/joint TV, the nuclear norm, analysis sparsity, $\ldots$) instantiate the framework on the application domains catalogued in Section~\ref{sec:applications}.

A canonical instance, used as our running example throughout the paper, is the \emph{Generalized Beurling--Lasso} (GBL), obtained by specializing~\eqref{eq:f_plus_r_intro} to signed measures $\mu\in\calM(\calX)$ on a closed domain $\calX\subseteq\R^d$, possibly unbounded, with the total-variation norm $R(\mu)=\|\mu\|_{\TV}$:
\begin{equation}
    \addtocounter{equation}{1}
    \tag{\theequation,\,$\mathrm{GBL}$}
    \label{eq:gbl_intro}
    \min_{\mu \in \calM(\calX)}
        \Big\{
            F(\Phi\mu) + \lambda \|\mu\|_{\TV}
        \Big\}
    \,.
\end{equation}
The classical \emph{Beurling--Lasso} (BLASSO) of \citet{de2012exact,bredies2013inverse} is the further specialization of GBL to the squared-error data fidelity $F(h)=\tfrac{1}{2}\|h-y\|_{\calH}^2$ with measurement $y\in\calH$:
\begin{equation}
    \addtocounter{equation}{1}
    \tag{\theequation,\,$\mathrm{BLASSO}$}
    \label{eq:blasso_standard}
    \min_{\mu \in \calM(\calX)}
        \Big\{
            \tfrac{1}{2}\|\Phi\mu - y\|_{\calH}^2 + \lambda \|\mu\|_{\TV}
        \Big\}
    \,.
\end{equation}
Moving from BLASSO to the GBL form~\eqref{eq:gbl_intro}, and then to the abstract F+R program~\eqref{eq:f_plus_r_intro}, is needed to model settings where additive Gaussian noise is not appropriate, such as photon-limited imaging which uses Poisson regression; binary classification which uses logistic losses; robust regression which uses losses that tolerate heavy-tailed noise, such as the Huber loss; or inverse optimal transport which uses an entropic generating functional. The same Fenchel--Young geometry applies to all of these cases, and they are treated in this paper.

\subsection{Contributions}\label{subsec:contributions}
We organize our contributions as answers to two questions about the convex program \eqref{eq:f_plus_r_intro} and its GBL specialization~\eqref{eq:gbl_intro}. Our main message is that algebraic manipulations on Fenchel--Young gaps/losses are very informative and provide early stopping certificates.

\medskip\noindent
\textit{Q1. When does the duality gap give a computable error bound on the optimization error?}
We prove that for any primal variable $\mu$ and any dual variable $h$, the total duality gap admits an exact algebraic decomposition (Theorem~\ref{thm:dual_gap}):
\begin{equation*}
    \Delta(\mu,h) = L_F(\Phi\mu \parallel h) + \lambda L_R(\mu \parallel \eta)\,, \quad \text{where} \quad \eta = -\frac{\Phi^\star h}{\lambda}\,.
\end{equation*}
Here $L_F$ is the data-fidelity Fenchel--Young loss (equal, for any $h\in\operatorname{int}\dom F^\star$, to the Bregman divergence of~$F$ between $\Phi\mu$ and $\nabla F^\star(h)$, and vanishing exactly at \emph{Mirror Alignment} $h=\nabla F(\Phi\mu)$), and $L_R$ penalizes the violation of the subgradient condition for the regularizer. A primal--dual pair $(\mu,h)$ meeting both the source condition $\eta\in\partial R(\mu)$ and Mirror Alignment $h=\nabla F(\Phi\mu)$ is optimal, and conversely every primal solution $\mu$ admits such a dual $h$ (Theorem~\ref{thm:robust_align_first_order}); at optimality, primal prediction and dual variable are Legendre--Fenchel conjugates. The loss $L_F$ then gives two guarantees: (i) evaluated at a dual-feasible proxy $\tilde h$, the duality gap~$\Delta(\mu,\tilde h)$ is a \emph{computable}, oracle-free upper bound on the objective suboptimality (Theorem~\ref{thm:dual_gap}); and (ii) under the source condition, the same Fenchel--Young losses bound the estimation and prediction errors \emph{a priori} through the scalar $L_F(\Phi\mu^\star\parallel h^\star)$ corresponding to the irreducible model/noise error (Theorem~\ref{thm:FY_discrepancy_error_bounds}). 

\medskip\noindent
\textit{Q2. How early can an algorithm be stopped while still certifying admissibility?}
Practical algorithms drive the regularizer Fenchel--Young loss $\lambda L_R(\mu\parallel\eta)$ below a tolerance $\epsilon>0$, where $\eta\coloneqq -\Phi^\star h/\lambda$ is the candidate certificate associated to the dual iterate $h$. The classical Br\o ndsted--Rockafellar theorem produces a nearby exact subgradient pair, and \citet{carlier2023fenchel} makes this \emph{constructive} by realizing the proxy as a proximal step. This dual-feasible proxy lifts to a primal certificate $\bar\mu$ with an \emph{exact} source condition $\bar\eta\in\partial R(\bar\mu)$, unconditionally for a norm regularizer (Proposition~\ref{prop:BR_primal_lift}). Since nothing in the construction is tied to the quadratic geometry, we finally replace Carlier's Euclidean prox by a mirror proximal step with kernel $F^\star$ tilted by the prediction $\Phi\mu$: it obeys an exact duality-gap identity (Proposition~\ref{prop:mirror_gap_identity}) and, in the quadratic case, improves the gap by exactly $\lambda L_R(\mu\parallel\eta)$. 

\paragraph{Algorithmic implications.} Mirror Alignment also relates to several heuristics used in deep-learning optimization. The continuous-time trajectories of the Lion-K framework \citep{chen2023lionk} can be read as gradient flows of dual certificates of a continuous regularizer. The Muon optimizer \citep{jordan2024muon} performs matrix-sign (polar) orthogonalization of gradient updates onto the Stiefel manifold, which corresponds to the Mirror Alignment condition of a nuclear-norm penalty. Note that this polar step admits efficient implementations via the Polar Express algorithm \citep{amsel2026polar}. Section~\ref{subsec:exp_lion_muon} demonstrates this experimentally: when Lion-K's $\operatorname{sign}$ step and Muon's matrix-sign step are used as the proximal operator of their regularizer ($\ell_1$ and the nuclear norm), the resulting iterations coincide with ISTA (iterative shrinkage-thresholding) and SVT (singular-value thresholding), and hence solve the $F+\lambda R$ program. The Fenchel--Young duality-gap decomposition $\Delta = L_F + \lambda L_R$ (evaluated at a rescaled, dual-feasible certificate) certifies these iterates as shown in Figure~\ref{fig:lion_muon_delta} where the ill-conditioning of the problem makes convergence gradual: for both optimizers $\Delta_t$ decays steadily to zero and stays above the optimality gap, certifying (without any oracle access to the optimum) that they solve the $F+\lambda R$ program.

\begin{figure}[ht]
    \centering
    \includegraphics[width=\linewidth]{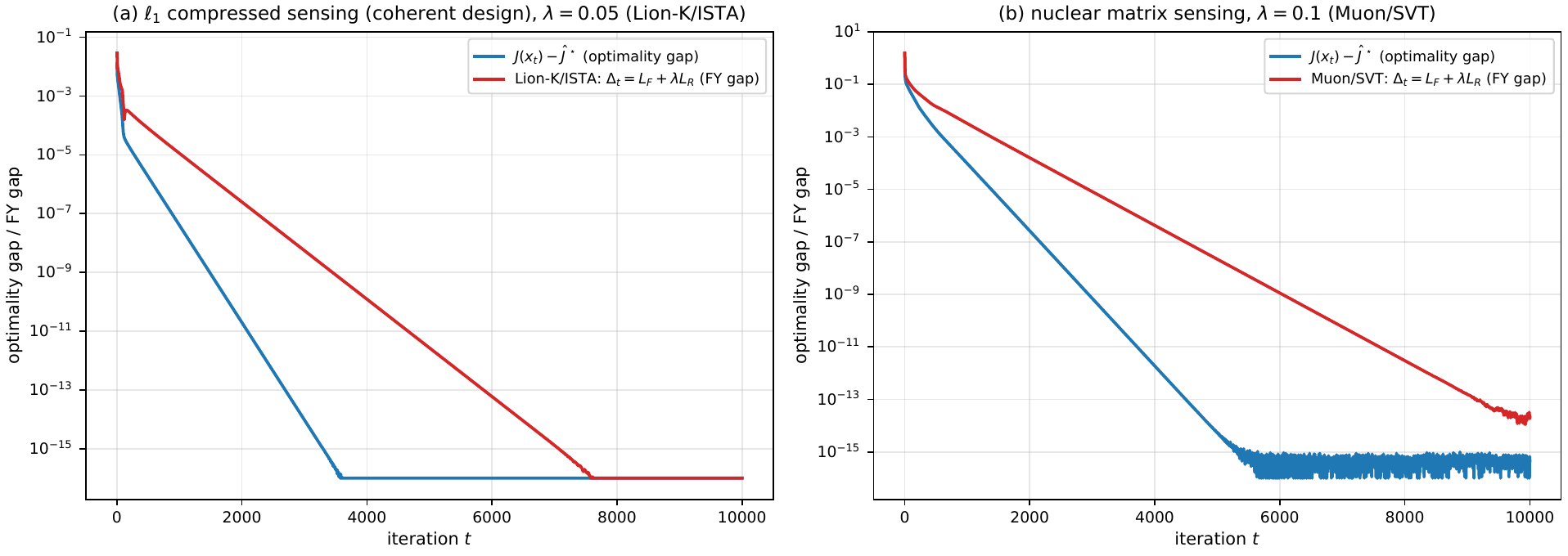}
    \caption{Fenchel--Young gap at a rescaled Mirror-Alignment dual, along the trajectories of two proximal $F+\lambda R$ solvers on ill-conditioned instances (double precision; first $10\,000$ iterations). \textbf{(a)} Coherent $\ell_1$ compressed sensing (Lion-K\,/\,ISTA: the $\operatorname{sign}$ step used as the $\ell_1$ proximal operator; $\mathrm{AR}(1)$-correlated design with $\rho=0.9$), penalty $\lambda=0.05$. \textbf{(b)} Nuclear-norm matrix sensing from $p=120$ Gaussian measurements (Muon\,/\,SVT: the $\operatorname{matsign}$ step used as the nuclear-norm proximal operator), penalty $\lambda=0.1$. In each panel, red: the oracle-free Fenchel--Young gap $\Delta_t$, which upper-bounds the (oracle) primal optimality gap $J(\mu_t)-\hat J^\star$ shown in blue, with $\hat J^\star$ the minimal objective over an extended reference~run.}
    \label{fig:lion_muon_delta}
\end{figure}

\paragraph{Companion paper.} Building on the Fenchel--Young error bounds of this paper, a companion work \citep{aubin2026localization} introduces a functional noise model and a non-degenerate source condition (NDSC; \citealp{duval2013exact,azais2015spike,poon2023geometry}) under which the GBL solution is unique, exactly $k$-sparse with the correct signs, and converges at a strict first-order spatial rate in $(\lambda,\|w\|_\calH)$; it verifies the NDSC through a Fisher-reweighted Local Positive Curvature pivot, obtains sign-free support localization from any non-degenerate certificate via a Kernel Switch principle, and instantiates the theory on the well-specified population logistic GBL with the sinc-4 pivot, attaining the parametric statistical rate up to a logarithmic factor. The present paper supplies the gap decomposition and the error bounds on which that analysis rests.

\subsection{Related work}
The mathematical foundations of exact recovery over spaces of measures were formalized by \citet{bredies2013inverse}, who established the well-posedness of Tikhonov regularization with total-variation penalties on the space of finite Radon measures. The first quantitative error bounds for variational regularization with a \emph{non-quadratic} convex penalty $R$ are due to \citet{burger2004convergence}, who proved that under the source condition $\Phi^\star\tilde h\in\partial R(\mu^\star)$ the generalized Bregman divergence $D_R(\hat\mu_\lambda,\mu^\star)$ between the Tikhonov minimizer and the $R$-minimizing solution decays at rate $\calO(\lambda+\delta)$ for $\lambda\asymp\delta$, with the explicit geometric specialization $D_R\equiv|\cdot|_{BV}$ on $BV(\Omega)$. Our framework recovers their identity as the regularizer-side component $\lambda L_R(\mu\parallel\eta^\star)$ of the duality gap (Proposition~\ref{prop:FYL_source_Bregman}) and generalizes the quadratic fidelity to arbitrary strictly convex $F$ via Fenchel--Young losses. Continuous greedy algorithms (generalized conditional gradient and its sliding variants \citealp{bredies2013inverse,boyd2017alternating,denoyelle2019sliding}) give $\calO(1/k)$ duality-gap rates for sparse deconvolution, but their analysis is tied almost exclusively to quadratic data fidelities.

Building on these conditions, the duality gap admits an exact decomposition into Fenchel--Young losses (see \cite{blondel2019learning} for a review of the latter). The geometric utility of these losses was highlighted by \citet{carlier2023fenchel}, who showed that a sharpened Fenchel--Young inequality controls the distance to the graph of the gradient mapping, ensuring measurement-space stability even when the primal is not unique; the same Br\o ndsted--Rockafellar bound (made \emph{constructive} in the same paper) is what we use for early-stopping certification. \citet{andrade2025learning} extended these sharpened losses to inverse entropic unbalanced optimal transport and to inverse Jordan--Kinderlehrer--Otto learning of population dynamics from independent snapshot samples, obtaining high-probability parameter recovery bounds. Other approximations of duality gaps exist, see \cite{Walwil2025} for a recent review, here we stick to the analysis for the true gap, tying it to Bregman divergences and Fenchel--Young losses.

\paragraph{Notation.}
For a set, $\operatorname{int}(\cdot)$ represents its topological interior. Let $\calX \subseteq \R^d$ be a closed domain with $\calX=\overline{\operatorname{int}\calX}$, possibly unbounded. We denote by $\calM(\calX)$ the space of finite signed Radon measures on~$\calX$, equipped with the total variation norm $\|\cdot\|_{\TV}$. Its pre-dual is the space~$\calC_0(\calX)$ of continuous functions vanishing at infinity, endowed with the supremum norm $\|\cdot\|_\infty$; for compact $\calX$, $\calC_0(\calX)=\calC(\calX)$. The duality pairing between $f \in \calC_0(\calX)$ and $\mu \in \calM(\calX)$ is denoted $\langle f, \mu \rangle_{\calC,\calM} = \int_{\calX} f\,\mathrm{d}\mu$; for a general pre-dual pair, $\langle f, \mu \rangle_{\calB,\calB^\star}$ denotes the pairing between $\calB$ and $\calB^\star$. We let~$\calH$ denote a separable Hilbert space with inner product $\langle \cdot, \cdot \rangle_{\calH}$ and norm $\|\cdot\|_{\calH}$. The forward operator is a linear map $\Phi: \calM(\calX) \to \calH$, continuous from the weak-$*$ topology of $\calM(\calX)$ to the weak topology of~$\calH$ (hence norm-bounded), with its pre-adjoint denoted by $\Phi^\star: \calH \to \calC_0(\calX)$. For a convex function $F$, $F^\star$ denotes its Legendre--Fenchel conjugate, $\dom(F)$ its domain and $\partial F$ its subdifferential. For a closed convex set $C$, $\iota_C$ denotes its indicator function $($zero on $C$ and $+\infty$ outside$)$ and, for $x\in C$, $N_C(x)\coloneqq\partial\iota_C(x)$ its outward normal cone, that is the set of $v$ in the dual of the ambient space with $\langle v,y-x\rangle\le0$ for every $y\in C$; for $C\subset\calH$ the pairing is the inner product of~$\calH$. Furthermore, we denote the set of non-negative real numbers by $\mathbb{R}_+$. Given a discrete measure $\mu = \sum_{i=1}^k a_i \delta_{x_i}$, we denote its support by $S = \{x_1, \dots, x_k\}$ and the signs of its amplitudes by $s_i = \operatorname{sign}(a_i)$. For $z\in\mathbb{R}^n$ and $\delta>0$, $B_\infty^\delta(z)\coloneqq\{w\in\mathbb{R}^n:\|w-z\|_\infty\le\delta\}$ denotes the closed $\ell_\infty$-ball of radius $\delta$ about $z$, and $B_\infty^\delta\coloneqq B_\infty^\delta(0)$.

\section{Fenchel duality, gap decomposition, and error bounds}
\label{sec:controlled_gaps}
This section starts with the GBL example, stating the source condition and Mirror Alignment informally, and then develops the abstract framework (duality-gap decomposition and Fenchel--Young error bounds) that makes them precise; the Br\o ndsted--Rockafellar algorithmic-stability theory, its primal-certificate lift, and a mirror-prox refinement follow in Section~\ref{sec:approx_source_cond}.

\subsection{Setting and assumptions}
\begin{subequations}

To foster intuition, we begin with the GBL specialization~\eqref{eq:gbl_intro}, where the primal is a signed measure $\mu\in\calM(\calX)$ and the regularizer is the total variation norm,
\begin{equation}
    \label{eq:gbl_problem}
    \inf_{\mu\in\calM(\calX)}\Big\{\, J_{\mathrm{GBL}}(\mu) \;\coloneqq\; F(\Phi\mu) + \lambda\|\mu\|_{\TV}\,\Big\},
\end{equation}
and $F:\calH\to\R$ is a strictly convex, continuously differentiable data-fidelity term and $\lambda>0$ is the regularization parameter. Varying $F$ encompasses several standard problems: \emph{continuous sparse regression} (least squares; $F(h)=\tfrac12\|h-y\|^2_\calH$, the classical BLASSO of \citealp{de2012exact,bredies2013inverse}); \emph{classification BLASSO} (logistic or squared-hinge $F$); \emph{counting/intensity estimation} (Poisson or KL); and \emph{robust continuous sparse regression} (Huber). The full catalogue is treated in Section~\ref{sec:applications}. The Fenchel pre-dual of~\eqref{eq:gbl_problem} maximizes over an auxiliary variable $h\in\calH$:
\begin{equation}
    \label{eq:gbl_dual}
    \sup_{h\in\calH}\,\Big\{-F^\star(h)\;\;\text{subject to}\;\;\|\Phi^\star h\|_\infty\le\lambda\,\Big\}.
\end{equation}
The constraint $\|\Phi^\star h\|_\infty\le\lambda$ enforces that the continuous function $\Phi^\star h$ stays bounded uniformly on $\calX$, preventing dual residuals from exceeding the sparsity threshold. Strong duality between~\eqref{eq:gbl_problem} and~\eqref{eq:gbl_dual} is established below as a consequence of the general F+R duality results.
\end{subequations}

\medskip\noindent
The first-order condition for \eqref{eq:gbl_problem} reads $0\in\partial(F\circ\Phi+\lambda R)(\mu^\star)$. Since the sum of subdifferentials is always contained in the subdifferential of the sum, it is enough to find $\mu^\star$ such that $0\in\partial(F\circ\Phi)(\mu^\star)+\partial(\lambda R)(\mu^\star)$; conversely, every primal solution admits such a dual variable (Theorem~\ref{thm:robust_align_first_order}). As $F$ is differentiable, $\partial(F\circ\Phi)(\mu^\star)=\{\Phi^\star\nabla F(\Phi\mu^\star)\}$, so, introducing (i) $h^\star=\nabla F(\Phi\mu^\star)$ and (ii) $\eta^\star\in\partial R(\mu^\star)$, the problem becomes equivalent to having (i) and (ii) satisfied together with (iii) $\eta^\star=-\Phi^\star h^\star/\lambda$. The combination of (ii) and (iii) asks the certificate $-\Phi^\star h^\star/\lambda$ to lie in $\partial\|\mu^\star\|_{\TV}$, and is known as the \emph{source condition}; condition (i) is what we call \emph{Mirror Alignment}. Both (i) and (ii) are equality cases of a Fenchel--Young inequality, that is, points of vanishing Fenchel--Young gap, which will be key for our reasoning below.

\paragraph{Source condition.} The candidate certificate $\eta^\star=-\Phi^\star h^\star/\lambda$ belongs to $\calC_0(\calX)$. Concretely, for a $k$-sparse measure $\mu^\star=\sum_{i=1}^k a_i^\star\delta_{x_i^\star}$ with signs $s_i=\operatorname{sign}(a_i^\star)$, the source condition asks that the continuous function $\eta^\star$ interpolates the signs on the support and is bounded by one everywhere:
\[
    \eta^\star(x_i^\star)=s_i\quad\text{for }i=1,\dots,k,\qquad \|\eta^\star\|_\infty\le 1.
\]
Asking in addition the strict inequality $|\eta^\star(x)|<1$ off the support is the stronger \emph{non-degenerate} source condition studied in the companion paper~\citep{aubin2026localization} and well known in the literature~\citep{duval2013exact,azais2015spike,poon2023geometry}.

\paragraph{Mirror Alignment.} Condition (i) reads backwards as $\Phi\mu^\star=\nabla F^\star(h^\star)$. We call this \emph{Mirror Alignment}, because it generalizes the mirror-descent primal/dual identity beyond the squared loss.

\paragraph{Framework and notation.}
We work over a generic Banach-dual primal space $\calB^\star$ and a separable Hilbert measurement space $\calH$. The framework rests on a triple of spaces and a pair of functions.

\begin{de}[Set $\calG_0$]
\label{def_assumption:spaces}
    We define $\calG_0$ as the set of triples $(\calH,\calB,\Phi)$ where $(\calH,\langle\cdot,\cdot\rangle_{\calH})$ is a \emph{separable Hilbert} space; $\calB$ is a Banach space with dual space $\calB^\star$; and $\Phi\,:\,\calB^\star\to\calH$ is a linear operator that is \emph{weak-$*$-to-weak continuous}, i.e.\ continuous from $(\calB^\star,\sigma(\calB^\star,\calB))$ to $\calH$ with its weak topology. The operator~$\Phi$ is referred to as the \emph{forward operator}.
\end{de}
\noindent
 By Lemma~\ref{lem:dual_Phi_banach}, the pre-adjoint $\Phi^\star\,:\,\mathcal{H}\to\calB$ maps into the pre-dual space $\calB$, itself identified with a subspace of the bidual $\calB^{\star\star}$.

\medskip

\begin{remark}[The signed measures example]
\label{rem:signed_measures}
   Let $\calX\subseteq\mathbb{R}^d$ be a \emph{closed} set, possibly unbounded, and let $\calB=(\calC_0(\calX),\|\cdot\|_\infty)$ be the space of continuous functions \emph{vanishing at infinity} on $\calX$ and $\calB^\star=(\calM(\calX),\|\cdot\|_{\TV})$ its dual space, identified with the space of finite signed Radon measures on $\calX$ by the Riesz--Markov theorem \citep[Theorem~6.19]{rudin1987real}; for compact $\calX$ one has $\calC_0(\calX)=\calC(\calX)$ and recovers the classical setting. By Lemma~\ref{lem:dual_Phi}, for any linear weak-$*$-to-weak continuous operator~$\Phi$, one has
\[
        \Phi\mu \,=\, \int_{\calX} \varphi_t\,\mathrm d\mu(t)\,,\quad\mu\in\calM(\calX)\,,
\]
where $\varphi_t\coloneqq \Phi\delta_t$ with $\delta_t$ the Dirac measure at $t\in\calX$; and the dual operator $\Phi^\star\,:\,\calH\to\calC_0(\calX)$ reads
\[
    \Phi^\star\,:\,h\in\calH\mapsto \big(t\in\calX\mapsto\langle\varphi_t,h\rangle_{\calH}\big)\in\calC_0(\calX)\,,
\]
where we identified the pre-dual space $\calC_0(\calX)$ with a subspace of the dual $\calM(\calX)^\star$. For unbounded $\calX$, membership of $\Phi^\star h$ in $\calC_0(\calX)$ is a real restriction on the forward operator: since $\delta_t\rightharpoonup^* 0$ as $|t|\to\infty$, weak-$*$-to-weak continuity forces the features to vanish weakly, $\varphi_t\rightharpoonup 0$ in $\calH$. Observables that do not vanish at infinity are excluded, such as the total mass $\mu\mapsto\mu(\calX)=\langle\mathbf{1},\mu\rangle$, admissible only for compact~$\calX$ since $\mathbf{1}\notin\calC_0(\calX)$. Translation-invariant feature maps on $\calH=L^2(\R^d)$ are however admissible, their translates converging weakly to zero. Among the examples of Section~\ref{sec:applications}, only the total variation of the gradient (Section~\ref{subsubsec:TV_gradient}) requires a bounded domain.
\end{remark}

\begin{example}[Mixture of Gaussians]
\label{ex:gaussian_mixture}
Take $\calX=\R^d$, $\calH=L^2(\R^d)$ and let the feature $\varphi_t$ be the Gaussian density on $\R^d$ centered at $t$, with covariance matrix $\sigma^2\Id$. The features are translates of one fixed function of $L^2(\R^d)$, so they converge weakly to zero as $|t|\to\infty$ and the triple $(\calH,\calB,\Phi)$ belongs to $\calG_0$ even though $\calX=\R^d$ is unbounded. The prediction
\[
    \Phi\mu=\int_{\R^d}\varphi_t\,\mathrm d\mu(t)
\]
is a mixture of Gaussians with mixing measure $\mu$, and the GBL estimates $\mu$ from the mixture. With the squared norm $F=\frac12\|\cdot-y\|_\calH^2$ this is the BLASSO for mixture models of \citet{de2021supermix}, where the data $y$ is built from the empirical measure of a sample of the mixture and $\calH$ is taken to be a reproducing kernel Hilbert space, so that this empirical measure can be compared with the prediction. The present paper covers the other data-fidelity terms of Table~\ref{tab:common_losses} as well, in particular the logistic loss of Section~\ref{sec:datafidelity_FYL_app:logistic}, for which the same duality gap, certificates and error bounds hold.
\end{example}

\noindent
We also need to consider the following functions.
\begin{de}[Set $\calF_0$]
\label{def_assumption:fidelity}
    Given $(\calH,\calB,\Phi)\in\calG_0$, we define $\calF_0$ as the set of pairs $(F,R)$ where $F\,:\,\calH\to \mathbb{R}$ is a \emph{continuously differentiable strictly convex} function such that there exist~$a>0$ and $b\in\mathbb R$ with
\begin{equation}
    \label{eq:hyp_coercivite_faible}
    \forall h\in\calH\,,\quad F(h)\geq a\|h\|_{\calH}+b\,;
\end{equation}
    $R\,:\,\calB^\star\to\mathbb{R}\cup\{+\infty\}$ is a proper convex and \emph{weak-$*$ lower semi-continuous} function $(\sigma(\calB^\star,\calB)$-l.s.c., equivalently $R=R^{\star\star}$ with the conjugate $R^\star$ taken on the pre-dual~$\calB)$, referred to as the \emph{regularizer}. The function $F$ is referred to as the \emph{data-fidelity term}. We also write $F\in\calF_0$ when $F$ alone satisfies the conditions above on the data-fidelity side.
\end{de}

\noindent
The lower bound~\eqref{eq:hyp_coercivite_faible} is a weak-coercivity condition: it makes $F$ bounded below, so that the objective is bounded below and, together with the coercivity supplied by the regularizer, the $F+\lambda R$ program admits a minimizer. The norm-type regularizers in Section~\ref{sec:applications} (total variation, group/joint TV, nuclear norm) are weak-$*$ lower semi-continuous, as required in Definition~\ref{def_assumption:fidelity}, being suprema of weak-$*$ continuous linear functionals $($e.g.\ $\|\mu\|_{\TV}=\sup_{\|f\|_\infty\le1}\langle f,\mu\rangle)$; the constrained regularizers (analysis sparsity, TV of the gradient) are weak-$*$ l.s.c.\ provided the underlying constraint set is weak-$*$ closed, see the standing hypothesis in Section~\ref{subsubsec:analysis_sparsity}.

\medskip

Now, consider the following convex program on the dual space $\calB^\star$
\begin{equation}
    \label{eq:convex_program_J}
    \inf_{\mu\in\calB^\star}\Big\{
    \objmeas(\mu)\coloneqq F(\Phi\mu)+\lambda R(\mu)
    \Big\}\,,
\end{equation}
where $\objmeas\,:\,\calB^\star\to \mathbb{R}\cup\{+\infty\}$ is the objective function.

\begin{subequations}
\label{eq:Fenchel_duals}
\paragraph{Legendre--Fenchel duality.} We define the \emph{Legendre--Fenchel conjugates} of~$F$ and $R$ as follows: for all $h\in\mathcal{H}$ and all $f\in\calB$,
\begin{align}
    F^\star(h)
        &\coloneqq \sup_{h'\in\calH}
    \big\{
        \langle h,h'\rangle_{\calH}-F(h')
    \big\}\in\mathbb{R}\cup\{+\infty\}
    \\
    R^\star(f)
        &\coloneqq \sup_{\mu\in\calB^\star}
    \big\{
        \langle f,\mu\rangle_{\calB,\calB^\star}-R(\mu)
    \big\}\in\mathbb{R}\cup\{+\infty\}
\end{align}
where $R^\star$ is only given on the pre-dual space $\calB$, identified with a weak-$*$ dense subset of the bidual $\calB^{\star\star}$.
\end{subequations}

\begin{subequations}
\label{eq:Fenchel_Young_loss}
\paragraph{Fenchel--Young losses.}
We also consider the \emph{Fenchel--Young losses}~\citep{blondel2019learning}, defined as follows: for all $h,h'\in\calH$ and all $(f,\mu)\in\calB\times\calB^\star$,
\begin{align}
    L_F(h'\parallel h)
        &\coloneqq F(h')+F^\star(h)-\langle h,h'\rangle_{\calH}
    \in\mathbb{R}_+\cup\{+\infty\}\,,
    \\
    L_R(\mu\parallel f)
        &\coloneqq R(\mu)+R^\star(f)-\langle f,\mu\rangle_{\calB,\calB^\star}
    \in\mathbb{R}_+\cup\{+\infty\}\,.
    \label{eq:LR_def}
\end{align}
\end{subequations}
Fenchel--Young losses can be understood as a generalization with primal--dual variables of the more common concept of Bregman divergence, written instead on primal--primal variables. We now recall this second notion and state the identity that links the two.

\begin{subequations}
\label{eq:Bregman_identities_primal_and_dual}
\paragraph{Bregman divergences.} Given a convex Gâteaux-differentiable function $G: \calH \to \mathbb R$, its Bregman divergence is, for all $x,y\in\calH$,
\[
    G(x \mid y)\;\coloneqq \; G(x) - G(y) - \langle \nabla G(y),\, x-y \rangle_\calH \;\in\; \mathbb R_+.
\]
For $F$ and its convex conjugate $F^\star$, we use for all $u,v\in\calH$ and all $h\in\calH$, $h'\in\operatorname{int}(\dom(F^\star))$,
\begin{align}
    F(u \mid v)
        &\coloneqq  F(u) - F(v) - \langle \nabla F(v),\, u-v \rangle_\calH\;\in\; \mathbb R_+, \\
    F^\star(h \mid h')
        &\coloneqq  F^\star(h) - F^\star(h') - \langle \nabla F^\star(h'),\, h-h' \rangle_\calH\;\in\; \mathbb R_+ \cup\{+\infty\}.
\end{align}
The function $F$ is assumed differentiable and strictly convex so that $F^\star$ is also differentiable on $\operatorname{int}(\dom(F^\star))$. Indeed, strict convexity of $F$ ensures that the supremum defining~$F^\star(h)$ admits at most one maximizer. On $\operatorname{int}(\dom(F^\star))$, where subgradients are guaranteed to exist, this uniqueness implies that $F^\star$ is Gâteaux differentiable; see \citet{zalinescu2002convex} for the infinite-dimensional Hilbert setting. The two notions then agree, for all $u\in\calH$ and all $h\in\operatorname{int}(\dom(F^\star))$,
\begin{equation}
    \label{eq:FY_equals_Bregman}
    F(u \mid \nabla F^\star(h)) = F^\star(h \mid \nabla F(u)) = L_F(u \parallel h)\,,
\end{equation}
which is Lemma~\ref{lem:bregman_identity}. Fenchel--Young losses have the advantage over Bregman divergences of not requiring differentiability of the functional.
\end{subequations}

\subsection{Results on duality and optimality gaps}
The \emph{dual program} of~\eqref{eq:convex_program_J} (strictly speaking its Fenchel \emph{pre-dual}, since $\calB^\star$ is itself a dual space) reads
\begin{equation}
    \label{eq:dual_program}
    \sup_{h\in\calH}\Big\{\
    \objdual(h)\coloneqq -F^\star(h)-\lambda R^\star\Big(-\frac{\Phi^\star h}{\lambda}\Big)
    \Big\}\,,
\end{equation}
where $\objdual\,:\,\calH\to \mathbb{R}\cup\{-\infty\}$ is the objective function.

\paragraph{Duality-gap decomposition.} \Cref{thm:dual_gap} below expresses the duality gap in terms of Fenchel--Young losses. 

\begin{theo}[Fenchel--Young losses and duality gap]
\label{thm:dual_gap}
Let $(\calH,\calB,\Phi)\in\calG_0$, let $F\,:\,\calH\to\mathbb{R}\cup\{+\infty\}$ be proper, l.s.c.\ convex, let $R\,:\,\calB^\star\to\mathbb{R}\cup\{+\infty\}$ be proper convex weak-$*$ l.s.c., and let $\lambda>0$. Fix $\mu\in\dom(R)\subset\calB^\star$ with $\Phi\mu\in\dom(F)$, and $h\in\dom(F^\star)\subset\calH$. Then the \emph{duality gap}~$\Delta(\mu,h)$ satisfies
\begin{equation}
    \label{eq_thm:dual_gap}
\Delta(\mu,h)\;\coloneqq \; \objmeas(\mu)-\objdual(h)
=
L_F(\Phi\mu \parallel h)
    \;+\;\lambda\,
    L_R(\mu \parallel \eta)\;\in\;\mathbb{R}_+\cup\{+\infty\}\,.
\end{equation}
where $\eta\coloneqq -\Phi^\star h/\lambda$.
\end{theo}

\noindent
When $(F,R)\in\calF_0$ the requirement $\Phi\mu\in\dom(F)$ is automatic, since $\dom(F)=\calH$. Note that the gap is finite (and the identity holds in $\R_+$) precisely when $\eta\in\dom(R^\star)$, equivalently $h\in\dom(\objdual)$; otherwise $L_R(\mu\parallel\eta)=+\infty$, $\objdual(h)=-\infty$, and both sides equal $+\infty$.

\begin{proof}
First, we verify that the functional $\objdual$ defined in \eqref{eq:dual_program} corresponds to the Fenchel pre-dual of the primal problem \eqref{eq:convex_program_J}. We write the Fenchel--Rockafellar dual (see \citealp[Section~2.8]{zalinescu2002convex}) of the objective~\eqref{eq:convex_program_J} rewritten as $\inf_{\mu}\{ F(\Phi \mu) + g_\lambda(\mu)\}$, where $g_\lambda(\mu) \coloneqq \lambda R(\mu)$:
\begin{equation}
\notag
    \sup_{h \in \calH} \big\{-F^\star(h) - g_\lambda^\star(-\Phi^\star h) \big\}.
\end{equation}
We compute the conjugate of the scaled regularizer $g_\lambda(\mu) = \lambda R(\mu)$, giving $g_\lambda^{\star}(v)= \lambda R^{\star}({v}/{\lambda})$; substituting into the dual formulation, we match exactly the definition of $\objdual(h)$ in~\eqref{eq:dual_program}. Thus, $\objdual$ is the canonical (pre-)dual.\footnote{One can start from~\eqref{eq:dual_program} and, by the same algebraic manipulations on the Legendre--Fenchel transform together with $F=F^{\star\star}$ and $R=R^{\star\star}$ (weak-$*$ l.s.c.), deduce that~\eqref{eq:convex_program_J} is the dual problem to \eqref{eq:dual_program}. For the sake of readability, we present its algebraic proof the other way around using implicitly that $\calB\hookrightarrow\calB^{\star\star}$.}

\medskip\noindent
If $\eta=-\Phi^\star h/\lambda\notin\dom(R^\star)$, then $R^\star(\eta)=+\infty$, hence $L_R(\mu\parallel\eta)=+\infty$ and $\objdual(h)=-\infty$, so $\Delta(\mu,h)=+\infty$ matches the right-hand side and the identity holds in $\R_+\cup\{+\infty\}$. Assume henceforth $\eta\in\dom(R^\star)$, so all terms below are finite. We add and subtract the same terms to recognize the Fenchel--Young losses and we use Lemma~\ref{lem:dual_Phi_banach}:
    \begin{align*}
        \objmeas(\mu)&=F(\Phi\mu)+\lambda R(\mu)\\
        &=F(\Phi\mu)+F^\star(h)-\langle h,\Phi\mu\rangle_\calH
        +\lambda \big(R(\mu) +R^\star(\eta)-\langle \eta,\mu\rangle_{\calB,\calB^\star}\big)
        -F^\star(h)-\lambda R^\star(\eta)\\
        &=L_F(\Phi\mu \parallel h)+\lambda L_R(\mu\parallel\eta)+\objdual(h)\,,
    \end{align*}
as claimed, where $\eta=-\Phi^\star h/\lambda$.
\end{proof}

\begin{remark}[Scope of the gap identity]\label{rem:gap_identity_scope}
The gap identity of Theorem~\ref{thm:dual_gap} uses none of the differentiability, strict convexity or coercivity of~$F$: the proof only adds and subtracts conjugate terms. Section~\ref{sec:applications} uses this generality for the losses outside~$\calF_0$ (logistic, Huber, squared hinge, linear-link KL).
\end{remark}

The identity~\eqref{eq_thm:dual_gap} has direct consequences. By weak duality, it provides an upper bound on the optimality gap of $\objmeas$: for all $\mu\in\dom(R)$ with $\Phi\mu\in\dom(F)$ and all $h\in\dom(F^\star)$,
\begin{equation}\label{eq:opt_gap_bound}
\begin{aligned}
0\leq \objmeas(\mu)-\inf_{\nu\in\calB^\star}\objmeas(\nu)
&\leq\inf_{h'\in\calH}\Delta(\mu,h')\leq \Delta(\mu,h)
\stackrel{\eqref{eq_thm:dual_gap}}{=}
    L_F(\Phi\mu \parallel h)+\lambda\,
    L_R(\mu \parallel \eta)\,.
\end{aligned}
\end{equation}
Since the duality gap \eqref{eq_thm:dual_gap} is the sum of two non-negative Fenchel--Young losses, driving it to zero requires both terms to vanish. By a standard result recalled in Lemma~\ref{lem:fy_properties}, the condition $L_R(\mu \parallel \eta) = 0$ is equivalent to the subdifferential inclusion $\eta \in \partial R(\mu)$. In the context of inverse problems, this inclusion governs the structural properties of the solution, such as the sparsity promoted by the total variation norm. This motivates the following definition of the source condition.
\begin{de}[Dual certificate/source condition]
    \label{def:source_cnd}
    Let $(\calH,\calB,\Phi)\in\calG_0$, $R\,:\,\calB^\star\to\mathbb{R}\cup\{+\infty\}$ a proper convex weak-$*$ l.s.c.\ function, $\lambda>0$, and $\mu\in\calB^\star$. We say that $\eta\in\calB$ is a \emph{dual certificate of $\mu$ for $(\calH,\lambda,R,\Phi)$} if it satisfies the \emph{source condition at point $\mu$}: 
     there exists $h\in\calH$ such that 
    \begin{equation}
        \label{def:dual_certificate_definition}
        \eta = -\frac{\Phi^\star h}{\lambda}\in\partial R(\mu)\, ,
    \end{equation}
    where $\partial R$ denotes the subdifferential of~$R$, understood through the pairing $(\calB,\calB^\star)$: we only consider subgradients lying in~$\calB$.
\end{de}
\noindent
We can state the following theorem.
\begin{theo}[First-order optimality and Mirror Alignment]
\label{thm:robust_align_first_order}
Let $(\calH,\calB,\Phi)\in\calG_0$, $(F,R)\in\calF_0$ and $\lambda>0$.
Let $\mu\in\calB^\star$ and let $h\in\calH$ be such that the source condition \eqref{def:dual_certificate_definition} holds. Then
\[
0\le \objmeas(\mu)-\inf_{\nu\in\calB^\star}\objmeas(\nu)\leq L_F(\Phi\mu \parallel h)\,,
\]
and, when $h\in\operatorname{int}(\dom F^\star)$, 
$L_F(\Phi\mu \parallel h)=F(\Phi\mu \mid \nabla F^\star(h))=F^\star(h \mid \nabla F(\Phi\mu))$.
If additionally $h\in\calH$ satisfies \emph{one of these equivalent conditions}
\begin{subequations}
    \label{eq:mirror_alignment_dual_certificate}
\begin{align}
    h&=\nabla F(\Phi\mu)\\
    \Phi\mu&=\nabla F^\star(h)
\end{align}
\end{subequations}
then $\mu$ is a solution to \eqref{eq:convex_program_J} and $h$ a solution to \eqref{eq:dual_program}. In~\eqref{eq:mirror_alignment_dual_certificate}, $\nabla F^\star(h)$ denotes the unique subgradient of $F^\star$ at $h$: by strict convexity of $F$, $\partial F^\star(h)$ contains at most one element, and it is nonempty whenever $h=\nabla F(\Phi\mu)$.

\emph{The converse is true:} If $\mu$ is a solution to \eqref{eq:convex_program_J}, then there exists a solution $h\in\calH$ to \eqref{eq:dual_program} such that~\eqref{def:dual_certificate_definition} holds
and \eqref{eq:mirror_alignment_dual_certificate} holds $($\emph{Mirror Alignment}$)$.
\end{theo}

\begin{proof}
Set $\eta\coloneqq-\Phi^\star h/\lambda$. By the source condition $\eta\in\partial R(\mu)$, so $L_R(\mu\parallel\eta)=0$ (Lemma~\ref{lem:fy_properties}), and the gap decomposition of Theorem~\ref{thm:dual_gap} gives $\Delta(\mu,h)=L_F(\Phi\mu\parallel h)$. By weak duality, $0\le\objmeas(\mu)-\inf_{\nu\in\calB^\star}\objmeas(\nu)\le\Delta(\mu,h)=L_F(\Phi\mu\parallel h)$, which is the first inequality. Moreover, by the zero-loss equivalence of Fenchel--Young losses (Lemma~\ref{lem:fy_properties}), $L_F(\Phi\mu\parallel h)=0$ iff $h=\nabla F(\Phi\mu)$ (Mirror Alignment) and $L_R(\mu\parallel\eta)=0$ iff $\eta\in\partial R(\mu)$ (the source condition); the two conditions together are therefore equivalent to $\Delta(\mu,h)=0$, i.e.\ to optimality of $(\mu,h)$. 

Conversely, suppose $\mu$ is a solution to~\eqref{eq:convex_program_J}. Since $(F,R)\in\calF_0$, the fidelity $F$ is finite and continuous on $\calH$ and $\dom R\neq\varnothing$, so $F$ is continuous at $\Phi\mu_0$ for some $\mu_0\in\dom R$ and the qualification of the Fenchel--Rockafellar duality theorem~\citep[Corollary~2.8.5]{zalinescu2002convex} holds; it yields strong duality together with attainment of the dual supremum in~\eqref{eq:dual_program} (the weak-coercivity bound~\eqref{eq:hyp_coercivite_faible} is not used here). Hence there is $h\in\calH$ with $\Delta(\mu,h)=0$.

By the gap identity of Theorem~\ref{thm:dual_gap} and the non-negativity of the Fenchel--Young losses, $0=\Delta(\mu,h)=L_F(\Phi\mu\parallel h)+\lambda L_R(\mu\parallel\eta)$ with $\eta=-\Phi^\star h/\lambda$ forces $L_F(\Phi\mu\parallel h)=0$ and $L_R(\mu\parallel\eta)=0$; the zero-loss equivalences used above then give Mirror Alignment $h=\nabla F(\Phi\mu)$ and the source condition $\eta\in\partial R(\mu)$, so $h$ solves~\eqref{eq:dual_program}.
\end{proof}

\begin{remark}[Mirror Alignment]
The optimality conditions of Theorem~\ref{thm:robust_align_first_order}, $L_R=0$ $($the source condition~\eqref{def:dual_certificate_definition}$)$ and $L_F=0$ $($Mirror Alignment~\eqref{eq:mirror_alignment_dual_certificate}$)$, recover the classical inverse-problem conditions \citep{bredies2013inverse}: the source condition is the subgradient inclusion $-\Phi^{\star} h / \lambda \in \partial R(\mu)$. Our formulation highlights \emph{Mirror Alignment} $h = \nabla F(\Phi \mu)$ $($equivalently $\Phi\mu = \nabla F^{\star}(h))$: whereas first-order conditions present the primal--dual relation algebraically $($linearly in the quadratic case$)$, Mirror Alignment reads it as the Legendre--Fenchel conjugacy induced by~$F$. This underlies the Fenchel--Young error bounds of Theorem~\ref{thm:FY_discrepancy_error_bounds} below, since $L_F$ coincides with a Bregman divergence of~$F$ under Mirror Alignment (Lemma~\ref{lem:bregman_identity}).
\end{remark}

The source condition~\eqref{def:dual_certificate_definition} is a sufficient condition for the Fenchel--Young loss of~$R$ to coincide with its \emph{generalized Bregman divergence} at the certified point, 
\[
    R_{\eta^\star\!}(\mu \mid \mu^\star) \coloneqq R(\mu)-R(\mu^\star)-\langle \eta^\star,\mu-\mu^\star\rangle_{\calB,\calB^\star}
\]
for $\eta^\star\in\partial R(\mu^\star)$, written $D_R$ in \citet{burger2004convergence}. The converse direction (Br\o ndsted--Rockafellar, Theorem~\ref{thm:algorithmic_stability}) shows that an $\epsilon$-bounded $L_R$ already yields a nearby exact certificate.

\begin{prop}[Regularizer Fenchel--Young loss equals the Bregman divergence under the source condition]
\label{prop:FYL_source_Bregman}
    Let $(\calH,\calB,\Phi)\in\calG_0$, $R\,:\,\calB^\star\to\mathbb{R}\cup\{+\infty\}$ a proper convex weak-$*$ l.s.c.\ function, $\lambda>0$, $\mu^\star\in\calB^\star$ and $\eta^\star\in\calB$ a \emph{dual certificate} of $\mu^\star$ for $(\calH,\lambda,R,\Phi)$ $($i.e.\ \eqref{def:dual_certificate_definition} holds at $\mu^\star)$. Then
    \begin{equation}
        \label{eq:FYL_reg_source_Bregman}
        \forall \mu\in\calB^\star\,,\qquad
        L_R(\mu \parallel \eta^\star) = R_{\eta^\star\!}(\mu \mid \mu^\star)
        \,.
    \end{equation}
\end{prop}

\begin{proof}
    By the source condition, $\eta^\star\in\partial R(\mu^\star)$, so $L_R(\mu^\star\parallel\eta^\star)=0$ (see Lemma~\ref{lem:fy_properties}), i.e.\ $R^\star(\eta^\star)=\langle\eta^\star,\mu^\star\rangle_{\calB,\calB^\star}-R(\mu^\star)$. Substituting into the Fenchel--Young loss~\eqref{eq:LR_def},
    \[
        L_R(\mu\parallel\eta^\star)=R(\mu)+R^\star(\eta^\star)-\langle\eta^\star,\mu\rangle_{\calB,\calB^\star}
        =R(\mu)-R(\mu^\star)-\langle\eta^\star,\mu-\mu^\star\rangle_{\calB,\calB^\star}=R_{\eta^\star\!}(\mu\mid\mu^\star),
    \]
    the generalized Bregman divergence, which is~\eqref{eq:FYL_reg_source_Bregman}.
\end{proof}

\subsection{Error bounds}
\label{subsec:error_bounds}

We now turn the duality-gap identity into a priori error bounds. Fix $\mu^\star\in\dom(R)$ and consider \emph{any} $h^\star\in\dom(F^\star)$ satisfying the source condition~\eqref{def:dual_certificate_definition} at $\mu^\star$, that is 
\begin{align*}
\text{Certificate:}\qquad & \eta^\star\coloneqq-\frac{\Phi^\star h^\star}{\lambda}\in\partial R(\mu^\star)\,. \\
\intertext{The message of this subsection is that a single scalar, given by}
\text{Error upper bound:}\qquad & L_F(\Phi\mu^\star\parallel h^\star)\,, \\
\intertext{controls the recovery errors of any near minimizer $\mu$ (making an improvement over $\mu^\star$) in the sense that}
\text{Set of near minimizers:}\qquad & \objmeas(\mu)\leq\objmeas(\mu^\star)\,.
\end{align*}
Theorem~\ref{thm:FY_discrepancy_error_bounds} states an exact identity, the resulting bound, and its split into an estimation part and a prediction part.

\begin{subequations}
\begin{theo}[Fenchel--Young error bounds]
\label{thm:FY_discrepancy_error_bounds}
    Let $(\calH,\calB,\Phi)\in\calG_0$, let $F\,:\,\calH\to\mathbb{R}\cup\{+\infty\}$ be proper, lower semi-continuous and convex, let $R\,:\,\calB^\star\to\mathbb{R}\cup\{+\infty\}$ be proper convex and weak-$*$ l.s.c., and $\lambda>0$. Let $\mu^\star\in\dom(R)$ and $h^\star\in\dom(F^\star)$ satisfy the source condition~\eqref{def:dual_certificate_definition} at $\mu^\star$, with $\objmeas(\mu^\star)<+\infty$, and set $\eta^\star\coloneqq-\Phi^\star h^\star/\lambda$. Then, for all $\mu\in\dom(R)$ with $\Phi\mu\in\dom(F)$, the following holds.

\noindent $\bullet$ One has 
    \begin{equation}\label{eq:obj_decomposition}
        \objmeas(\mu^\star)-\objmeas(\mu)=L_F(\Phi \mu^\star \parallel  h^\star) - L_F(\Phi \mu\parallel h^\star)-\lambda\, L_R(\mu \parallel \eta^\star)
        \,,
    \end{equation}
where $L_R(\mu \parallel \eta^\star)$ equals the Bregman divergence $R_{\eta^\star\!}(\mu \mid \mu^\star)$.

\noindent $\bullet$ If moreover $\objmeas(\mu)\leq\objmeas(\mu^\star)$, then the \emph{Fenchel--Young error bounds} hold:
\begin{equation}
\label{eq:obj_decomposition_bound}
0\leq
L_F(\Phi \mu \parallel h^\star) +\lambda\,
L_R(\mu \parallel \eta^\star)
\leq L_F(\Phi \mu^\star \parallel h^\star)
\,.
\end{equation}
\noindent $\bullet$ 
In particular for $\objmeas(\mu)\leq\objmeas(\mu^\star)$, as $\max\{\lambda L_R(\mu\parallel\eta^\star),\,L_F(\Phi\mu\parallel h^\star)\}\le L_F(\Phi\mu^\star\parallel h^\star)$, the two bounds below hold and are controlled simultaneously: the estimation bound $($on $\mu)$
\begin{equation}
\label{eq:obj_decomposition_estimation_bound}
0\leq\lambda L_R( \mu \parallel\eta^\star)\leq
L_F(\Phi \mu^\star \parallel h^\star)
\end{equation}
and the prediction bound $($on $\Phi\mu)$
\begin{equation}
\label{eq:obj_decomposition_prediction_bound}
0\leq L_F(\Phi \mu \parallel h^\star) \leq L_F(\Phi \mu^\star \parallel h^\star).
\end{equation}

\noindent $\bullet$
Finally, for any $\tau>0$,
\begin{equation}\label{eq:obj_decomposition_bound_Carlier}
    \frac{1}{2 \tau} \dist\nolimits_{\calH\times \calH}((\Phi\mu^\star,\tau h^\star),\graph(\tau\partial F))^2
    \le
    \frac{1}{\tau} \|\Phi \mu^\star-(\Id +\tau \partial F)^{-1}(\Phi \mu^\star +\tau h^\star)\|_\calH^2 
    \le L_F(\Phi \mu^\star \parallel h^\star)\,,
\end{equation}
where $\graph(\tau\partial F)$ denotes the graph in $\calH\times \calH$ of $\tau\partial F$, and $(\Id+\tau\partial F)^{-1}$ is the single-valued resolvent of the maximal monotone $\partial F$; for a smooth $F$, $\partial F=\{\nabla F\}$.
\end{theo}
\end{subequations}
\noindent
Note that \eqref{eq:obj_decomposition_bound_Carlier} is the sharpened Fenchel--Young inequality of \citet[Lemma~1.1]{carlier2023fenchel} which bounds the scalar $L_F(\Phi \mu^\star \parallel h^\star)$ from below by a squared distance to the graph of the subdifferential~$\partial F$ (the gradient $\nabla F$ when $F$ is smooth).

\begin{proof}
By the source condition, $L_R(\mu^\star \parallel \eta^\star)=0$ (Lemma~\ref{lem:fy_properties}), so $R^\star(\eta^\star)$ is finite and $h^\star\in\dom(\objdual)$. The gap identity of Theorem~\ref{thm:dual_gap}, applied at $(\mu^\star,h^\star)$ and at $(\mu,h^\star)$ gives, the finite dual value $\objdual(h^\star)$ cancelling,
\[
    \objmeas(\mu^\star)-\objmeas(\mu)=\Delta(\mu^\star,h^\star)-\Delta(\mu,h^\star)
    =L_F(\Phi\mu^\star\parallel h^\star)-L_F(\Phi\mu\parallel h^\star)-\lambda\, L_R(\mu\parallel\eta^\star)\,,
\]
which is~\eqref{eq:obj_decomposition}; the identification with the Bregman divergence stems from Proposition~\ref{prop:FYL_source_Bregman}. If $\objmeas(\mu)\leq\objmeas(\mu^\star)$, the left-hand side of~\eqref{eq:obj_decomposition} is non-negative and the upper bound in~\eqref{eq:obj_decomposition_bound} follows; the lower bound holds because both Fenchel--Young losses are non-negative (Lemma~\ref{lem:fy_properties}).

For~\eqref{eq:obj_decomposition_bound_Carlier}, the subdifferential $\partial F$ is maximal monotone ($F$ proper, l.s.c.\ and convex, by Rockafellar's theorem), so the resolvent $(\Id+\tau\partial F)^{-1}$ is single-valued and defined on all of $\calH$ (Minty); for a smooth $F$ it is $(\Id+\tau\nabla F)^{-1}$. \citet[Lemma~1.1]{carlier2023fenchel} applied to $F$ at $(\Phi\mu^\star,h^\star)$ gives $L_F(\Phi\mu^\star\parallel h^\star)\ge\tau^{-1}\|\Phi\mu^\star-\bar x\|_\calH^2$ with $\bar x\coloneqq(\Id+\tau\partial F)^{-1}(\Phi\mu^\star+\tau h^\star)$. 

Set $\bar p\coloneqq h^\star+\tau^{-1}(\Phi\mu^\star-\bar x)\in\partial F(\bar x)$ (the resolvent identity). Then, one has $\tau\|h^\star-\bar p\|_\calH=\|\Phi\mu^\star-\bar x\|_\calH$ and $(\bar x,\tau\bar p)\in\graph(\tau\partial F)$, so the lower bound equals $\tfrac{1}{2\tau}\bigl(\|\Phi\mu^\star-\bar x\|_\calH^2+\tau^2\|h^\star-\bar p\|_\calH^2\bigr)\ge\tfrac{1}{2\tau}\dist\nolimits_{\calH\times\calH}((\Phi\mu^\star,\tau h^\star),\graph(\tau\partial F))^2$.

Dropping either of the two non-negative terms on the left of~\eqref{eq:obj_decomposition_bound} gives~\eqref{eq:obj_decomposition_estimation_bound} and~\eqref{eq:obj_decomposition_prediction_bound}; both being non-negative, their maximum is at most their sum, hence at most $L_F(\Phi\mu^\star\parallel h^\star)$.
\end{proof}

\paragraph{Interpretation as error bounds.}
Theorem~\ref{thm:FY_discrepancy_error_bounds} bounds the estimation error (measured by $\lambda L_R$) and the prediction error (measured by $L_F$) by the single scalar $L_F(\Phi \mu^\star \parallel h^\star)$. Under the source condition this scalar is the duality gap at the certified pair, $L_F(\Phi\mu^\star\parallel h^\star)=\Delta(\mu^\star,h^\star)$. It vanishes exactly when Mirror Alignment holds at the certificate, $h^\star\in\partial F(\Phi\mu^\star)$ (Lemma~\ref{lem:fy_properties}), the single-valued $h^\star=\nabla F(\Phi\mu^\star)$ under~$\calF_0$; in that case both recovery errors are zero. 

\begin{remark}[Best certificate and bounds without the source condition]
\label{rem:best_certificate_free_bounds}
$(i)$ The bounds above are tightest for the certificate
\[
    h^\star \in \argmin_{h\in\calH}\Big\{L_F(\Phi \mu^\star \parallel h)
\,;\,
-\frac{\Phi^\star h}{\lambda}\in\partial R(\mu^\star)
\Big\}\,.
\]
One can check that a minimizer exists provided $\Phi\mu^\star\in\operatorname{int}\dom F$ (in particular whenever $(F,R)\in\calF_0$, where $\dom F=\calH$) and it is 
unique whenever $F$ is differentiable (in particular under $\calF_0$) and one minimizer lies in $\operatorname{int}\dom F^\star$. 

$(ii)$ Without the source condition, Theorem~\ref{thm:dual_gap} still gives error bounds. Let $\mu,\mu^\star\in\dom(R)$ and $h\in\dom(\objdual)$, and set $\eta\coloneqq-\Phi^\star h/\lambda$. Subtracting the gap identity at $(\mu^\star,h)$ and $(\mu,h)$ gives the identity $\objmeas(\mu^\star)-\objmeas(\mu)=\Delta(\mu^\star,h)-\Delta(\mu,h)$, so $\objmeas(\mu)\leq\objmeas(\mu^\star)$ implies
\[
0\leq L_F(\Phi \mu \parallel h) +\lambda L_R( \mu \parallel\eta)\leq \Delta(\mu^\star,h)=L_F(\Phi\mu^\star\parallel h)+\lambda L_R(\mu^\star\parallel\eta)\,.
\]
The bound $\Delta(\mu^\star,h)$ now carries the extra regularizer term; it reduces to the single scalar of Theorem~\ref{thm:FY_discrepancy_error_bounds} exactly under the source condition.
\end{remark}

\section{Approximate source conditions and algorithmic stability}
\label{sec:approx_source_cond}

Section~\ref{sec:controlled_gaps} characterized optimality through the source condition $\eta^\star\in\partial R(\mu^\star)$ holding \emph{exactly}. Practical continuous optimization algorithms (conditional gradient methods, Forward--Backward splitting, or non-convex particle flows like Lion-K) are subject to discretization error, finite stopping times, and non-convex landscapes, so they typically do not satisfy this inclusion. Early stopping yields instead an $\epsilon$-approximate source condition measured directly by the regularizer Fenchel--Young loss, as $\lambda L_R(\mu\parallel\eta)\le\epsilon$, where $\epsilon$ depends (not exclusively) on the number of (conditional) gradient steps, the dimension and the initialization, as described by the available rates of convergence of the chosen algorithm. We then show that any such $\epsilon$-approximate primal--dual pair can be upgraded to a \emph{dual-feasible} proxy. The classical theorem by \cite{Brondsted1965}, stated on Banach spaces, guarantees the \emph{existence} of a nearby exact subgradient pair based on an order argument via Zorn’s lemma and Bishop--Phelps lemma. \citet{carlier2023fenchel} makes the argument \emph{constructive} on Hilbert spaces, realizing the proxy as a proximal step: it lies within Hilbert distance $\sqrt{\tau\epsilon}$ of the iterate for every free balancing parameter $\tau>0$ (Theorem~\ref{thm:algorithmic_stability}), and the Fenchel--Young tolerance becomes an early-stopping criterion. We state the result for an arbitrary proper convex weak-$*$ lower semi-continuous $(\sigma(\calB^\star,\calB)$-l.s.c.$)$ regularizer $R$ and specialize it to the GBL total-variation norm afterwards (Remark~\ref{rem:prox_u_concrete}). Section~\ref{subsec:mirror_refinement} then refines the construction in the geometry of the data-fidelity term: the Euclidean prox is replaced by a mirror proximal step with kernel $F^\star$, tilted by the prediction $\Phi\mu$, which recovers Carlier's step when $F$ is quadratic and ties the proxy to the Mirror Alignment of Section~\ref{sec:controlled_gaps}.

\subsection{Hilbert pull-back of the regularizer}

Carlier's \emph{constructive} Br\o ndsted--Rockafellar theorem~\citep[Theorem~4.1]{carlier2023fenchel} is stated for convex functions on a Hilbert space, whereas $R$ acts on the Banach space $\calB^\star$. We hence pull $R$ back to $\calH$ through $\Phi^\star$ so that Carlier's result applies (Theorem~\ref{thm:algorithmic_stability}).

\begin{assumption}[Regularizer admissible for Br\o ndsted--Rockafellar]\label{ass:BR_regularizer} 
Assume that 
$(\calH,\calB,\Phi)\in\calG_0$, 
the regularizer $R:\calB^\star\to\R\cup\{+\infty\}$ is proper, convex and weak-$*$ lower semi-continuous, 
and $\lambda>0$. 
Define the \emph{pulled-back regularizer}
\begin{equation}
\label{eq:BR_pullback}
    u:\calH\to\R\cup\{+\infty\},\qquad
    u(h)\;\coloneqq\;\lambda\,R^\star\Big(-\frac{\Phi^\star h}{\lambda}\Big),
    \qquad
    \dom u=\Big\{h\in\calH:-\frac{\Phi^\star h}{\lambda}\in\dom R^\star\Big\}.
\end{equation}
We assume $\dom u\neq\varnothing$.
\end{assumption}

\noindent The pulled-back regularizer $u$ in~\eqref{eq:BR_pullback} is convex, l.s.c.\ and proper on $\calH$;
see Lemma~\ref{lem:pullback_regularity} for the proof. The feasibility requirement $\dom u\neq\varnothing$ holds
whenever $0\in\dom R^\star$, in particular for every norm regularizer, where $R^\star=\iota_{B_\calB}$ is the
indicator of the unit ball $B_\calB$ of the pre-dual and $0$ lies inside it, so $0\in\dom u$.

\begin{prop}[Hilbert pull-back and gap inequality]\label{prop:BR_gap_inequality}
Under \Cref{ass:BR_regularizer}, for every $\mu\in\dom R$ and every $h\in\dom u$, writing~$\eta\coloneqq-\Phi^\star h/\lambda$,
one has $u^\star(-\Phi\mu)\le\lambda R(\mu)$ and hence
\begin{equation}\label{eq:BR_gap_inequality}
    L_u(h\parallel-\Phi\mu)\;=\;u(h)+u^\star(-\Phi\mu)-\langle h,-\Phi\mu\rangle_\calH
    \;\le\;\lambda\,L_R(\mu\parallel\eta).
\end{equation}
Equality holds iff the supremum over the subspace already reaches the full value, i.e.\ iff
\[
\sup_{f\in\ran\Phi^\star}\big[\langle f,\mu\rangle_{\calB,\calB^\star}-R^\star(f)\big]=R(\mu)\,.
\]
\end{prop}

\noindent
Note that a sufficient condition for equality in~\eqref{eq:BR_gap_inequality} is that $\mu$ admits a dual certificate in $\ran\Phi^\star$, that is some $f\in\ran\Phi^\star$ with $f\in\partial R(\mu)$ (Lemma~\ref{lem:certificate_range_closure}(i)). For a norm regularizer $R=\|\cdot\|_{\calB^\star}$, in particular $R=\|\cdot\|_{\TV}$, a certificate in the closure $\overline{\ran\Phi^\star}$ is already enough, by rescaling (Lemma~\ref{lem:certificate_range_closure}(ii)). The converse of the range condition is then false: if $\ran\Phi^\star$ is dense in $\calB$ (a non-closed $\ran\Phi^\star$ is the generic situation for a compact $\Phi$), equality holds for every $\mu\in\dom R$ while there may be no certificate (Lemma~\ref{lem:certificate_range_closure}(iii)).

\begin{proof}
By definition,
\[
u^\star(-\Phi\mu)=\sup_{h'\in\calH}\langle-\Phi\mu,h'\rangle_\calH-\lambda R^\star(-\Phi^\star h'/\lambda).
\]
Set $f=-\Phi^\star h'/\lambda$; then
\[
\langle-\Phi\mu,h'\rangle_\calH=-\langle\mu,\Phi^\star h'\rangle_{\calB^\star,\calB}
=\lambda\langle f,\mu\rangle_{\calB,\calB^\star},
\]
and as $h'$ ranges over $\calH$, $f$ ranges over the subspace $\Phi^\star(\calH)$. Therefore
\[
    u^\star(-\Phi\mu)
    =\lambda\sup_{f\in\Phi^\star(\calH)}\big[\langle f,\mu\rangle_{\calB,\calB^\star}-R^\star(f)\big]
    \;\le\;\lambda\sup_{f\in\calB}\big[\langle f,\mu\rangle_{\calB,\calB^\star}-R^\star(f)\big]
    =\lambda R^{\star\star}(\mu)\;\underset{(\ast)}{=}\;\lambda R(\mu),
\]
with equality iff the supremum over the subspace $\Phi^\star(\calH)$ already equals the unconstrained one; here $(\ast)$ is the Fenchel--Moreau identity $R^{\star\star}=R$ for the dual pair $(\calB,\calB^\star)$, valid since $R$ is proper, convex and weak-$*$ l.s.c.\ (\Cref{ass:BR_regularizer}), with $R^\star$ taken on the pre-dual $\calB$ so that no reflexivity of $\calB$ is used (cf.\ Lemma~\ref{lem:fy_properties}). Finally,
using $u(h)=\lambda R^\star(\eta)$ and~$\langle h,-\Phi\mu\rangle_\calH=-\langle\Phi^\star h,\mu\rangle_{\calB,\calB^\star}=\lambda\langle\eta,\mu\rangle_{\calB,\calB^\star}$,
\[
    L_u(h\parallel-\Phi\mu)
    =\lambda R^\star(\eta)+u^\star(-\Phi\mu)-\lambda\langle\eta,\mu\rangle_{\calB,\calB^\star}
    \le\lambda\big(R(\mu)+R^\star(\eta)-\langle\eta,\mu\rangle_{\calB,\calB^\star}\big)
    =\lambda L_R(\mu\parallel\eta).\qedhere
\]
\end{proof}

\begin{remark}[Why the gap is evaluated at $(h,-\Phi\mu)$]\label{rem:gap_eval_point}
The slack between $L_u(h\parallel-\Phi\mu)$ and $\lambda L_R(\mu\parallel\eta)$ measures the failure of topological
certifiability, i.e.\ the gap between $u^\star(-\Phi\mu)$ and $\lambda R(\mu)$; it is generically non-zero
when $\Phi^\star(\calH)$ is not sup-norm dense in $\calB$.
The inequality direction is exactly what we need: an upper bound on $\lambda L_R$ propagates to an upper bound
on~$L_u$. 
\end{remark}

\subsection{Algorithmic stability}

In this section we show that any $\epsilon$-approximate primal--dual pair $(\mu,h)$ can be upgraded to a nearby \emph{dual-feasible} proxy $(\bar\mu,\bar h)$ satisfying the exact subgradient inclusion $\bar\eta\in\partial R(\bar\mu)$, with $\bar\eta=-\Phi^\star \bar h/\lambda$. The construction proceeds in two steps: a Br\o ndsted--Rockafellar proximal step on the pulled-back regularizer $u$ of~\eqref{eq:BR_pullback}, within Hilbert distance set by the Fenchel--Young tolerance $\epsilon$ (Theorem~\ref{thm:algorithmic_stability}); and its lift to a primal certificate $\bar\mu$ (Proposition~\ref{prop:BR_primal_lift}). Nothing in that step is specific to the quadratic geometry, and Section~\ref{subsec:mirror_refinement} builds the refinement adapted to $F^\star$. 
For all $\tau>0$ the proximal operator of $\tau u$ is defined~by
\begin{equation}
\notag
\begin{aligned}
    \operatorname{prox}_{\tau u}(h)
    \;\coloneqq\;\argmin_{h'\in\calH}\Big\{u(h')+\frac{1}{2\tau}\|h'-h\|_\calH^2\Big\}
    \;=\;\argmin_{h'\in\calH}\Big\{\frac{1}{2\tau}\|h'-h\|_\calH^2+\lambda R^\star\Big(-\frac{\Phi^\star h'}{\lambda}\Big)\Big\}\;.
\end{aligned}
\end{equation}

\begin{subequations}\label{eq:BR_proxy_subequations}
\begin{theo}[Algorithmic stability via Br\o ndsted--Rockafellar, general regularization]
\label{thm:algorithmic_stability}
Let \Cref{ass:BR_regularizer} hold and suppose an algorithm produces an approximate primal--dual pair~$(\mu,h)\in\dom R\times\dom u$
whose regularizer Fenchel--Young loss obeys
\begin{equation}\label{eq:BR_eps}
    \lambda\,L_R(\mu\parallel\eta)\;\le\;\epsilon,\qquad\text{where } \eta=-\frac{\Phi^\star h}{\lambda}.
\end{equation}
Then for every $\tau>0$ there exist $\bar h\in\dom u$ and $\bar v\in\partial u(\bar h)\subseteq\calH$ given by
\begin{equation}
\label{eq:BR_proximal_formula}
    \bar h\;\coloneqq\;\operatorname{prox}_{\tau u}\!\left(h - \tau\Phi\mu\right),\qquad
    \bar v\;\coloneqq\;-\Phi\mu + \tau^{-1}(h - \bar h),
\end{equation}
and satisfying
\begin{equation}\label{eq:BR_proxy}
    \|\bar h-h\|_\calH\;\le\;\sqrt{\tau\,\epsilon},\qquad
    \|\bar v-(-\Phi\mu)\|_\calH\;\le\;\sqrt{\epsilon/\tau}.
\end{equation}
\end{theo}
\end{subequations}
\medskip\noindent
Note that the proxy is \emph{dual-feasible}, $\bar\eta\coloneqq-\Phi^\star\bar h/\lambda\in\dom R^\star$, and $\bar v\in\partial u(\bar h)$ is an \emph{exact} subgradient inclusion for the pulled-back regularizer~$u$ of~\eqref{eq:BR_pullback}. Note also that the tolerance $\epsilon$ need not be small: one may take $\epsilon=\lambda L_R(\mu\parallel\eta)$, the running regularizer loss itself, so that the bounds~\eqref{eq:BR_proxy} hold at any iterate with no smallness assumption. 
\begin{proof}
By \Cref{prop:BR_gap_inequality} and~\eqref{eq:BR_eps}, $L_u(h\parallel-\Phi\mu)\le\lambda L_R(\mu\parallel\eta)\le\epsilon$.
We apply Carlier's constructive Br\o ndsted--Rockafellar theorem~\citep[Theorem~4.1]{carlier2023fenchel} to the convex l.s.c.\ proper function $u$ of~\eqref{eq:BR_pullback} at~$(x,p)=(h,-\Phi\mu)$:
for every $\tau>0$ the proximal point $\bar h=\operatorname{prox}_{\tau u}(h-\tau\Phi\mu)$ and $\bar v=-\Phi\mu+\tau^{-1}(h-\bar h)$ form a graph point $(\bar h,\bar v)\in\graph(\partial u)$
satisfying~\eqref{eq:BR_proxy}. Feasibility $\bar h\in\dom u$ gives $\bar\eta=-\Phi^\star\bar h/\lambda\in\dom R^\star$,
and $\bar v\in\partial u(\bar h)$ is the (exact) subgradient inclusion for the pulled-back regularizer $u$.
\end{proof}

\paragraph{Lift to a primal certificate.}
We now show that the proxy~$(\bar h,\bar v)$ of Theorem~\ref{thm:algorithmic_stability} lifts to a primal certificate $(\bar\mu,\bar\eta)$, with $\bar\eta=-\Phi^\star\bar h/\lambda$ and $\bar v=-\Phi\bar\mu$, under a mild composition constraint qualification. For a norm regularizer $R$ the qualification holds for every bounded~$\Phi$, compact or not, and the lift is unconditional.

\begin{prop}[Pulling the subgradient back to a primal certificate]\label{prop:BR_primal_lift}
Let $(\bar h,\bar v)$ be the Br\o ndsted--Rockafellar proxy of Theorem~\ref{thm:algorithmic_stability}, with $\bar\eta=-\Phi^\star\bar h/\lambda$. The inclusion $\bar v\in\partial u(\bar h)$ holds unconditionally. Suppose in addition the Robinson--Rockafellar constraint qualification
\[
  0\in\operatorname{int}\bigl(\dom R^\star-\ran\Phi^\star\bigr)
\]
holds or, more generally, the Attouch--Brezis condition that $\bigcup_{t\ge0}t(\dom R^\star-\ran\Phi^\star)$ is a closed subspace of $\calB$. %
Then the subdifferential chain rule holds with equality at $\bar h$,
\[
  \partial u(\bar h)=-\Phi\,\partial R^\star(\bar\eta)
  \qquad\Bigl(\text{GBL case of Remark~\ref{rem:prox_u_concrete}: }N_{K_\lambda}(\bar h)=-\Phi\,N_{\{\|\cdot\|_\infty\le1\}}(\bar\eta)\Bigr),
\]
and $\bar v$ lifts to a primal certificate: there exists $\bar\mu\in\calB^\star$ with $\bar v=-\Phi\bar\mu$ and $\bar\eta\in\partial R(\bar\mu)$, a source condition for $R$ at $\bar\mu$. 
\end{prop}
\begin{proof}[Proof of Proposition~\ref{prop:BR_primal_lift}]
The inclusion $\bar v\in\partial u(\bar h)$ is the proximal optimality established in the proof of Theorem~\ref{thm:algorithmic_stability}. Re-expressing it as a source condition needs the reverse inclusion $\partial u(\bar h)\subseteq-\Phi\,\partial R^\star(\bar\eta)$, i.e.\ the subdifferential composition rule $\partial(g\circ A)=A^\star\partial g(A\,\cdot)$ for $g=\lambda R^\star$ and $A=-\Phi^\star/\lambda$ (the inclusion $\supseteq$ is automatic). Under this constraint qualification the composition rule holds with equality and the primal minimizer is attained \citep[Section~2.8]{zalinescu2002convex}, so $\partial u(\bar h)=-\Phi\,\partial R^\star(\bar\eta)$; since $\bar v\in\partial u(\bar h)$, there is $\bar\mu\in\partial R^\star(\bar\eta)$ with $\bar v=-\Phi\bar\mu$, and $\bar\mu\in\partial R^\star(\bar\eta)$ is equivalent to $\bar\eta\in\partial R(\bar\mu)$ by Fenchel--Young duality.
\end{proof}
\noindent

For a norm regularizer the Robinson--Rockafellar condition holds for free, with no assumption on~$\Phi$: $\dom R^\star=B_\calB$ is the closed unit ball of the pre-dual, so $0\in\operatorname{int}\dom R^\star$ and $R^\star=\iota_{B_\calB}$ is finite and continuous there; since $0=A0\in\ran\Phi^\star$, the interior condition $0\in\operatorname{int}(\dom R^\star-\ran\Phi^\star)$ holds at once, and a fortiori the Attouch--Brezis condition, since $\bigcup_{t\ge0}t(\dom R^\star-\ran\Phi^\star)\supseteq\bigcup_{t\ge0}tB_\calB=\calB$ is a closed subspace \citep{attouch1986duality}. For a compact $\Phi$ with infinite-dimensional range (bandlimited or RKHS feature maps), $\ran\Phi^\star$ is not closed in $\calB$, yet the interior route above still holds for a norm regularizer, so the equality is unchanged. Closedness of $\ran\Phi^\star$ is not on its own a constraint qualification for the chain rule: the qualification bears on the pair $(\dom R^\star,\ran\Phi^\star)$, requiring the cone $\bigcup_{t\ge0}t(\dom R^\star-\ran\Phi^\star)$ to be a linear \emph{subspace} and not merely closed. For $\Phi=0$ the range $\{0\}$ is closed, yet this cone is the one generated by $\dom R^\star$, in general not a subspace, and the equality can fail: e.g.\ $R^\star(f)=-\sqrt{f}$ on $[0,\infty)$, where $\partial R^\star(0)=\varnothing$ while $\partial u(\bar h)=\{0\}$.

Consequently the proxy $\bar v$ therefore always arises from an actual primal measure $\bar\mu$ with $\bar\eta\in\partial R(\bar\mu)$, for every bounded $\Phi$. What a compact $\Phi$ can obstruct is a separate property: certificate existence in $\ran\Phi^\star$ for a prescribed measure $\mu$ (Remark~\ref{rem:gap_eval_point}, Lemma~\ref{lem:certificate_range_closure}(iii)), not the chain rule at $\bar h$.

\begin{remark}[Concrete proximal operators]
\label{rem:prox_u_concrete}
{\ }
\begin{enumerate}[label=$(\roman*)$]
    \item \textbf{Total variation (GBL).} For $R=\|\cdot\|_{\TV}$ the conjugate is $R^\star=\iota_{\{\|\cdot\|_\infty\le1\}}$, so $u=\iota_{K_\lambda}$ is the indicator of the dual-feasible set 
    \[
        K_\lambda\coloneqq\{h'\in\calH:\|\Phi^\star h'\|_\infty\le\lambda\}
    \] 
    feasible since $0\in K_\lambda$, with $\partial u=N_{K_\lambda}$ the outward normal cone and $\operatorname{prox}_{\tau u}=\mathcal{P}_{K_\lambda}$, the projection onto $K_\lambda$, for all $\tau>0$. 
    
    Writing $\bar\eta=-\Phi^\star\bar h/\lambda$, so that $\bar h\in K_\lambda$ iff $\|\bar\eta\|_\infty\le1$, the cone $N_{K_\lambda}(\bar h)$ is carried by the \emph{saturation set} $\{x\in\calX:|\bar\eta(x)|=1\}$: any measure $\bar\mu\in\calM(\calX)$ supported there with $\operatorname{sign}\bar\mu=\operatorname{sign}\bar\eta$ (equivalently $\bar\eta\in\partial R(\bar\mu)$, a dual certificate for $\bar\mu$) gives $-\Phi\bar\mu\in N_{K_\lambda}(\bar h)$, and 
    \[
        N_{K_\lambda}(\bar h)=-\Phi\,N_{\{\|\cdot\|_\infty\le1\}}(\bar\eta)\,,
    \]
    with equality; the inclusion $\supseteq$ is automatic, and for a norm regularizer the reverse inclusion also holds, for every bounded $\Phi$, since $\dom R^\star$ is the unit ball of the pre-dual and $0\in\operatorname{int}\dom R^\star\cap\ran\Phi^\star$ supplies the qualification of Proposition~\ref{prop:BR_primal_lift}, whether or not $\Phi$ is compact.
    
    Consequently, for any algorithmic pair $(\mu,h)$ with $\lambda L_R(\mu\parallel\eta)\le\epsilon$, the proxy of Theorem~\ref{thm:algorithmic_stability} is an \emph{exactly dual-feasible} point $\bar h\in K_\lambda$ with $\bar v\in N_{K_\lambda}(\bar h)$, $\|\bar h-h\|_\calH\le\sqrt{\tau\epsilon}$ and $\|\bar v-(-\Phi\mu)\|_\calH\le\sqrt{\epsilon/\tau}$; its certificate $\bar\eta=-\Phi^\star\bar h/\lambda$ satisfies $\|\bar\eta\|_\infty\le1$, an admissible TV certificate. The Fenchel--Young tolerance $\epsilon$ thus certifies a fully feasible TV certificate within Hilbert distance $\sqrt{\tau\epsilon}$ of the iterate.
    \item \textbf{Squared norm.} For $\calB$ a Hilbert space and $R(\mu)=\tfrac{1}{2\lambda}\|\mu\|_{\calB^\star}^2$, we have $u(h)=\tfrac{1}{2}\|\Phi^\star h\|_{\calB}^2$ and $\operatorname{prox}_{\tau u}(h)=(\Id+\tau\,\Phi\Phi^\star)^{-1}h$, the Tikhonov resolvent in $\calH$.
    \item \textbf{Nuclear norm.} For $R(W)=\|W\|_*$, $u$ is the indicator of the set $\{h:\|\Phi^\star h\|_{\mathrm{op}}\le\lambda\}$ and $\operatorname{prox}_{\tau u}$ is the orthogonal projection onto this set. For $\Phi=\Id$ this projection clips the singular values at level~$\lambda$, $\sigma_i\mapsto\min(\sigma_i,\lambda)$; it is \emph{not} the singular-value soft-thresholding, which is the proximal operator of the nuclear norm itself.
\end{enumerate}
\end{remark}

\begin{remark}[Cheap implementable proxy: dual rescaling]\label{rem:dual_rescaling}
In the TV case the orthogonal projection~$\mathcal P_{K_\lambda}$ is a continuous-domain quadratic program with infinitely many linear constraints. For numerical applications, it can be replaced by the scalar contraction
\[
    \tilde h\coloneqq c\,h,\qquad c\coloneqq\min\!\Big(1,\tfrac{\lambda}{\|\Phi^\star h\|_\infty}\Big)\in[0,1],
\]
a one-step retraction onto $K_\lambda$. Even if it is not obtained by projection, the certificate $\tilde\eta=-\Phi^\star\tilde h/\lambda=c\,\eta$
stays in $\ran\Phi^\star$, and the duality-gap identity (Theorem~\ref{thm:dual_gap}) at~$(\mu,\tilde h)$
yields the computable decomposition
\[
\Delta(\mu,\tilde h)=L_F(\Phi\mu\parallel\tilde h)+\lambda L_R(\mu\parallel\tilde\eta)\ge\objmeas(\mu)-\inf\objmeas\,,
\]
both pieces finite and oracle-free (see Section~\ref{subsec:exp_certificate} for further details).
\end{remark}

\subsection{Mirror proximal refinement adapted to the fidelity}
\label{subsec:mirror_refinement}

A limitation of \Cref{thm:algorithmic_stability} is that it is based on a $\operatorname{prox}$ step, hence gives a particular role to the squared Hilbert distance on $\calH$ whereas in Section~\ref{sec:controlled_gaps} we only dealt with Bregman divergences, comparing images $\Phi\mu$ through the lens of $F$. Also \Cref{thm:algorithmic_stability} is meant to obtain a point satisfying the source condition, which leaves aside the notion of mirror alignment that we introduced and discussed in \Cref{thm:robust_align_first_order}. We now explore a Bregman variant of the previous approach, recovering the previous computations in the particular case of quadratic $F$.

\medskip

Throughout, $F$ is differentiable and strictly convex on $\calH$ and $F^\star$ is differentiable on $\operatorname{int}(\dom(F^\star))$, as in Section~\ref{sec:controlled_gaps}. For $\tau>0$ we define the \emph{tilted Bregman proximal step}
\begin{subequations}
\label{eq:mirror_prox_sub}
\begin{equation}\label{eq:mirror_prox_step}
    \bar h\;\coloneqq\;\argmin_{h'\in\calH}\Big\{u(h')+\langle\Phi\mu,h'\rangle_\calH+\frac{1}{\tau}\,F^\star(h'\mid h)\Big\},
\end{equation}
and we assume that the minimum is attained at some $\bar h$ with $h,\bar h\in\operatorname{int}(\dom(F^\star))$. The first-order optimality condition of~\eqref{eq:mirror_prox_step} reads
\begin{equation}\label{eq:mirror_prox_subgrad}
    \bar v\;\coloneqq\;-\Phi\mu+\frac{1}{\tau}\big(\nabla F^\star(h)-\nabla F^\star(\bar h)\big)\;\in\;\partial u(\bar h).
\end{equation}
\end{subequations}
In the quadratic case $F=\frac12\|\cdot-y\|_\calH^2$ one has $F^\star(h'\mid h)=\frac12\|h'-h\|_\calH^2$, since the data $y$ only enters the linear part of $F^\star$, so~\eqref{eq:mirror_prox_step} is $\operatorname{prox}_{\tau u}(h-\tau\Phi\mu)$ and~\eqref{eq:mirror_prox_subgrad} is~\eqref{eq:BR_proximal_formula}: the step reduces to the construction of Theorem~\ref{thm:algorithmic_stability}. The tilt $\langle\Phi\mu,\cdot\rangle_\calH$ uses the prediction $\Phi\mu$ as a surrogate of the dual gradient $\nabla F^\star(h)$, and by~\eqref{eq:FY_equals_Bregman} the surrogate error is the fidelity loss itself, $F(\Phi\mu\mid\nabla F^\star(h))=L_F(\Phi\mu\parallel h)$. The generalization of Carlier's proximal construction is a mirror proximal step with kernel $F^\star$, tilted by the prediction.

\medskip

The first result extends the Br\o ndsted--Rockafellar bounds of Theorem~\ref{thm:algorithmic_stability}. It does not use any lift to a primal certificate.

\begin{prop}[Bregman Br\o ndsted--Rockafellar]
    \label{prop:mirror_BR_bound}
Let \Cref{ass:BR_regularizer} hold and let $(\mu,h)\in\dom R\times\dom u$ satisfy $\lambda\,L_R(\mu\parallel\eta)\le\epsilon$ with $\eta=-\Phi^\star h/\lambda$. For $\tau>0$, let $\bar h$ and $\bar v\in\partial u(\bar h)$ be given by~\eqref{eq:mirror_prox_sub}. Then $\bar h\in\dom u$, so the proxy is dual-feasible, $\bar\eta\coloneqq-\Phi^\star\bar h/\lambda\in\dom R^\star$, and
\begin{equation}\label{eq:mirror_BR_pairing}
    \frac{1}{\tau}\Big[F^\star(\bar h\mid h)+F^\star(h\mid\bar h)\Big]
    \;=\;\big\langle\bar v-(-\Phi\mu),\,h-\bar h\big\rangle_\calH\;\le\;\epsilon.
\end{equation}
If moreover $\nabla F$ is $M$-Lipschitz then 
\[
\|\bar h-h\|_\calH\le\sqrt{M\tau\epsilon}\,,
\]
and if $F$ is $\gamma$-strongly convex then 
\[
\|\bar v-(-\Phi\mu)\|_\calH\le\sqrt{\epsilon/(\gamma\tau)}\,.
\]
The quadratic case, where $M=\gamma=1$, recovers~\eqref{eq:BR_proxy_subequations}.
\end{prop}

\begin{proof}
By Proposition~\ref{prop:BR_gap_inequality}, $L_u(h\parallel-\Phi\mu)\le\lambda L_R(\mu\parallel\eta)\le\epsilon$. The Fenchel--Young inequality at $\bar h$ gives $u^\star(-\Phi\mu)\ge\langle-\Phi\mu,\bar h\rangle_\calH-u(\bar h)$, hence
\[
\epsilon\;\ge\;u(h)+u^\star(-\Phi\mu)+\langle\Phi\mu,h\rangle_\calH\;\ge\;u(h)-u(\bar h)+\langle\Phi\mu,h-\bar h\rangle_\calH\,.
\]
The subgradient inequality $u(h)\ge u(\bar h)+\langle\bar v,h-\bar h\rangle_\calH$ then yields $\epsilon\ge\langle\bar v+\Phi\mu,h-\bar h\rangle_\calH$, and by~\eqref{eq:mirror_prox_subgrad} together with the definition of the Bregman divergence,
\[
\langle\bar v+\Phi\mu,h-\bar h\rangle_\calH
=\frac{1}{\tau}\big\langle\nabla F^\star(h)-\nabla F^\star(\bar h),\,h-\bar h\big\rangle_\calH
=\frac{1}{\tau}\Big[F^\star(\bar h\mid h)+F^\star(h\mid\bar h)\Big].
\]
If $\nabla F$ is $M$-Lipschitz then $F^\star$ is $(1/M)$-strongly convex, so the bracket is at least $\frac{1}{M}\|h-\bar h\|_\calH^2$. If $F$ is $\gamma$-strongly convex then $\nabla F^\star$ is $(1/\gamma)$-Lipschitz, so by co-coercivity $\langle\nabla F^\star(h)-\nabla F^\star(\bar h),h-\bar h\rangle_\calH\ge\gamma\|\nabla F^\star(h)-\nabla F^\star(\bar h)\|_\calH^2=\gamma\tau^2\|\bar v+\Phi\mu\|_\calH^2$.
\end{proof}

The second result presents an exact identity for the duality gap.

\begin{prop}[Exact gap identity for the tilted step, $\tau=1$]\label{prop:mirror_gap_identity}
In the setting of Proposition~\ref{prop:mirror_BR_bound} with $\tau=1$, suppose in addition that $\bar v$ lifts to a primal certificate: 
\begin{subequations}
\begin{equation}\label{eq:mirror_residual_requirement}
\text{there is $\bar\mu\in\calB^\star$ with $\bar v=-\Phi\bar\mu$ and $\bar\eta=-\frac{\Phi^\star\bar h}{\lambda}\in\partial R(\bar\mu)$.}
\end{equation}
Write $g\coloneqq\nabla F^\star(h)$, $\bar g\coloneqq\nabla F^\star(\bar h)$ and the \emph{mirror residual} $r\coloneqq\Phi\mu-g$. Then the step preserves the residual,
\begin{equation}\label{eq:mirror_residual_preserved}
    \Phi\bar\mu-\bar g\;=\;\Phi\mu-g\;=\;r\,,
\end{equation}
and the duality gap satisfies
\begin{equation}\label{eq:mirror_gap_identity}
    \Delta(\mu,h)-\Delta(\bar\mu,\bar h)\;=\;\lambda\,L_R(\mu\parallel\eta)\;+\;\big[F(g+r\mid g)-F(\bar g+r\mid\bar g)\big]\,.
\end{equation}
\end{subequations}
\end{prop}
\noindent
In the quadratic case $F(p+r\mid p)=\frac12\|r\|_\calH^2$ does not depend on the base point $p$, the bracket vanishes, and~\eqref{eq:mirror_gap_identity} gives $\Delta(\bar\mu,\bar h)=\Delta(\mu,h)-\lambda L_R(\mu\parallel\eta)$ for the quadratic case. Note also that by Proposition~\ref{prop:BR_primal_lift}, the requirement \eqref{eq:mirror_residual_requirement} holds under the stated constraint qualification, and unconditionally for a norm regularizer. 

\begin{proof}
From~\eqref{eq:mirror_prox_subgrad} with $\tau=1$ and $\bar v=-\Phi\bar\mu$ we get $\Phi\bar\mu=\Phi\mu-g+\bar g$, which is~\eqref{eq:mirror_residual_preserved}. Since $\bar\eta\in\partial R(\bar\mu)$ we have $L_R(\bar\mu\parallel\bar\eta)=0$. By the Bregman identity~\eqref{eq:FY_equals_Bregman}, $L_F(\Phi\mu\parallel h)=F(\Phi\mu\mid\nabla F^\star(h))=F(g+r\mid g)$ and, using~\eqref{eq:mirror_residual_preserved}, $L_F(\Phi\bar\mu\parallel\bar h)=F(\bar g+r\mid\bar g)$. Subtracting the two decompositions of the gap,
\[
\Delta(\mu,h)-\Delta(\bar\mu,\bar h)=L_F(\Phi\mu\parallel h)-L_F(\Phi\bar\mu\parallel\bar h)+\lambda\,L_R(\mu\parallel\eta)\,,
\]
gives~\eqref{eq:mirror_gap_identity}.
\end{proof}

The drift term in~\eqref{eq:mirror_gap_identity} is second order: it is controlled by the squared residual times the move of the mirror point, and the move is controlled by Proposition~\ref{prop:mirror_BR_bound}.

\begin{coro}[Gap bound with a second-order remainder]\label{coro:mirror_gap_remainder}
In the setting of Proposition~\ref{prop:mirror_gap_identity}, assume moreover that $F$ is $\gamma$-strongly convex and twice G\^ateaux differentiable with 
\[
\text{$\|\nabla^2F(p)-\nabla^2F(q)\|_{\operatorname{op}}\le L_2\|p-q\|_\calH$ for all $p,q\in\calH$.}
\]
Then
\begin{equation}\label{eq:mirror_gap_remainder}
    \Delta(\bar\mu,\bar h)\;\le\;\Delta(\mu,h)\;-\;\lambda\,L_R(\mu\parallel\eta)\;+\;\frac{L_2}{\gamma^{3/2}}\,L_F(\Phi\mu\parallel h)\,\sqrt{\lambda\,L_R(\mu\parallel\eta)}\,.
\end{equation}
\end{coro}
\noindent
No tolerance parameter appears: the bound uses Proposition~\ref{prop:mirror_BR_bound} with the choice $\epsilon=\lambda L_R(\mu\parallel\eta)$, which is admissible since $\epsilon$ need not be small, and every term on the right-hand side is computable from the pair~$(\mu,h)$. The remainder vanishes when $F$ is quadratic, where $L_2=0$: the proxy then improves the gap by exactly $\lambda L_R(\mu\parallel\eta)$.

\begin{proof}
Set $\psi(p)\coloneqq F(p+r\mid p)=F(p+r)-F(p)-\langle\nabla F(p),r\rangle_\calH$. For every $p\in\calH$ and every direction $w\in\calH$, using $\nabla F(p+r)-\nabla F(p)=\int_0^1\nabla^2F(p+sr)\,r\,\mathrm{d}s$ and the self-adjointness of $\nabla^2F(p)$,
\[
\mathrm{d}\psi(p)[w]=\Big\langle\int_0^1\big[\nabla^2F(p+sr)-\nabla^2F(p)\big]\,r\,\mathrm{d}s,\,w\Big\rangle_\calH\,,
\]
so $\|\mathrm{d}\psi(p)\|\le\int_0^1 L_2\,s\,\|r\|_\calH^2\,\mathrm{d}s=\frac{L_2}{2}\|r\|_\calH^2$ and, by the mean value inequality, $|\psi(g)-\psi(\bar g)|\le\frac{L_2}{2}\|r\|_\calH^2\,\|g-\bar g\|_\calH$. By $\gamma$-strong convexity, $L_F(\Phi\mu\parallel h)=F(g+r\mid g)\ge\frac{\gamma}{2}\|r\|_\calH^2$, and Proposition~\ref{prop:mirror_BR_bound} with $\tau=1$ and $\epsilon=\lambda L_R(\mu\parallel\eta)$ gives $\|g-\bar g\|_\calH=\|\bar v+\Phi\mu\|_\calH\le\sqrt{\lambda L_R(\mu\parallel\eta)/\gamma}$. Hence the bracket in~\eqref{eq:mirror_gap_identity} is at least $-\frac{L_2}{2}\cdot\frac{2}{\gamma}L_F(\Phi\mu\parallel h)\cdot\sqrt{\lambda L_R(\mu\parallel\eta)/\gamma}$, which is~\eqref{eq:mirror_gap_remainder}.
\end{proof}

\begin{coro}[When the tilted step improves the gap]\label{coro:mirror_descent_criterion}
Write $\delta\coloneqq L_F(\Phi\bar\mu\parallel\bar h)-L_F(\Phi\mu\parallel h)$ for the change of fidelity loss and $\kappa\coloneqq\tfrac1\tau\big[F^\star(\bar h\mid h)+F^\star(h\mid\bar h)\big]\ge0$. Since $L_R(\bar\mu\parallel\bar\eta)=0$, the gap decomposition gives, for every $\tau>0$,
\begin{equation}\label{eq:mirror_exact_margin}
    \Delta(\mu,h)-\Delta(\bar\mu,\bar h)=\lambda\,L_R(\mu\parallel\eta)-\delta\,,
\end{equation}
so the tilted step improves the gap if and only if $\delta\le\lambda L_R(\mu\parallel\eta)$. Two sufficient regimes make this explicit.
\begin{enumerate}[label=\textnormal{(\alph*)}]
\item \emph{Step-size balance.} By Proposition~\ref{prop:mirror_BR_bound} with $\epsilon=\lambda L_R(\mu\parallel\eta)$ one has $\kappa\le\lambda L_R(\mu\parallel\eta)$ unconditionally, so $\delta\le\kappa$ suffices, with margin $\lambda L_R(\mu\parallel\eta)-\delta\ge\kappa-\delta\ge0$.
\item \emph{Near-quadratic, $\tau=1$.} Under the hypotheses of Corollary~\ref{coro:mirror_gap_remainder}, improvement holds whenever $\tfrac{L_2}{\gamma^{3/2}}\,L_F(\Phi\mu\parallel h)\le\sqrt{\lambda L_R(\mu\parallel\eta)}$, and always for quadratic $F$ $(L_2=0)$, where the improvement equals $\lambda L_R(\mu\parallel\eta)$.
\end{enumerate}
\end{coro}

\begin{remark}[Scope of the tilted step]\label{rem:mirror_step_scope}
{\ }
\begin{enumerate}[label=(\roman*)]
\item \textbf{Free parameter $\tau$.} For $\tau\neq1$ the residual is not preserved: $\Phi\bar\mu-\bar g=r+(1-\tfrac{1}{\tau})(g-\bar g)$, and~\eqref{eq:mirror_gap_identity} picks up an extra first order term. Proposition~\ref{prop:mirror_BR_bound} keeps the free parameter $\tau$, as in Theorem~\ref{thm:algorithmic_stability}.
\item \textbf{Local constants.} For the losses of Section~\ref{sec:applications}, the constants $\gamma$ and $L_2$ are finite on sublevel sets only in general; in the Kullback--Leibler and entropic cases $L_2$ blows up at the boundary of $\dom F$, so~\eqref{eq:mirror_gap_remainder} should be applied on a sublevel set containing the iterates.
\end{enumerate}
\end{remark}

\section{Applications: specific fidelities and regularizers}\label{sec:applications}

In this section we apply the F+R framework to a catalogue of popular data-fidelity terms and regularizers, deriving for each loss its Fenchel--Young geometry (primal $F$, dual $F^\star$, and discrepancy $L_F$) and for each penalty the triple $(R,R^\star,L_R)$, so that the duality gap
$\Delta(\mu,h)=L_F(\Phi\mu\parallel h)+\lambda L_R(\mu\parallel\eta)$ of Theorem~\ref{thm:dual_gap} becomes explicit on each problem class. Three of our results hold for \emph{every} data-fidelity term below: the early-stopping certification of Theorem~\ref{thm:algorithmic_stability}, whose hypotheses (Assumption~\ref{ass:BR_regularizer}) constrain only the regularizer~$R$; this gap decomposition, in the general form of Theorem~\ref{thm:dual_gap} since several of these losses leave $\calF_0$; and the Fenchel--Young error bounds of Theorem~\ref{thm:FY_discrepancy_error_bounds}, which need only a proper, l.s.c.\ convex $F$, a source-condition certificate, and $\objmeas(\mu^\star)<+\infty$ (that is $\Phi\mu^\star\in\dom F$, automatic when $\dom F=\calH$). 

The strictly-convex $\calC^1$ class $\calF_0$ is needed only for the Mirror-Alignment characterization, the uniqueness of the aligned certificate, and, through the weak-coercivity bound~\eqref{eq:hyp_coercivite_faible}, the existence of a minimizer. Least squares, Poisson with positive counts $y>0$, and the data-coupled iOT primal~\eqref{eq:iUOT_primal_actual} with $\pi^{\mathrm{obs}}>0$ lie in $\calF_0$; the logistic, Huber, squared-hinge, and Kullback--Leibler do not, and we record each loss's $\calF_0$ status in the loss-by-loss computations of Appendix~\ref{sec:datafidelity_FYL_app}.  
The exact-recovery analysis built on these geometries is developed in the companion paper~\citep{aubin2026localization}.

\subsection{Specific data-fidelity terms}
\label{subsec:data-fidelity}
\subsubsection{Common losses and their Fenchel conjugates}

Each common data-fidelity term is separable, $F(h)=\sum_{i=1}^n f(h_i)$, and likewise $F^\star(w)=\sum_i f^\star(w_i)$ and $L_F(h\parallel w)=\sum_i \ell_F(h_i,w_i)$, with $y_i$ the paired observation. Table~\ref{tab:common_losses} collects the scalar per-coordinate summands for six continuously differentiable ($\calC^1$) losses (least squares, logistic, Poisson, Kullback--Leibler, Huber, squared-hinge); the full derivations and the scope condition of each loss are in Appendix~\ref{sec:datafidelity_FYL_app}.

\begin{table}[htbp]
\centering
\small
\setlength{\tabcolsep}{5pt}
\renewcommand{\arraystretch}{1.5}
\begin{tabular}{@{}lllll@{}}
\toprule
Loss & $f(h)$ & $f^\star(w)$ & $\ell_F(h,w)$ & $\calF_0$\\
\midrule
Least squares & $\tfrac12(h-y)^2$ & $\tfrac12 w^2+wy$ & $\tfrac12(h-w-y)^2$ & Yes\\
Logistic & $\log(1+e^{-yh})$ & \shortstack[c]{$(-yw)\log(-yw)$\\ $+\,(1+yw)\log(1+yw)$} & $f(h)+f^\star(w)-wh$ & No\\
Poisson & $e^{h}-yh$ & $(w+y)\log(w+y)-(w+y)$ & $\KL(w+y\parallel e^{h})$ & Yes $(y>0)$\\
Kullback--Leibler & $h\log\tfrac{h}{y}-h+y$ & $y(e^{w}-1)$ & $\KL(h\parallel y e^{w})$ & No\\
Huber & $\rho_\delta(h-y)$ & $\tfrac12 w^2+wy$ & $f(h)+f^\star(w)-wh$ & No\\
Squared hinge & $\tfrac12\max(0,1-yh)^2$ & $\tfrac12 w^2+yw$ & $f(h)+f^\star(w)-wh$ & No\\
\bottomrule
\end{tabular}
\caption{Common continuously differentiable ($\calC^1$) data-fidelity losses and their Fenchel conjugates, as scalar per-coordinate summands ($F(h)=\sum_i f(h_i)$, and likewise for $F^\star$ and $L_F$, with $y_i$ the paired observation). In every row $\ell_F(h,w)=f(h)+f^\star(w)-hw$; the closed forms shown for least squares, Poisson and Kullback--Leibler follow by simplification. The conjugate $f^\star$ is finite only on $-yw\in[0,1]$ (logistic), $w+y\ge0$ (Poisson), $|w|\le\delta$ (Huber) and $-yw\ge0$ (squared hinge), and is $+\infty$ otherwise; $\rho_\delta(r)=\tfrac12 r^2$ for $|r|\le\delta$ and $\delta(|r|-\tfrac12\delta)$ otherwise; Poisson requires $y_i\ge0$ and Kullback--Leibler requires $y_i>0$, and the Kullback--Leibler primal $f$ (hence its Fenchel--Young loss $\ell_F$) is defined only on $h\ge0$, with the convention $0\log0=0$; the logistic and squared-hinge rows assume binary labels $y_i\in\{-1,+1\}$, under which the tabulated conjugates hold. 
The last column records membership in the class $\calF_0$ of Definition~\ref{def_assumption:fidelity} (finite-valued on $\calH$, $\calC^1$, strictly convex, and weakly coercive~\eqref{eq:hyp_coercivite_faible}). 
}
\label{tab:common_losses}
\end{table}

\begin{subequations}
\label{eq:iUOT}
\subsubsection{Inverse optimal transport (iOT)}

In the iOT setting, following \citet[Section 2]{andrade2024sparsistency}, the goal is to learn a cost matrix $C \in \mathbb{R}^{n \times n}$ from an observed (unbalanced) transport plan $\pi^{\mathrm{obs}} \in \mathbb{R}_+^{n \times n}$.

\medskip

\paragraph{Inverse optimal transport.}
Given $\bbX$ and $\bbY$ two compact subsets of $\R^d$, a continuous cost $c:\bbX\times \bbY \to \R$, probability measures $\alpha\in \calP(\bbX)$, $\beta \in \calP(\bbY)$, and the set $\Pi(\alpha,\beta)\subset \calP(\bbX\times \bbY)$ of couplings with marginals $\alpha$ and $\beta$, the problem of entropic optimal transport is to find $\pi^* \in \argmin_{\pi \in \Pi(\alpha,\beta)} \bracket{c,\pi}+\varepsilon \KL(\pi | \alpha \otimes \beta)$. Denoting by $\Omega(\pi)\coloneqq \varepsilon \KL(\pi | \alpha \otimes \beta) + \iota_{\{\pi \in \Pi(\alpha,\beta)\}}(\pi)$, first-order conditions give that $-c \in \partial \Omega(\pi^*)$. This suggests to introduce the corresponding Fenchel-Young loss. However in iOT, instead of looking for $\pi^*$ given a $c$, we look for $c$ given a $\pi^{\mathrm{obs}}$. This leads to considering 
\begin{equation}\label{eq:FYL_iOT}
    F_{\mathrm{iOT}}(c)\coloneqq L_{\Omega}(-c \parallel \pi^{\mathrm{obs}})\;=\; \Omega^*(-c) \;+\;  \bracket{c,\,\pi^{\mathrm{obs}}} \;+\; \Omega(\pi^{\mathrm{obs}}).
\end{equation}
Known duality computations on $\Omega^*(-c)$ give that
\begin{equation}
\begin{aligned}
    \Omega^*(-c)&=\sup_{\pi \in \Pi(\alpha,\beta)} \bracket{-c,\pi}-\varepsilon \KL(\pi | \alpha \otimes \beta)\\
    &=-\sup_{f\in \calC(\bbX),g\in \calC(\bbY)} \Big\{\bracket{f,\alpha}+\bracket{g,\beta}- \varepsilon \int \Big(\exp{\Big( \frac{f+g-c}{\varepsilon}\Big)}-1\Big) \mathrm{d}\alpha(x) \mathrm{d}\beta(y)\Big\}
\end{aligned}
\end{equation}
Note that $\pi^* \in \argmin 
\{\bracket{c,\pi}+\varepsilon \KL(\pi | \alpha \otimes \beta)\,:\,\pi \in \Pi(\alpha,\beta)\}$ is invariant, due to the marginals constraints, when replacing $c$ by $\tilde c=c-f-g$ for any $f,g$. Hence we can arbitrarily consider a reduced problem defined on $\tilde c$, that is equivalent to \eqref{eq:FYL_iOT} when dropping the constant term~$\Omega(\pi^{\mathrm{obs}})$,
\begin{equation}
    \label{eq:iUOT_primal_actual_continuous}
    F_{\mathrm{red.iOT}}(\tilde c) \;\coloneqq \; \bracket{\tilde c,\,\pi^{\mathrm{obs}}}+ \varepsilon \int \exp{\Big( \frac{-\tilde c}{\varepsilon}\Big)} \mathrm{d}\alpha(x) \mathrm{d}\beta(y)
\end{equation}
Instead of $\alpha$ and $\beta$, one has in practice samples $\hat \alpha=\frac{1}{n}\sum_{i=1}^n \delta_{x_i}$ and $\hat\beta=\frac{1}{m}\sum_{j=1}^m \delta_{y_j}$ deduced from the marginals of the discrete $\pi^{\mathrm{obs}}$. Unbalanced settings would appear from handling the mismatch between $(\alpha,\beta)$ and $(\hat \alpha,\hat \beta)$. For simplicity we assume the latter coincide and can then write the problem in matrix form $C_{i,j}=\tilde c(x_i,y_j)$ and with functional
\begin{equation}
    \label{eq:iUOT_primal_actual}
    F(C) \;\coloneqq \; \frac{\varepsilon}{m n}\sum_{i,j}\exp\!\left(-\frac{C_{i,j}}{\varepsilon}\right) \;+\; \bigl\langle C,\,\pi^{\mathrm{obs}}\bigr\rangle.
\end{equation}
where the key term is given by
\begin{equation}
    \label{eq:iUOT_generator}
    F_{\mathrm{gen}}(C) \;=\; \varepsilon \sum_{i,j} \exp\!\left(-\frac{C_{i,j}}{\varepsilon}\right),
\end{equation}
The gradient $\nabla F(C) = -\pi_\varepsilon(C)+\pi^{\mathrm{obs}}$, for $[\pi_\varepsilon(C)]_{i,j}\coloneqq \frac{1}{mn}e^{-C_{i,j}/\varepsilon}\ge 0$, vanishes precisely at $C=C^{\mathrm{obs}}\coloneqq -\varepsilon\log\pi^{\mathrm{obs}}$ when $\pi^{\mathrm{obs}}>0$ entrywise. 
The dual computation, Fenchel--Young loss, and local sharpening below are written for $F_{\mathrm{gen}}$ for concision; the corresponding objects for $F$ differ by a known affine shift in the dual variable ($p\mapsto p-\pi^{\mathrm{obs}}$) which does not affect Hessians, Fenchel--Young loss values along the regularization path, or the certificate analysis. Similarly we incorporate the term $mn$ into the definition of $C$.

\paragraph{Dual data-fidelity and gradient.}
    The Fenchel conjugate $F_{\mathrm{gen}}^\star: \mathbb{R}^{n \times n} \to \mathbb{R} \cup \{+\infty\}$ is the negative generalized Boltzmann--Shannon entropy:
\begin{equation}
        F_{\mathrm{gen}}^\star(\pi) = 
    \begin{cases} 
    \varepsilon \sum_{i,j} (-\pi_{i,j}) \left( \log(-\pi_{i,j}) - 1 \right) & \text{if } \pi \le 0, \\
    +\infty & \text{otherwise.}
    \end{cases}
\end{equation}
The gradient of the dual maps a transport plan back to the cost that generated it:
\begin{equation}
        [\nabla F_{\mathrm{gen}}^\star(\pi)]_{i,j} = -\varepsilon \log(-\pi_{i,j}).
\end{equation}
This confirms the inverse relationship: if $\pi = -\exp(-C/\varepsilon)$, then $C = -\varepsilon \log(-\pi)$.

\paragraph{Fenchel--Young loss.}
The Fenchel--Young loss $L_{F}(C \parallel -\pi^{\mathrm{obs}})$ corresponds to the generalized Kullback--Leibler divergence:
\begin{equation}
    L_{F_{\mathrm{gen}}}(C \parallel -\pi^{\mathrm{obs}}) = \varepsilon \KL(\pi^{\mathrm{obs}} \parallel \pi_\varepsilon(C)) = \varepsilon \sum_{i,j} \left( \pi^{\mathrm{obs}}_{i,j} \log\left(\frac{\pi^{\mathrm{obs}}_{i,j}}{e^{-C_{i,j}/\varepsilon}}\right) - \pi^{\mathrm{obs}}_{i,j} + e^{-C_{i,j}/\varepsilon} \right).
\end{equation}
\end{subequations}

\subsection{Popular regularizers and their Fenchel--Young geometry}
\label{subsec:regularizers}

We now consider the regularizer side $R$ and record, for a catalogue of standard sparsity-promoting penalties, the regularizer $R$, its Fenchel conjugate $R^\star$, and the associated Fenchel--Young loss $L_R$. We order them so that the canonical instance, the total-variation norm of signed measures (the GBL regularizer), comes first, followed by its vector-valued extension to group/joint TV, then \emph{analysis sparsity} (the regularizer acting on a linear transform $\mu=Lu$, subsuming TV-of-gradient as $L=\nabla$) and the \emph{nuclear norm} on low-rank operators.

\begin{subequations}
\subsubsection{Regularization by the total variation norm (sparse spikes)}
\label{subsubsec:TV_regularizer}

We consider the classical problem of recovering a sparse signed measure $\mu \in \calM(\calX)$ (a sum of Dirac masses) from measurements. The regularizer is the total variation norm (or Radon norm).

\paragraph{Primal and dual functions.}
The regularizer $R: \calM(\calX) \to \mathbb{R}$ is the total variation norm:
\begin{equation}
    R(\mu) = \|\mu\|_{\TV} \coloneqq  \sup \left\{ \int_{\calX} f(x) \,\mathrm{d}\mu(x) \;\middle|\; f \in \calC_0(\calX), \|f\|_\infty \le 1 \right\}.
\end{equation}
The pre-dual space is $\calC_0(\calX)$ equipped with the supremum norm. The Fenchel conjugate $R^\star$ is the indicator function of the unit ball of continuous functions:
\begin{equation}
    R^\star(f) = \iota_{\{\|f\|_\infty \le 1\}}(f) = 
    \begin{cases} 
    0 & \text{if } \sup_{x \in \calX} |f(x)| \le 1, \\
        +\infty & \text{otherwise}
    \end{cases}
\end{equation}

\paragraph{Fenchel--Young loss.}
The associated Fenchel--Young loss penalizes the lack of correlation between the measure and the dual certificate (or can be seen as a slack in H\"older's inequality), while enforcing the dual constraint:
\begin{equation}
    L_R(\mu \parallel f) = 
    \begin{cases} 
    \|\mu\|_{\TV} - \langle f, \mu \rangle_{\calC, \calM} & \text{if } \|f\|_\infty \le 1, \\
        +\infty & \text{otherwise}
    \end{cases}
\end{equation}
This loss is zero if and only if $f$ is a subgradient of the TV norm at $\mu$, meaning $f$ saturates the norm constraints on the support of $\mu$ with the correct sign.

\end{subequations}

\begin{subequations}
\subsubsection{Regularization by group sparsity (vector-valued measures)}
\label{subsubsec:group_sparsity}

We consider the recovery of a vector-valued measure $\mu \in \calM(\calX; \mathbb{R}^d)$ where the channels are coupled. The goal is to recover spikes that share the same support across all $d$ dimensions (joint sparsity).

\paragraph{Primal and dual functions.}
The regularizer $R$ is the total variation of the vector measure, defined using the Euclidean norm $|\cdot|_2$ on $\mathbb{R}^d$ (acting as a group-Lasso penalty in the continuum):
\begin{equation}
    R(\mu) = \|\mu\|_{\calM, 2} \coloneqq  \int_{\calX} \left| \frac{\mathrm{d}\mu}{\mathrm{d}|\mu|}(x) \right|_2 \,\mathrm{d}|\mu|(x),
\end{equation}
where $\frac{\mathrm{d}\mu}{\mathrm{d}|\mu|}$ is the Radon--Nikod\'ym derivative (the local orientation vector) and $|\mu|$ denotes the total-variation (Euclidean) measure of the vector-valued $\mu$.
The pre-dual space is $\calC(\calX; \mathbb{R}^d)$, paired with $\calM(\calX; \mathbb{R}^d)$ via $\langle f, \mu \rangle_{\calC, \calM} = \int_\calX \langle f(x), \mathrm{d}\mu(x) \rangle_{\mathbb{R}^d}$. The Fenchel conjugate $R^\star$ is the indicator of the unit ball of the dual norm (the sup-norm of the pointwise Euclidean norms):
\begin{equation}
    R^\star(f) = \iota_{K_R}(f) \quad \text{with} \quad K_R \coloneqq  \left\{ f \in \calC(\calX; \mathbb{R}^d) \;\middle|\; \sup_{x \in \calX} |f(x)|_2 \le 1 \right\}.
\end{equation}

\paragraph{Fenchel--Young loss.}
The associated Fenchel--Young loss enforces the ``block-constraint'' on the dual certificate:
\begin{equation}
    L_R(\mu \parallel f) = 
    \begin{cases} 
    \|\mu\|_{\calM, 2} - \langle f, \mu \rangle_{\calC, \calM} & \text{if } f \in K_R, \\
        +\infty & \text{otherwise}
    \end{cases}
\end{equation}
This loss is zero if and only if $f \in \partial R(\mu)$, i.e.\ $|f(x)|_2 = 1$ and $f(x)$ is aligned with $\frac{\mathrm{d}\mu}{\mathrm{d}|\mu|}(x)$ on $\operatorname{supp}\mu$.

\end{subequations}

\begin{subequations}

\subsubsection{Sparse analysis prior}
\label{subsubsec:analysis_sparsity}

We consider the recovery of a signal $u\in\calU$ (a Banach signal space, e.g.\ $L^2(\R^d)$ or a Sobolev space) by imposing sparsity on its analysis transform $\mu = Lu$, where $L:\calU\to\calM(\calX)$ is a bounded linear operator (the analysis operator). Following the analysis-sparsity representer theorem of \citet{boyer2019representer}, this is equivalent to optimizing over the measure $\mu$ with a structural constraint. The special case $L=\nabla$ (TV regularization of a gradient) admits a simple Meyer-norm dual and is treated separately in Section~\ref{subsubsec:TV_gradient}.

\paragraph{Standing assumption.} Throughout this subsection we assume $\operatorname{Im}(L)$ is \emph{weak-$*$ closed} in $\calM(\calX)$ (equivalently $\operatorname{Im}(L)=\operatorname{Ker}(L^\star)^\perp$, the annihilator of $\operatorname{Ker}(L^\star)$ in $\calM(\calX)$), so that the regularizer $R$ below is weak-$*$ l.s.c.\ and $R=R^{\star\star}$.

\paragraph{Primal and dual functions.}
The regularizer $R: \calM(\calX) \to \mathbb{R} \cup \{+\infty\}$ combines the total variation norm with the hard constraint that the measure must belong to the (weak-$*$ closed) image of the analysis operator:
\begin{equation}
    R(\mu) = \|\mu\|_{\TV} + \iota_{\operatorname{Im}(L)}(\mu) =
    \begin{cases}
    \|\mu\|_{\TV} & \text{if } \exists u \in \calU, \mu = L u, \\
        +\infty & \text{otherwise.}
    \end{cases}
\end{equation}
The dual function $R^\star$ is the indicator of a quotient-norm ball: a dual vector $f \in \calC_0(\calX)$ is feasible when its sup-norm distance to $\operatorname{Ker}(L^\star)$, the kernel of the pre-adjoint operator $L^\star:\calC(\calX)\to\calU^\star$, is at most one,
\begin{equation}
    R^\star(f) = \iota_{\{\operatorname{dist}_\infty(\cdot,\operatorname{Ker}(L^\star)) \le 1\}}(f),
    \qquad
    \operatorname{dist}_\infty\big(f,\operatorname{Ker}(L^\star)\big)\coloneqq\inf_{g \in \operatorname{Ker}(L^\star)} \|f - g\|_\infty,
\end{equation}
i.e.\ $R^\star$ is the indicator of the unit ball of the quotient norm on $\calC_0(\calX)/\operatorname{Ker}(L^\star)$; the standing assumption ($\operatorname{Im}(L)$ weak-$*$ closed) identifies $\operatorname{Im}(L)$ with the dual of this quotient space, which gives the formula. The exact infimal convolution $\inf_{g\in\operatorname{Ker}(L^\star)}\iota_{\{\|\cdot\|_\infty\le1\}}(f-g)$ takes the same value when the distance is attained; in general the conjugate is its lower semi-continuous closure, because the distance to a closed subspace of $\calC(\calX)$ need not be attained.

\paragraph{Fenchel--Young loss.}
The Fenchel--Young loss $L_R$ penalizes the lack of an ``analysis sparsity'' certificate. For a measure $\mu \in \operatorname{Im}(L)$ and a dual certificate $f$:
\begin{equation}
    L_R(\mu \parallel f) =
    \begin{cases}
    \|\mu\|_{\TV} - \langle f, \mu \rangle_{\calC, \calM} & \text{if } \mu \in \operatorname{Im}(L) \text{ and } \operatorname{dist}_\infty(f,\operatorname{Ker}(L^\star)) \le 1, \\
        +\infty & \text{otherwise.}
    \end{cases}
\end{equation}
The loss is finite only when $\operatorname{dist}_\infty(f,\operatorname{Ker}(L^\star))\le 1$, i.e.\ when $f$ is close to a function bounded by one, up to a component $g \in \operatorname{Ker}(L^\star)$ invisible to the signal space.

\end{subequations}

\begin{subequations}
\label{eq:regularizers_TV_gradient}
\subsubsection{Regularization by the total variation norm of the gradient}
\label{subsubsec:TV_gradient}

Throughout this subsection, $\calX\subset\R^d$ is a \emph{bounded Lipschitz domain}. It treats the special case $L=\nabla$ of analysis sparsity (Section~\ref{subsubsec:analysis_sparsity}): the regularizer is the TV norm of a gradient field $\mu=\nabla u$ of a function of bounded variation. We keep it as a separate subsubsection because the Meyer-norm structure of the dual needs its own derivation, as it is specific to the divergence operator $L^\star=-\operatorname{div}$ (the sign is immaterial here, as the dual constraint set $K_G$ below is symmetric under $f\mapsto-f$).

\paragraph{Primal and dual functions.}
Let $\calB^\star = \calM(\calX; \mathbb{R}^d)$ be the space of vector-valued Radon measures. We consider the regularization of the gradient of a function $u$ of bounded variation, writing 
\[
BV_0(\calX)\coloneqq\big\{u\in BV(\R^d):u=0\text{ a.e.\ on }\R^d\setminus\calX\big\}
\]
for the $BV$ functions vanishing outside $\calX$. The regularizer $R: \calM(\calX; \mathbb{R}^d) \to \mathbb{R} \cup \{+\infty\}$ is defined as the total variation norm restricted to gradient fields:
\begin{equation}
    R(\mu) = 
    \begin{cases} 
    \|\mu\|_{\TV} & \text{if } \mu = \nabla u \text{ for some } u \in BV_0(\calX), \\
        +\infty & \text{otherwise}
    \end{cases}
\end{equation}
where the equality $\mu = \nabla u$ holds in the distributional sense. The constraint set $\{\nabla u:u\in BV_0(\calX)\}$ is \emph{weak-$*$ closed} in $\calM(\calX;\R^d)$: if $\nabla u_n\rightharpoonup^*\mu$ with $u_n\in BV_0(\calX)$, then, assuming $\calX$ is a bounded Lipschitz domain in $\R^d$, the Poincar\'e inequality bounds $(u_n)$ in $BV$, the compact embedding $BV\hookrightarrow L^1$ gives a limit $u\in BV_0(\calX)$ (identifying $u$ with its extension by zero to $\R^d$), and $\nabla u=\mu$. Hence $R$ is weak-$*$ l.s.c.\ and the standing assumption of Section~\ref{subsubsec:analysis_sparsity} holds here.
The pre-dual space is $\mathcal{B} = \mathcal{C}(\mathcal{X}; \mathbb{R}^d)$. Following the framework of \citet{de2023towards}, the Fenchel conjugate $R^\star$ is the indicator function of the convex set $K_G$ defined by a constraint on the divergence:
\begin{equation}
    R^\star(f) = \iota_{K_G}(f) \quad \text{with} \quad K_G \coloneqq  \left\{ f \in \calC(\calX; \mathbb{R}^d) \;\middle|\; \left| \int_E \operatorname{div}(f) \right| \le \operatorname{Per}(E),
    \quad \forall E \subset \calX \right\},
\end{equation}
where $E$ ranges over sets of finite perimeter in $\mathcal{X}$. This condition is equivalent to $\|\operatorname{div}(f)\|_G \le 1$, where $\|\cdot\|_G$ denotes the Meyer norm (the dual of the $BV$ semi-norm). The Meyer norm $\|\cdot\|_G$, introduced in \citet{meyer2001oscillating}, is defined on the space of distributions $G$ that can be represented as the divergence of a bounded vector field. Formally:
\begin{equation}
    \|v\|_G \coloneqq  \inf \left\{ \|g\|_{L^\infty(\calX; \mathbb{R}^d)} \;\middle|\; v = \operatorname{div}(g), \, g \in L^\infty(\calX; \mathbb{R}^d) \right\}.
\end{equation}
Equivalently, the identity $\|v\|_G=\sup\{\langle u,v\rangle : u\in BV_0(\calX),\ \|\nabla u\|_{\TV}\le 1\}$ exhibits $\|\cdot\|_G$ as the dual norm of the $BV$ semi-norm (consistent with the description above).

\paragraph{Fenchel--Young loss.}
The associated Fenchel--Young loss reads:
\begin{equation}
    L_R(\mu \parallel f) = 
    \begin{cases} 
    \|\mu\|_{\TV} - \langle f, \mu \rangle_{\calC, \calM} & \text{if } \mu = \nabla u \text{ for some } u \in BV_0(\calX) \text{ and } f \in K_G, \\
        +\infty & \text{otherwise}
    \end{cases}
\end{equation}

\paragraph{Stability under primal non-uniqueness.}
The regularizer $R$ is the TV norm restricted to gradient fields, hence positively $1$-homogeneous and not strictly convex, so primal uniqueness is generally not guaranteed in the Meyer-norm setting. The total duality gap nevertheless decomposes as
\begin{equation}
    \Delta(\mu, h) \;=\; \underbrace{L_{F}(\Phi\mu \parallel h)}_{\text{prediction-space control}} \;+\; \lambda\,\underbrace{L_{R}(\mu \parallel \eta)}_{\text{source-condition control}}.
\end{equation}
For any $\mu$ with $\objmeas(\mu)\le\objmeas(\mu^\star)$, the prediction bound $L_F(\Phi\mu\parallel h^\star)\le L_F(\Phi\mu^\star\parallel h^\star)$ (Theorem~\ref{thm:FY_discrepancy_error_bounds}) and Carlier's sharpened Fenchel--Young inequality \citep[Lemma~1.1]{carlier2023fenchel} applied at $(\Phi\mu,h^\star)$ imply that the squared Hilbert distance of $(\Phi\mu,\tau h^\star)$ to the graph of $\tau\partial F$ is at most $2\tau\,L_F(\Phi\mu^\star\parallel h^\star)$ (cf.\ Theorem~\ref{thm:FY_discrepancy_error_bounds}~\eqref{eq:obj_decomposition_bound_Carlier}); the prediction $\Phi\mu$ is thus tethered to the optimal one even when the primal measure is non-unique. Beyond this static control, the algorithmic stability of Section~\ref{sec:approx_source_cond} (Theorem~\ref{thm:algorithmic_stability}) supplies an exactly feasible dual proxy for any iterate with $\lambda L_R(\mu\parallel\eta)\le\epsilon$, so the Meyer-norm setting also enjoys the early-stopping guarantee.
\end{subequations}

\begin{subequations}
\subsubsection{Regularization by nuclear norm}\label{subsubsec:nuclear_norm}

We consider low-rank matrix recovery, where the optimization variable is a matrix $\mu \in \mathbb{R}^{n \times m}$ (identified here as the measure). The regularizer is the nuclear norm (or Schatten 1-norm), which promotes low-rank solutions.

\paragraph{Primal and dual functions.}
Let $R: \mathbb{R}^{n \times m} \to \mathbb{R}$ be the nuclear norm defined by the sum of the singular values:
\begin{equation}
    R(\mu) \coloneqq  \|\mu\|_* = \sum_{i=1}^{\min(n,m)} \sigma_i(\mu).
\end{equation}
The pre-dual space is $\mathbb{R}^{n \times m}$ itself, self-dual under the trace inner product $\langle f, \mu \rangle = \operatorname{tr}(f^\top \mu)$. The dual norm is the operator norm (spectral norm), denoted $\|\cdot\|_{\mathrm{op}}$ (the largest singular value). The Fenchel conjugate $R^\star$ is the indicator function of the spectral unit ball:
\begin{equation}
    R^\star(f) = \iota_{\{\|\cdot\|_{\mathrm{op}} \le 1\}}(f) = 
    \begin{cases} 
        0 & \text{if } \|f\|_{\mathrm{op}} \le 1, \\
            +\infty & \text{otherwise.}
    \end{cases}
\end{equation}

\paragraph{Fenchel--Young loss.}
The associated Fenchel--Young loss $L_R$ measures the gap between the nuclear norm and the linear pairing with the dual variable. It enforces the spectral constraint $\|f\|_{\mathrm{op}} \le 1$ as a hard (indicator) constraint on the dual certificate:
\begin{equation}
    L_R(\mu \parallel f) = 
    \begin{cases} 
        \|\mu\|_* - \langle f, \mu \rangle & \text{if } \|f\|_{\mathrm{op}} \le 1, \\
            +\infty & \text{otherwise.}
    \end{cases}
\end{equation}
This loss is zero if and only if $f$ is a subgradient of the nuclear norm at $\mu$: for a reduced SVD $\mu = U \operatorname{diag}(\sigma) V^\top$, this holds if and only if $f = U V^\top + W$ with $\|W\|_{\mathrm{op}} \le 1$, $U^\top W = 0$ and $W V = 0$.

\end{subequations}

\section{Numerical experiments}
\label{sec:experiments}
We focus on the Fenchel--Young gap $\Delta = L_F + \lambda L_R$ of Theorem~\ref{thm:dual_gap}. Throughout this section, $J$ denotes the primal objective $\objmeas$ of~\eqref{eq:convex_program_J} (and its finite-dimensional instances). 
Section~\ref{subsec:exp_certificate} runs on a synthetic 1D Gaussian super-resolution BLASSO problem and tracks $\Delta_k$ along the Conic Particle Gradient Descent (CPGD) of \citet{chizat2022sparse,decastro2023fastpart}.

This multiplicative amplitude update is a mirror-descent step for the total-variation regularizer of Section~\ref{sec:controlled_gaps}. The dual certificate $\eta_k$ converges to the source condition while the inset gap $\Delta_k$ decreases with no oracle, and the decomposition of $\Delta_k$ tracks the true optimality gap (and, as the theory guarantees, dominates it), which is the operational early-stopping signal. Section~\ref{subsec:exp_lion_muon} applies the same scalar diagnostic to the deep-learning optimizers Lion-K and Muon, certifying that in proximal form they solve the $F+\lambda R$ program, the gap $\Delta_t$ decreasing to $0$ along the trajectory. Throughout, the single scalar $\Delta$ bounds the distance to optimality without knowing the optimum.
An open-source notebook reproduces the figures below in full.\footnote{Code and companion notebook available at \url{https://github.com/ydecastro/fenchel-young-gaps}.}

We call Mirror-Alignment dual certificate (at iterate $\mu_k$) the continuous function $\eta_k$ defined by
\[
\eta_k = -\frac{\Phi^\star \nabla F(\Phi \mu_k)} \lambda.
\]
We call the \emph{rescaled} Mirror-Alignment dual the continuous function $\tilde\eta_k$, which satisfies dual feasibility $\tilde h_k\in K_\lambda$, defined by
\begin{subequations}
\begin{equation}
    \label{eq:dual_rescaling}
    \begin{aligned}
    \tilde h_k &\;\coloneqq \; c_k\,h_k\,,
    \qquad
    h_k \;=\; \nabla F(\Phi\mu_k)\,,
    \qquad
    \tilde\eta_k \;=\; -\Phi^\star\tilde h_k/\lambda \;=\; c_k\,\eta_k\,,
    \\
    c_k &\;\coloneqq \; \min\Bigl(1,\,\frac{\lambda}{\|\Phi^\star h_k\|_\infty}\Bigr) \;=\; \min\Bigl(1,\,\frac1{\|\eta_k\|_\infty}\Bigr)\in[0,1]\,;\\
    K_\lambda&\;\coloneqq\Big\{h'\in\calH:\|\Phi^\star h'\|_\infty\le\lambda\Big\}\,,
    \qquad \tilde h_k\in K_\lambda\,;
    \end{aligned}
\end{equation}
and the duality gap of the BLASSO program 
\begin{equation}
    \label{eq:exp_dual_gap}
    \Delta(\mu, h)
    \;=\;
    L_F(\Phi\mu \parallel h)
    \;+\;
    \lambda\, L_R(\mu \parallel \eta)\,,
    \qquad \eta = -\Phi^\star h / \lambda\,,
\end{equation}
\end{subequations}
which decomposes into a data-fidelity Fenchel--Young loss $L_F$ and a regularizer Fenchel--Young loss $L_R$. 

\begin{figure}[!t]
    \centering
    \includegraphics[width=\linewidth]{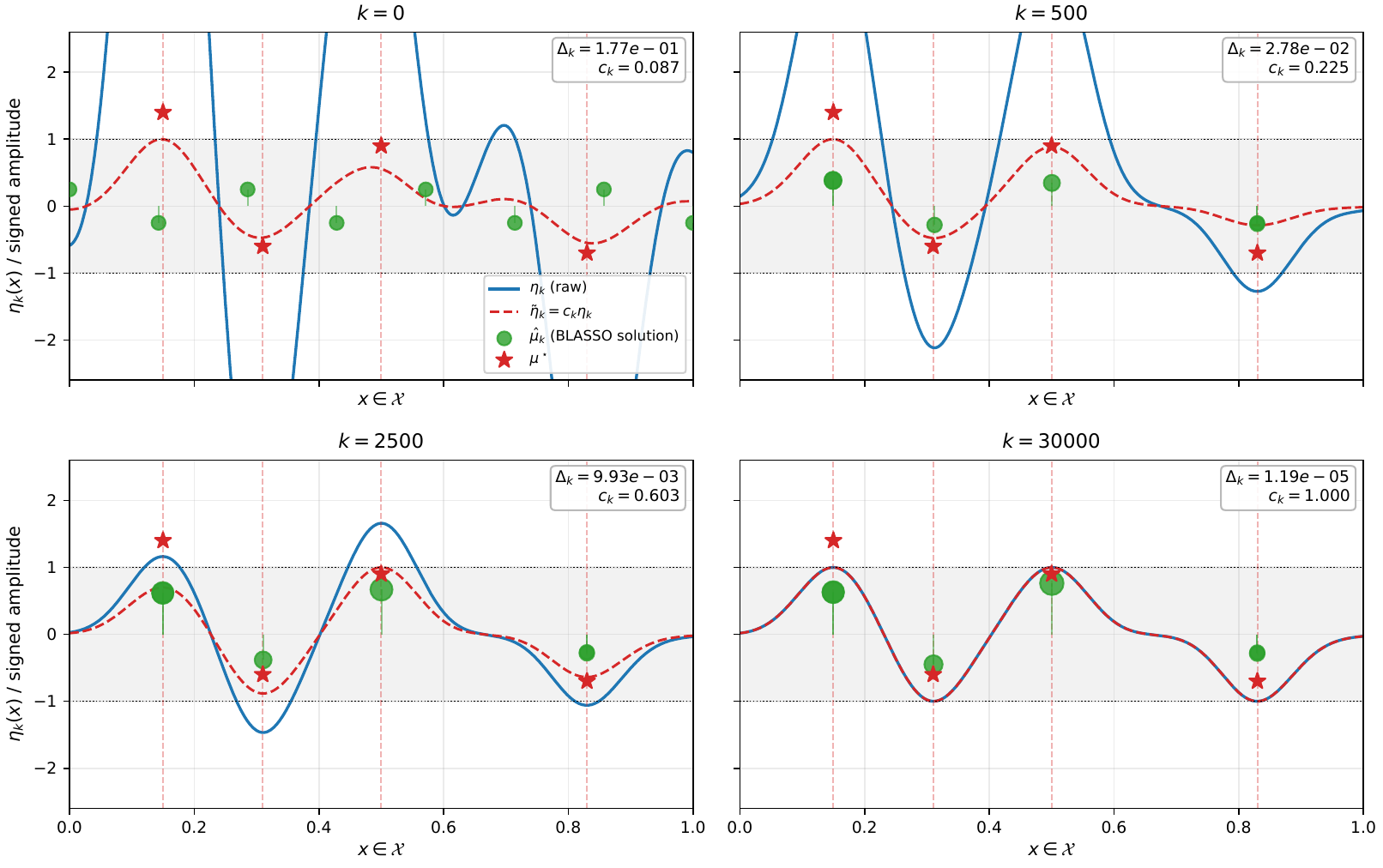}
    \caption{Dual certificate evolution along CPGD (synthetic 1D Gaussian). Each panel overlays, at iteration $k$: the Mirror-Alignment certificate $\eta_k=-\Phi^\star h_k/\lambda$ (blue); its rescaled contraction $\tilde\eta_k=c_k\eta_k$, $c_k=\min(1,1/\|\eta_k\|_\infty)$ (red dashed); the recovered BLASSO solution $\hat\mu_k$ as individual particles (green; one stem per active CPGD particle, marker size scaled by amplitude); and the truth $\mu^\star$ (red~stars).}
    \label{fig:certificate_evolution}
\end{figure}

\subsection{Dual certificate evolution along CPGD}\label{subsec:exp_certificate} 
Figure~\ref{fig:certificate_evolution} tracks the dual certificate $\eta_k$ along the CPGD trajectory of a synthetic 1D super-resolution problem: four spikes at $(0.15,0.31,0.50,0.83)$ with signed amplitudes $(+1.4,-0.6,+0.9,-0.7)$, Gaussian feature map of bandwidth $\sigma_\varphi=0.04$, $N=200$ samples, observation noise $\sigma_w=10^{-2}$, $\lambda=10^{-2}$, and $T=30\,000$ iterations, the particles being initialized on a uniform grid of $[0,1]$ with alternating signs.
The convergence pattern visible in Figure~\ref{fig:certificate_evolution} is the operational form of the early-stopping theory: the rescaled dual $\tilde h_k\in K_\lambda$ is dual-feasible at every iteration (Remark~\ref{rem:dual_rescaling}), and the Fenchel--Young duality gap $\Delta_k = L_F(\Phi\mu_k\,\|\,\tilde h_k)+\lambda L_R(\mu_k\,\|\,\tilde\eta_k)$ is an oracle-free optimality certificate (Theorem~\ref{thm:dual_gap} and~\eqref{eq:opt_gap_bound}) computed without any access to the unknown~$\mu^\star$; along the trajectory it is observed to decrease to zero. 

In Figure~\ref{fig:certificate_evolution}, the inset reports the oracle-free gap $\Delta_k$. Across the four snapshots $\Delta_k$ falls by more than four orders of magnitude ($1.8\!\times\!10^{-1}\to1.2\!\times\!10^{-5}$), $c_k\to1$, and $\|\eta_k\|_\infty\to1$ with on-support saturation $\eta_k(\hat x_i)\approx\operatorname{sign}(\hat a_i)$ at the active atoms $\hat a_i\delta_{\hat x_i}$ of $\mu_k$, the operational form of the source condition. Each spike is recovered by one or two atoms whose heights match the amplitude of the ground truth up to the standard $\sim\!10$--$25\%$ $\calO(\lambda)$ BLASSO shrinkage (positions exact to $\sim\!10^{-3}$).

\begin{subequations}
\noindent
Theorem~\ref{thm:dual_gap} applied at $(\mu_k,\tilde h_k)$ then delivers the finite, computable Fenchel--Young duality-gap decomposition
\begin{equation}
    \label{eq:honest_FY_gap}
    \Delta_k \;=\; \underbrace{L_F(\Phi\mu_k\parallel\tilde h_k)}_{\geq 0,\ =\,0\,\iff\,c_k=1}\;+\;\underbrace{\lambda\bigl(\|\mu_k\|_{\TV}-\langle\tilde\eta_k,\mu_k\rangle\bigr)}_{\geq 0,\ \text{finite}}\;\geq\;J(\mu_k)-\inf J\,.
\end{equation}
The data-fidelity piece $L_F$ in~\eqref{eq:honest_FY_gap} is generically non-zero off Mirror Alignment: trading exact Mirror Alignment ($L_F = 0$) for dual feasibility ($\tilde h_k\in K_\lambda$) is what produces a finite bound, and is unavoidable. In the least-squares case, the trade-off is $L_F(\Phi\mu_k\,\|\,\tilde h_k) = (1-c_k)^2\,F(\Phi\mu_k)$, which vanishes as $c_k\to 1$, i.e.\ once the iterate becomes dual-feasible, so $\Delta_k\to 0$ requires both pieces to vanish jointly.
\end{subequations}

\begin{figure}[!t]
    \centering
    \includegraphics[width=0.66\linewidth]{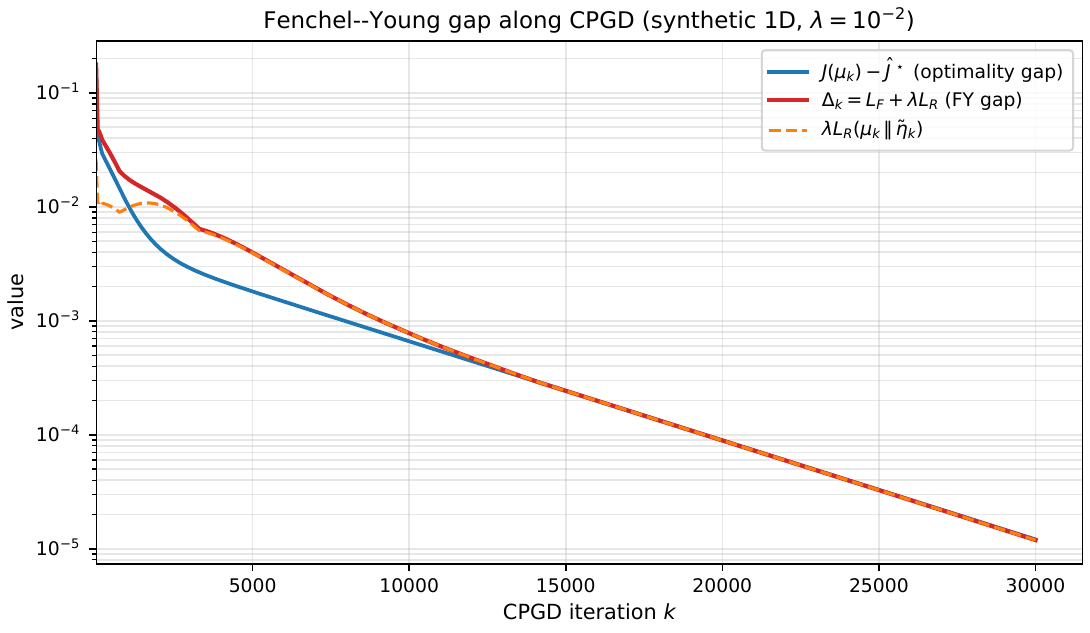}
    \caption{Fenchel--Young duality-gap decomposition $\Delta_k=L_F+\lambda L_R$ (Theorem~\ref{thm:dual_gap}) at the rescaled Mirror-Alignment dual $\tilde h_k=c_k\nabla F(\Phi\mu_k)$, along the trajectory of Figure~\ref{fig:certificate_evolution} ($T=30\,000$ iterations; semilog axes). Red: total Fenchel--Young gap $\Delta_k$; orange: the regularizer piece $\lambda L_R(\mu_k\parallel\tilde\eta_k)$, which tracks $\Delta_k$ once the data-fidelity piece $L_F=(1-c_k)^2F(\Phi\mu_k)$ has vanished ($c_k\to1$); blue: primal optimality gap $J(\mu_k)-\hat J^\star$, with $\hat J^\star$ the minimal objective over an extended reference run (twice the displayed horizon, minimum taken over the tail) as a proxy for $\inf J$.
    }
    \label{fig:fy_decomposition_synthetic}
\end{figure}

Figure~\ref{fig:fy_decomposition_synthetic} illustrates the oracle-free guarantee. The red curve $\Delta_k$ is computed from the iterate alone, yet it lies above the blue optimality gap $J(\mu_k)-\hat J^\star$ at every iteration (this is the inequality $\Delta_k\ge J(\mu_k)-\inf J$ of~\eqref{eq:opt_gap_bound}), and the two curves merge as $c_k\to1$, so the certificate becomes asymptotically tight.
Once the data-fidelity piece has vanished ($c_k=1$, $L_F=0$), the gap reduces to the single regularizer term~$\lambda L_R$, exactly the quantity whose sub-tolerance level Theorem~\ref{thm:algorithmic_stability} converts into a feasible Br\o ndsted--Rockafellar proxy. A practitioner targeting accuracy $\epsilon$ may therefore stop as soon as $\Delta_k\le\epsilon$ and certify $J(\mu_k)-\inf J\le\epsilon$ with no oracle. This is the early-stopping rule that the bound is meant to provide, certified by Theorem~\ref{thm:dual_gap} and~\eqref{eq:opt_gap_bound}. The reference value $\hat J^\star$ used to draw the blue curve is itself an oracle quantity (the minimal objective over an extended run); it is shown only to exhibit the bound, and plays no role in the certificate $\Delta_k$, which the optimizer computes online.

\subsection{Certifiability of Lion-K and Muon via the duality-gap decomposition}\label{subsec:exp_lion_muon}

The Fenchel--Young duality-gap decomposition $\Delta(\mu, h) = L_F(\Phi\mu \parallel h) + \lambda L_R(\mu \parallel \eta)$ of Theorem~\ref{thm:dual_gap} is optimizer-agnostic: it holds for any iterate $\mu_t$ produced by any optimizer solving the regularized program $\min_\mu F(\Phi\mu) + \lambda R(\mu)$, be it CPGD, ISTA, SVT, Lion-K, or Muon. Evaluated at the \emph{rescaled} Mirror-Alignment dual $\tilde h_t = c_t\,\nabla F(\Phi\mu_t)$ with $c_t = \min(1,\,\lambda/\|\Phi^\star\nabla F(\Phi\mu_t)\|_{R^\star})$, where $\|\cdot\|_{R^\star}$ is the norm whose unit ball is $\dom R^\star$ (the $\ell_\infty$ norm for $\ell_1$ regularization, the operator norm for the nuclear norm), so that $\tilde h_t$ is dual-feasible, the resulting Fenchel--Young gap $\Delta_t = L_F(\Phi\mu_t\parallel\tilde h_t) + \lambda L_R(\mu_t\parallel\tilde\eta_t)\ge J(\mu_t)-\inf J$~\eqref{eq:opt_gap_bound} is finite, non-negative, and computable from the iterate alone.

We apply this diagnostic to two non-convex optimizers used in the deep-learning community. \citet{chen2023lionk} introduced the Lion-K family, whose update applies the subgradient of a convex map $K$ (the $\operatorname{sign}$ nonlinearity for Lion); \citet{jordan2024muon} introduced Muon, whose update applies the matrix-sign $\operatorname{matsign}(G)=UV^\top$. Both nonlinearities are subgradients of a regularizer: $\operatorname{sign}(x)\in\partial\|x\|_1$ and $UV^\top\in\partial\|U\Sigma V^\top\|_*$. When this sign / matrix-sign step is used as the regularizer's \emph{proximal} operator (soft-thresholding $\operatorname{soft}_\tau(z)=\operatorname{sign}(z)\max(|z|-\tau,0)$ for $\ell_1$, and singular-value thresholding $U\operatorname{soft}_\tau(\Sigma)V^\top$ for the nuclear norm, with the polar factor $UV^\top$ supplied by Muon's matrix-sign / Polar-Express step), the resulting iteration is proximal gradient on $\min_\mu F(\Phi\mu)+\lambda R(\mu)$, namely ISTA and SVT, respectively. In this proximal form Lion-K and Muon solve exactly the $F+\lambda R$ program, and the Fenchel--Young gap certifies them. Figure~\ref{fig:lion_muon_delta} reports this on two ill-conditioned instances (a coherent, $\mathrm{AR}(1)$-correlated, $\ell_1$ compressed-sensing problem and a nuclear-norm matrix-sensing problem), chosen so that convergence is gradual rather than instantaneous: for both optimizers the oracle-free $\Delta_t$ decays steadily to zero while staying above the (oracle) optimality gap $J(\mu_t)-\hat J^\star$ along the trajectory, so the certification stays visible over the $10\,000$ iterations shown.

In summary, Lion-K and Muon share the same algorithmic template, a sign / matrix-sign step, which the source condition of Section~\ref{sec:controlled_gaps} (condition~\eqref{def:dual_certificate_definition}) identifies as the regularizer subgradient for $\ell_1$ and for the nuclear norm. Used as the proximal operator of that regularizer, the step yields proximal gradient on $\min_\mu F(\Phi\mu)+\lambda R(\mu)$, and the duality-gap decomposition of Figure~\ref{fig:lion_muon_delta} certifies, oracle-free, that the gap $\Delta_t\to 0$, hence that the $F+\lambda R$ program is solved. The companion notebook contains the full Lion-K / Muon / Newton--Schulz numerical implementation; we refer to \citet{amsel2026polar} for the polar-expansion analysis underlying the quintic Newton--Schulz iteration that approximates $\operatorname{matsign}$ in our Muon implementation.

\begin{remark}[Convex certificate for a non-convex deployment]
In their original deep-learning use, Lion-K and Muon are run on non-convex training objectives, whereas the certificate of this section targets the convex proximal-gradient surrogate $\min_\mu F(\Phi\mu)+\lambda R(\mu)$ that their ISTA/SVT proximal form solves: $\Delta_t$ certifies optimality \emph{for that convex program}, not for the non-convex training loss. This convex program is the informative object rather than a limitation, since by the mean-field/over-parameterization theory of \citet{chizat2018global,chizat2022sparse} over-parameterized non-convex gradient dynamics can realize the global minimizer of a convex problem lifted to the space of measures, the same style of lifting that underlies the CPGD trajectories of Section~\ref{subsec:exp_certificate}, see~\citep{decastro2023fastpart}. %
\end{remark}

The experiments make the same point in each setting. The oracle-free Fenchel--Young gap $\Delta=L_F+\lambda L_R$ of Theorem~\ref{thm:dual_gap} upper-bounds the primal optimality gap~\eqref{eq:opt_gap_bound} using only the running iterate (Figures~\ref{fig:certificate_evolution}--\ref{fig:fy_decomposition_synthetic}): it tracks the true gap, becomes asymptotically tight as the rescaling contraction $c_k\to1$, and supplies the early-stopping rule ``stop once $\Delta_k\le\epsilon$''. Because the construction is optimizer-agnostic, 
the same scalar certifies the deep-learning optimizers Lion-K and Muon: in proximal form their sign / matrix-sign step is the regularizer proximal operator, they solve $F+\lambda R$, and the gap $\Delta_t\to0$ along the trajectory (Figure~\ref{fig:lion_muon_delta}). For continuous BLASSO and for finite-dimensional sparse / low-rank recovery, certifying optimality reduces to computing a single scalar.

\section*{Conclusion and perspectives}
The Banach-space setting admits several relaxations. The measure space $\calM(\calX)$ may be replaced by any dual Banach space $\calB^\star$. Relaxing the Hilbert space $\calH$ to a reflexive Banach space would additionally require a counterpart of Carlier's inequality, which is where the Hilbert structure enters. A finer quantitative tracking of the Br\o ndsted--Rockafellar proxy under stochastic optimization noise is left for future work. Building on the error bounds established here, the companion paper~\citep{aubin2026localization} develops support localization and exact support recovery for the GBL under a non-degenerate source condition (verified through a Fisher-reweighted Local Positive Curvature pivot, with a Kernel Switch supplying sign-free localization), and attains the parametric statistical rate, up to a logarithmic factor, on the well-specified population logistic GBL.

\subsection*{Disclosure of AI use}
In line with the Leiden Declaration on Artificial Intelligence and Mathematics \citep{leiden2026declaration}, we disclose the use of AI agents (Claude Code, Anthropic) in the preparation of this paper. These agents were used to audit the statements and the proofs, through repeated adversarial verification passes recorded in a change-log, and to help with editorial tasks: consistency of notation, bibliography formatting, and typesetting. The authors produced and checked every result and every proof themselves and take sole responsibility for the correctness of the paper and for its citations.

\bibliographystyle{plainnat}
\bibliography{FYL_Gap}

\newpage

\appendix

\section{Technical lemmas}
\label{sec:lemmas_app}

We here regroup technical lemmas underlying some of our proofs. They sometimes recover well-known results in convex analysis, however, since our setting is naturally non-reflexive Banach spaces, it is easier to reprove the formulations we need rather than quote existing statements, often restricted to the Hilbert case in textbooks.

\subsection{On the forward operator}
\begin{lem}
    \label{lem:dual_Phi_banach}
    Let $\mathcal{B}$ be a Banach space and denote $\mathcal{B}^\star$ its dual space. Let $\mathcal{H}$ be a Hilbert space. Let $\Phi\,:\,\mathcal{B}^\star\to\mathcal{H}$ be a linear weak-$*$-to-weak continuous operator. Then its pre-adjoint $\Phi^\star\,:\,\mathcal{H}\to\mathcal{B}$ reads
    \[
    \Phi^\star\,:\,h\in\mathcal{H}\mapsto \big(b^\star\in\mathcal{B}^\star\mapsto\langle\varphi_{b^\star},h\rangle_{\mathcal H}\big)\in\mathcal{B}\,,
    \]
    where we identified the pre-dual space $\mathcal{B}$ with a subspace of the bidual $\mathcal{B}^{\star\star}$, and where $\varphi_{b^\star}\coloneqq \Phi\, b^\star$ for $b^\star\in\mathcal{B}^\star$.
\end{lem}

\begin{proof}
    Fix $h\in\calH$ and consider the linear functional $\ell_h: \calB^\star \to \mathbb R$ defined by
    \[
        \ell_h(b^\star)\;\coloneqq \; \bracket{h,\, \Phi b^\star}_\calH,\qquad b^\star\in\calB^\star.
    \]
    It is linear. Since $\Phi:(\calB^\star,\sigma(\calB^\star,\calB))\to\calH$ is weak-$*$-to-weak continuous by assumption, and $y\mapsto \bracket{h,y}_\calH$ is a weakly continuous linear functional on $\calH$, the composition $\ell_h$ is a \emph{weak-$*$ continuous} linear functional on $\calB^\star$. A standard result in functional analysis states that the dual of $(\calB^\star,\sigma(\calB^\star,\calB))$ is exactly the canonical identification $\calB\hookrightarrow\calB^{\star\star}$, hence every weak-$*$ continuous linear functional on $\calB^\star$ is of the form
    \[
        b^\star\;\longmapsto\; \bracket{b,\, b^\star}_{\calB,\calB^\star}\quad\text{for a unique } b\in\calB.
    \]
    Therefore there exists a unique $\Phi^\star h\in\calB$ such that, for all $b^\star\in\calB^\star$, $\bracket{h,\, \Phi b^\star}_\calH=\bracket{\Phi^\star h,\, b^\star}_{\calB,\calB^\star}$. This defines a linear operator $\Phi^\star:\calH\to\calB$. Moreover, we obtain the norm estimate
    \[
        \|\Phi^\star h\|_\calB\;=\;\sup_{\|b^\star\|\le 1}\big|\bracket{\Phi^\star h,\, b^\star}\big|\;=\;\sup_{\|b^\star\|\le 1}\big|\bracket{h,\, \Phi b^\star}_\calH\big|\;\le\; \|h\|_\calH\,\|\Phi\|,
    \]
    which shows that $\Phi^\star$ is bounded. Finally, set $\varphi_{b^\star}\coloneqq \Phi b^\star$ for $b^\star\in\calB^\star$. The previous identity can be written
    \[
        \bracket{\Phi^\star h,\, b^\star}_{\calB,\calB^\star}\;=\;\bracket{h,\, \varphi_{b^\star}}_\calH,\qquad \forall\, b^\star\in\calB^\star,
    \]
    which proves the claimed expression of $\Phi^\star$ via the canonical identification $\calB\hookrightarrow\calB^{\star\star}$.
\end{proof}

\begin{remark}[$\Phi^\star$ in the measure case]\label{rem:phi_star_via_diracs}
$\Phi^\star h$ is an element of the pre-dual $\calB$ of Lemma~\ref{lem:dual_Phi_banach}. When $\calB^\star=\calM(\calX)$, $\calM(\calX)$ is generated (in the weak-$*$ topology) by Dirac measures, so $\Phi^\star h$ is fully specified by its values on Dirac inputs: $(\Phi^\star h)(t)=\langle\varphi_t,h\rangle_\calH$ where $\varphi_t\coloneqq\Phi\delta_t$.
\end{remark}

\begin{lem}
    \label{lem:dual_Phi}
    Let $\Phi\,:\,\mathcal{M}(\mathcal{X})\to\mathcal{H}$ be a linear weak-$*$-to-weak continuous operator, with $\calH$ a separable Hilbert space. Then its pre-adjoint $\Phi^\star\,:\,\mathcal{H}\to\mathcal{C}_0(\mathcal{X})$ reads
    \[
    \Phi^\star\,:\,h\in\mathcal{H}\mapsto \big(t\in\mathcal{X}\mapsto\langle\varphi_t,h\rangle_{\mathcal H}\big)\in\mathcal{C}_0(\mathcal{X})\,,
    \]
    where we identified the pre-dual space $\mathcal{C}_0(\mathcal{X})$ with a subspace of the dual $\mathcal{M}(\mathcal{X})^\star$, and where $\varphi_t\coloneqq \Phi\, \delta_t$ with~$\delta_t$ the Dirac measure at $t\in\mathcal{X}$. Moreover, for every $\mu\in\calM(\calX)$, one has the Bochner integral representation in $\calH$:
    \[
        \Phi\mu \,=\, \int_{\calX} \varphi_t\,\mathrm d\mu(t).
    \]
\end{lem}

\begin{proof}
    Applying Lemma~\ref{lem:dual_Phi_banach} with $\calB=\calC_0(\calX)$, $\calB^\star=\calM(\calX)$ (Riesz--Markov, \citealp[Theorem~6.19]{rudin1987real}) gives a unique $\Phi^\star h\in\calC_0(\calX)$ such that $\bracket{h,\Phi\mu}_\calH = \bracket{\Phi^\star h,\mu}_{\calC_0(\calX),\calM(\calX)}$ for every $\mu\in\calM(\calX)$, with $\|\Phi^\star h\|_\infty\le\|h\|_\calH\,\|\Phi\|$. Setting $\varphi_t\coloneqq\Phi\delta_t\in\calH$ and evaluating the identity at $\mu=\delta_t$ yields $(\Phi^\star h)(t)=\bracket{\varphi_t,h}_\calH$, so $\Phi^\star h$ is precisely the continuous function $t\mapsto\bracket{\varphi_t,h}_\calH$.

    We now prove the integral representation. First, note that 
    \[
    \|\varphi_t\|_\calH=\|\Phi\delta_t\|_\calH\le \|\Phi\|\,\|\delta_t\|_{\TV}=\|\Phi\|
    \]
    for all $t\in\calX$, so $t\mapsto\varphi_t\in\calH$ is bounded. Next, for each fixed $h\in\calH$, we already identified $\Phi^\star h\in\calC_0(\calX)$ and established $(\Phi^\star h)(t)=\bracket{\varphi_t,h}_\calH$. Since $\Phi^\star h\in\calC_0(\calX)$, the scalar map $t\mapsto(\Phi^\star h)(t)=\bracket{\varphi_t,h}_\calH$ is continuous on~$\calX$. Hence $t\mapsto\varphi_t$ is weakly continuous (i.e., all scalar evaluations $\bracket{\varphi_t,h}_\calH$ are continuous in $t$). Because $\calH$ is separable, weak measurability/continuity implies strong (Bochner) measurability by Pettis' theorem, see~\citep[Theorem~1.1.6]{hytonen2016analysis}. Consequently, since $\int_{\calX}\|\varphi_t\|_\calH\,\mathrm d|\mu|(t)\le\|\Phi\|\,\|\mu\|_{\TV}<\infty$, for any finite signed measure $\mu\in\calM(\calX)$ the Bochner integral $\int_{\calX}\varphi_t\,\mathrm d\mu(t)$ is well-defined in $\calH$ and satisfies, for all $h\in\calH$,
    \[
            \Big\langle h, \int_{\calX}\varphi_t\,\mathrm d\mu(t) \Big\rangle_\calH \,=\, \int_{\calX} \langle h,\varphi_t \rangle_\calH\,\mathrm d\mu(t).
    \]
    Using the identity $\langle h,\Phi\mu \rangle_\calH=\langle \Phi^\star h,\mu \rangle_{\calC_0(\calX),\calM(\calX)}=\int_{\calX} (\Phi^\star h)(t)\,\mathrm d\mu(t)=\int_{\calX}\langle h,\varphi_t \rangle_\calH\,\mathrm d\mu(t)$, we deduce that 
    \[
            \langle h,\Phi\mu \rangle_\calH=\big\langle h, \int_{\calX}\varphi_t\,\mathrm d\mu(t) \big\rangle_\calH
    \] 
    for all $h\in\calH$. By uniqueness of the Riesz representation in $\calH$, this implies
    \[
            \Phi\mu \,=\, \int_{\calX} \varphi_t\,\mathrm d\mu(t),
    \]
    which is the desired representation.
\end{proof}

\subsection{On the Fenchel--Young losses}
\begin{lem}[Properties of Fenchel--Young losses on Banach spaces; cf.\ \citealp{blondel2019learning}]
\label{lem:fy_properties}
    Let $\calB$ be a Banach space and $\calB^\star$ its dual space. Let $G: \calB^\star \to \mathbb{R} \cup \{+\infty\}$ be a proper convex function on the Banach space $\calB^\star$ that is \emph{weak-$*$ lower semi-continuous} $(\sigma(\calB^\star,\calB)$-l.s.c., equivalently $G=G^{\star\star}$ with $G^\star$ taken on the pre-dual $\calB)$. Let $G^\star$ be the Legendre--Fenchel conjugate of $G$ given only on the pre-dual space $\calB$, identified with a weak-$*$ dense subset of the bidual $\calB^{\star\star}$. For any $u \in \calB^\star$ and $v \in \calB$, the following properties hold:
\begin{enumerate}
    \item \textbf{Non-negativity:} $L_G(u \parallel v) \coloneqq  G(u) + G^\star(v) - \langle v, u \rangle_{\calB,\calB^\star} \ge 0$.
    \item \textbf{Zero-loss equivalence:} The following statements are unconditionally equivalent for any $(u, v) \in \calB^\star \times \calB$:
    \[
    \text{$(i)$ }L_G(u \parallel v) = 0
    \qquad
    \text{$(ii)$ }v \in \partial G(u)
    \qquad
    \text{$(iii)$ }u \in \partial G^\star(v)
    \]
    \item \textbf{Differentiable case:} If $G$ is Gâteaux differentiable at $u$ with $\nabla G(u)\in\calB$, then $L_G(u \parallel v) = 0 \iff v = \nabla G(u)$.
\end{enumerate}
\end{lem}

\begin{proof}
The non-negativity (1) is a direct consequence of the Fenchel--Young inequality, which states that $G^\star(v) \ge \langle u, v \rangle_{\calB^\star,\calB} - G(u)$ by definition of the supremum in the conjugate.

\medskip\noindent
For the equivalence (2), we proceed as follows:
\begin{itemize}
    \item $(ii) \implies (i)$: If $v \in \partial G(u)$, then by definition of the subgradient, $G(w) \ge G(u) + \langle v, w - u \rangle_{\calB,\calB^\star}$ for all~$w$. Rearranging implies $\langle v, u \rangle_{\calB,\calB^\star} - G(u) \ge \langle v, w \rangle_{\calB,\calB^\star} - G(w)$ for all $w$. Taking the supremum over $w$ on the RHS gives $\langle v, u \rangle_{\calB,\calB^\star} - G(u) \ge G^\star(v)$, or equivalently $G(u) + G^\star(v) - \langle v, u \rangle_{\calB,\calB^\star} \le 0$. Since the loss is always non-negative, it must be 0.
    \item $(i) \implies (ii)$: If $G(u) + G^\star(v) - \langle v, u \rangle_{\calB,\calB^\star} = 0$, then $\langle v, u \rangle_{\calB,\calB^\star} -  G(u) = G^\star(v)$. By the Fenchel--Young inequality we have $G^\star(v) \ge \langle v, w \rangle_{\calB,\calB^\star} - G(w)$ for any $w$. Combining both leads to the inequality $G(w)\geq G(u) + \langle v, w-u \rangle_{\calB,\calB^\star}$ for any $w$ which proves that $v\in \partial G(u)$.
\end{itemize}

The equivalence $(i)\Leftrightarrow(iii)$ follows by symmetry from the $(i)\Leftrightarrow(ii)$ argument applied to $G^\star$, using the Fenchel--Moreau identity $G^{\star\star}=G$ on the dual pairing $(\calB,\calB^\star)$, which holds by the weak-$*$ lower semi-continuity of $G$ (no reflexivity of $\calB$ needed).
    
For (3), if $G$ is differentiable, the subdifferential contains a single element, $\partial G(u) = \{\nabla G(u)\}$. Thus condition (ii) simplifies to $v = \nabla G(u)$.
See \citet[Proposition~1]{blondel2019learning} for further details on these convex analysis identities in the context of machine learning losses.
\end{proof}

\begin{lem}[Bregman divergence identity]
\label{lem:bregman_identity}
    Let $F: \calH \to \mathbb{R}$ be a strictly convex, continuously differentiable function. Let $F^\star: \calH \to \mathbb{R} \cup \{+\infty\}$ be its Fenchel conjugate, which is differentiable on the interior of its domain. For any $u \in \calH$ and any $h \in \operatorname{int}(\operatorname{dom}(F^\star))$, the following identity holds:
\begin{equation}
\notag
        F(u \mid \nabla F^\star(h)) = F^\star(h \mid \nabla F(u)) = L_F(u \parallel h).
\end{equation}
Here $\nabla F^\star(\nabla F(u))$ is understood as $u$, the unique element of $\partial F^\star(\nabla F(u))$ $($a singleton by strict convexity of $F)$; with this reading the middle term is well defined even when $\nabla F(u)\notin\operatorname{int}(\dom F^\star)$.
\end{lem}

\begin{proof}

\textbf{1. Proving $F(u \mid \nabla F^{\star}(h)) = L_F(u \parallel h)$:}
Let $y = \nabla F^{\star}(h)$. By the properties of Legendre--Fenchel conjugates for strictly convex functions, the gradient maps are inverses: $\nabla F(y) = \nabla F(\nabla F^{\star}(h)) = h$.
Substituting $y = \nabla F^{\star}(h)$ and $\nabla F(y) = h$ into the Bregman definition:
\begin{align*}
    F(u \mid \nabla F^{\star}(h)) &= F(u) - F(\nabla F^{\star}(h)) - \langle h, u - \nabla F^{\star}(h) \rangle \\
    &= F(u) - F(\nabla F^{\star}(h)) - \langle h, u \rangle + \langle h, \nabla F^{\star}(h) \rangle.
\end{align*}
Recall the Fenchel equality for the pair $(h, \nabla F^{\star}(h))$: since $\nabla F^{\star}(h)$ is the subgradient at $h$, we have $F^{\star}(h) + F(\nabla F^{\star}(h)) = \langle h, \nabla F^{\star}(h) \rangle$, which implies $F(\nabla F^{\star}(h)) = \langle h, \nabla F^{\star}(h) \rangle - F^{\star}(h)$.
Substituting:%
\begin{align*}
    F(u \mid \nabla F^{\star}(h)) &= F(u) - (\langle h, \nabla F^{\star}(h) \rangle - F^{\star}(h)) - \langle h, u \rangle + \langle h, \nabla F^{\star}(h) \rangle \\
    &= F(u) + F^{\star}(h) - \langle u, h \rangle = L_F(u \parallel h).
\end{align*}

\textbf{2. Proving $F^{\star}(h \mid \nabla F(u)) = L_F(u \parallel h)$.} With the stated convention, $\nabla F^\star(\nabla F(u))=u$: indeed $u\in\partial F^\star(\nabla F(u))$ by the Fenchel equality at the pair $(u,\nabla F(u))$, and strict convexity of $F$ makes this subgradient unique. By definition,
\[
    F^\star(h \mid \nabla F(u)) = F^\star(h) - F^\star(\nabla F(u)) - \langle u,\, h - \nabla F(u) \rangle.
\]
The Fenchel equality at $(u,\nabla F(u))$ gives $F^\star(\nabla F(u)) = \langle u, \nabla F(u) \rangle - F(u)$. Substituting yields
\[
    F^\star(h \mid \nabla F(u)) = F^\star(h) + F(u) - \langle u, h \rangle = L_F(u \parallel h).
\]
Combining the two equalities yields the claimed identity.
\end{proof}

\subsection{On the pulled-back regularizer}
\label{subsec:pullback_app}

\begin{lem}[Regularity of the pulled-back regularizer]\label{lem:pullback_regularity}
Let \Cref{ass:BR_regularizer} hold, and let $R^\star$ be the Legendre--Fenchel conjugate of $R$ taken
on the pre-dual pairing $(\calB,\calB^\star)$,
\[
    R^\star(f)=\sup_{w\in\calB^\star}\big\{\langle f,w\rangle_{\calB,\calB^\star}-R(w)\big\},\qquad f\in\calB\,.
\]
Let $u$ be the pulled-back regularizer of~\eqref{eq:BR_pullback}. Then the following hold.
\begin{enumerate}
    \item[$(a)$] $R^\star:\calB\to\R\cup\{+\infty\}$ is proper, convex and norm lower semi-continuous, and
    $R^{\star\star}=R$. In particular $R^\star>-\infty$ everywhere on $\calB$ and $u>-\infty$ everywhere on $\calH$.
    \item[$(b)$] $u$ is convex on $\calH$.
    \item[$(c)$] $u$ is lower semi-continuous on $\calH$.
    \item[$(d)$] The effective domain is
    \[
        \dom u=\Big\{h\in\calH:-\frac{\Phi^\star h}{\lambda}\in\dom R^\star\Big\}
        =(\Phi^\star)^{-1}\big(-\lambda\,\dom R^\star\big),
    \]
    and $u$ is proper if and only if $\dom u\neq\varnothing$; under \Cref{ass:BR_regularizer} the map
    $u$ is therefore proper.
\end{enumerate}
\end{lem}

If $0\in\dom R^\star$ then $0\in\dom u$, so the feasibility requirement
    $\dom u\neq\varnothing$ holds automatically. This is the case for every norm regularizer
    $R=\|\cdot\|_{\calB^\star}$, whose conjugate on the pre-dual is $R^\star=\iota_{B_\calB}$, the indicator
    of the closed unit ball $B_\calB=\{f\in\calB:\|f\|_\calB\le1\}$ of the pre-dual, so that $0\in\dom R^\star$. For $R=\|\cdot\|_{\TV}$ on $\calB^\star=\calM(\calX)$,
    where $\calB=\calC_0(\calX)$, this reads $R^\star=\iota_{\{f\in\calC_0(\calX):\|f\|_\infty\le1\}}$
    (\Cref{sec:applications}). Consequently $u$ is a proper, convex and lower semi-continuous function on the Hilbert space
    $\calH$, which is exactly the regularity required to apply Carlier's constructive
    Br\o ndsted--Rockafellar theorem~\citep[Theorem~4.1]{carlier2023fenchel} in
    \Cref{thm:algorithmic_stability}.

    We do not use, and do not have,
coercivity of $u$. For a norm regularizer $u=\lambda\,\iota_{B_\calB}\circ A$ is the indicator of the set
$K_\lambda=\{h\in\calH:\|\Phi^\star h\|_\calB\le\lambda\}$, which contains $\operatorname{Ker}\Phi^\star$ and is unbounded
whenever $\Phi^\star$ is not bounded below, so $u$ does not tend to $+\infty$ as $\|h\|_\calH\to\infty$.
What the downstream argument needs is only that $u$ is proper, convex and lower semi-continuous,
equivalently that $u^{\star\star}=u$ and $u^\star$ is proper, which is what the value $u^\star(-\Phi\mu)$ in
\Cref{prop:BR_gap_inequality} uses.

\begin{proof}
Throughout, the pairing $\langle f,w\rangle_{\calB,\calB^\star}$ is between $f\in\calB$ and
$w\in\calB^\star$, and each $w\in\calB^\star$ is a bounded linear functional on $\calB$ with
$|\langle f,w\rangle_{\calB,\calB^\star}|\le\|f\|_\calB\,\|w\|_{\calB^\star}$. Write $A:\calH\to\calB$ for
the linear map:
\begin{equation*}
    \text{$Ah\coloneqq-\Phi^\star h/\lambda$, so that $u=\lambda\,(R^\star\circ A)$.}
\end{equation*}

\medskip\noindent\emph{Proof of $(a)$.}
For each $w\in\dom R$ the map $f\mapsto\langle f,w\rangle_{\calB,\calB^\star}-R(w)$ is affine and norm
continuous on~$\calB$, since $f\mapsto\langle f,w\rangle_{\calB,\calB^\star}$ is a bounded linear functional
and $-R(w)$ is a finite constant. By definition $R^\star$ is the pointwise supremum of this family.
A pointwise supremum of continuous affine functionals is convex and norm lower
semi-continuous, so $R^\star$ is convex and norm lower semi-continuous on $\calB$, and its epigraph
$\operatorname{epi}R^\star=\bigcap_{w\in\dom R}\operatorname{epi}[\langle\cdot,w\rangle_{\calB,\calB^\star}-R(w)]$
is a closed convex subset of $\calB\times\R$.
Now, fix $w_0\in\dom R$. Then for every $f\in\calB$,
\[
    R^\star(f)\;\ge\;\langle f,w_0\rangle_{\calB,\calB^\star}-R(w_0)\;>\;-\infty,
\]
so $R^\star>-\infty$ everywhere. It remains to show
$R^\star\not\equiv+\infty$, and this is the only step that uses the weak-$*$ lower semi-continuity of $R$,
rather than only its convexity. The space $(\calB^\star,\sigma(\calB^\star,\calB))$ is a locally convex
Hausdorff space whose topological dual is exactly $\calB$, so the conjugation on the pairing
$(\calB,\calB^\star)$ and the biconjugation on $\calB^\star$ are the ones associated with this duality. As
$R$ is proper, convex and $\sigma(\calB^\star,\calB)$-lower semi-continuous, the Fenchel--Moreau theorem
\citep[Theorem~2.3.3]{zalinescu2002convex} gives $R^{\star\star}=R$.
Hence
$\dom R^\star\neq\varnothing$, so $R^\star$ admits a $\sigma(\calB^\star,\calB)$-continuous affine minorant
indexed by $\calB$ and is finite somewhere. Together with $R^\star>-\infty$, this makes $R^\star$ proper.
Since $R^\star>-\infty$ and $\lambda>0$, the composition $u=\lambda\,R^\star\circ A$ satisfies
$u>-\infty$ everywhere. This proves (a).

Furthermore, by Lemma~\ref{lem:dual_Phi_banach},  $\Phi^\star:\calH\to\calB$ is bounded linear, with
$\|\Phi^\star h\|_\calB\le\|\Phi\|\,\|h\|_\calH$. Since $\lambda>0$, the map $A=-\Phi^\star/\lambda$ is
linear and norm-to-norm continuous from $\calH$ to $\calB$, with
$\|Ah\|_\calB\le(\|\Phi\|/\lambda)\,\|h\|_\calH$. 

\medskip\noindent\emph{Proof of $(b)$.}
Convexity of $u$ is the standard fact that a convex function composed with an affine map is convex. 

\medskip\noindent\emph{Proof of $(c)$.}
A function is lower semi-continuous if and only if its epigraph is closed. Define the linear map
\[
    \Lambda:\calH\times\R\to\calB\times\R,\qquad \Lambda(h,t)\coloneqq\big(Ah,\;t/\lambda\big),
\]
which is continuous, being the product of the continuous map $A$ from above and the continuous scaling
$t\mapsto t/\lambda$. Since $\lambda>0$, for $(h,t)\in\calH\times\R$ we have
\[
    (h,t)\in\operatorname{epi}u
    \;\Longleftrightarrow\; \lambda\,R^\star(Ah)\le t
    \;\Longleftrightarrow\; R^\star(Ah)\le t/\lambda
    \;\Longleftrightarrow\; \big(Ah,\,t/\lambda\big)\in\operatorname{epi}R^\star,
\]
so $\operatorname{epi}u=\Lambda^{-1}(\operatorname{epi}R^\star)$. From $(a)$, the set
$\operatorname{epi}R^\star$ is closed in $\calB\times\R$, and $\Lambda$ is continuous, hence
$\operatorname{epi}u$ is closed in $\calH\times\R$ and $u$ is lower semi-continuous. Equivalently, each
sublevel set $\{u\le\alpha\}=A^{-1}(\{R^\star\le\alpha/\lambda\})$ is the preimage of a closed convex set
under the continuous $A$, hence closed. The only inputs specific to our setting are the norm continuity of
$A$ and the closedness of $\operatorname{epi}R^\star$ above.

\medskip\noindent\emph{Proof of $(d)$.}
Since $R^\star>-\infty$ and $\lambda>0$, the value $u(h)=\lambda\,R^\star(Ah)$ is finite if and only if
$R^\star(Ah)<+\infty$, that is if and only if $Ah=-\Phi^\star h/\lambda\in\dom R^\star$. Hence
\[
    \dom u=\big\{h\in\calH:-\tfrac{\Phi^\star h}{\lambda}\in\dom R^\star\big\}
    =A^{-1}(\dom R^\star)=(\Phi^\star)^{-1}\big(-\lambda\,\dom R^\star\big),
\]
matching~\eqref{eq:BR_pullback}. As $u$ never takes the value $-\infty$ (by $(a)$), $u$ is proper if and
only if it is finite somewhere, that is if and only if $\dom u\neq\varnothing$. Under
\Cref{ass:BR_regularizer} we assume $\dom u\neq\varnothing$, so $u$ is proper.

\end{proof}

\begin{lem}[Dual certificates in the range and its closure]\label{lem:certificate_range_closure}
Let \Cref{ass:BR_regularizer} hold, fix $\mu\in\dom R$, and~write
\[
    s(\mu)\coloneqq\sup_{f\in\ran\Phi^\star}\big[\langle f,\mu\rangle_{\calB,\calB^\star}-R^\star(f)\big],
\]
so that, by \Cref{prop:BR_gap_inequality}, equality holds in~\eqref{eq:BR_gap_inequality} if and only if
$s(\mu)=R(\mu)$. Moreover:
\begin{enumerate}
    \item[$(i)$] If $\mu$ admits a dual certificate in the range, i.e.\ some
    $f\in\ran\Phi^\star$ with $f\in\partial R(\mu)$, then $s(\mu)=R(\mu)$ and equality holds
    in~\eqref{eq:BR_gap_inequality}.
\end{enumerate}
Assume in addition that $R=\|\cdot\|_{\calB^\star}$ is a norm regularizer, so that $R^\star=\iota_{B_\calB}$ is
the indicator of the closed unit ball $B_\calB=\{f\in\calB:\|f\|_\calB\le1\}$ of the pre-dual
(Lemma~\ref{lem:pullback_regularity}(e)) and
$s(\mu)=\sup_{f\in\ran\Phi^\star,\,\|f\|_\calB\le1}\langle f,\mu\rangle_{\calB,\calB^\star}$; the
canonical instance being $R=\|\cdot\|_{\TV}$ on $\calB^\star=\calM(\calX)$, with
$B_\calB=\{f\in\calC_0(\calX):\|f\|_\infty\le1\}$. Then:
\begin{enumerate}
    \item[$(ii)$] If $\mu$ admits a certificate only in the closure, that is
    $f^\star\in\overline{\ran\Phi^\star}\cap\partial R(\mu)$ with the closure taken in
    $\|\cdot\|_\calB$, then $s(\mu)=R(\mu)$ by rescaling, so equality holds in~\eqref{eq:BR_gap_inequality}.
    For $R=\|\cdot\|_{\TV}$ this hypothesis is an $f^\star$ in the sup-norm closure of
    $\ran\Phi^\star$ with $\|f^\star\|_\infty\le1$ and
    $\langle f^\star,\mu\rangle_{\calB,\calB^\star}=\|\mu\|_{\TV}$.
    \item[$(iii)$] If $\ran\Phi^\star$ is dense in $\calB$,  
    then $s(\mu)=R(\mu)$ for
    every $\mu\in\dom R$, so equality holds in~\eqref{eq:BR_gap_inequality} throughout~$\dom R$.
\end{enumerate}
\end{lem}
\begin{remark}A certificate in $\ran\Phi^\star$, or even in $\calB$, need not exist, so the converse of~$(i)$ is false. For a norm regularizer the subdifferential chain rule $\partial u(\bar h)=-\Phi\,\partial R^\star(\bar\eta)$ nonetheless holds with equality (Proposition~\ref{prop:BR_primal_lift}); the failure here is of certificate existence in $\ran\Phi^\star$, not of that equality.

To see that a certificate need not exist, take $\calX=[0,1]$, $\calB=\calC(\calX)$ and $R=\|\cdot\|_{\TV}$,
and choose two disjoint sequences $\{a_k\}_{k\ge1}$, $\{b_k\}_{k\ge1}$ in $[0,1]$, each dense in $[0,1]$. Now, the finite signed measure $\mu\coloneqq\sum_{k\ge1}2^{-k}(\delta_{a_k}-\delta_{b_k})$ has
$\|\mu\|_{\TV}=\sum_{k\ge1}2^{-k}\cdot2=2$. Any $f\in\partial R(\mu)$ has $\|f\|_\infty\le1$ and
$\langle f,\mu\rangle_{\calB,\calB^\star}=\sum_{k\ge1}2^{-k}\big(f(a_k)-f(b_k)\big)=2$; since each summand is
at most $2^{-k}\cdot2$, this forces $f(a_k)=1$ and $f(b_k)=-1$ for every $k$. Continuity of $f$ with $f=1$ on
the dense set $\{a_k\}$ gives $f\equiv1$, while $f=-1$ on the dense set $\{b_k\}$ gives $f\equiv-1$, a
contradiction. Hence $\partial R(\mu)\cap\calB=\varnothing$: no certificate exists in $\calB$, and a
fortiori none in $\ran\Phi^\star$, even though $s(\mu)=\|\mu\|_{\TV}$ as soon as $\ran\Phi^\star$ is dense in $\calC_0(\calX)$, by part~$(iii)$.
\end{remark}
\begin{proof}[Proof of Lemma~\ref{lem:certificate_range_closure}]
For any $f\in\calB$ the Fenchel--Young inequality in the pairing $(\calB,\calB^\star)$ reads
$R(\mu)+R^\star(f)\ge\langle f,\mu\rangle_{\calB,\calB^\star}$, with equality if and only if
$f\in\partial R(\mu)$; equivalently
\begin{equation}\label{eq:certificate_FY_equality}
    f\in\partial R(\mu)\iff \langle f,\mu\rangle_{\calB,\calB^\star}-R^\star(f)=R(\mu).
\end{equation}
The bound $s(\mu)\le R(\mu)$ always holds: it is the inequality chain in the proof of
\Cref{prop:BR_gap_inequality}, where the supremum over the subspace $\ran\Phi^\star$ is bounded
by the one over $\calB$, equal to $R^{\star\star}(\mu)=R(\mu)$ (Lemma~\ref{lem:pullback_regularity}(a)). So
in each part it is enough to prove $s(\mu)\ge R(\mu)$.

\medskip\noindent\emph{Part $(i)$.} If $f\in\ran\Phi^\star\cap\partial R(\mu)$, then by
\eqref{eq:certificate_FY_equality} the value $\langle f,\mu\rangle_{\calB,\calB^\star}-R^\star(f)=R(\mu)$ is
attained by an $f$ admissible in the supremum defining $s(\mu)$, so $s(\mu)\ge R(\mu)$. This part uses only
\Cref{ass:BR_regularizer}, and \Cref{prop:BR_gap_inequality} already contains it; it is recorded here for reference.

\medskip\noindent
For parts $(ii)$ and $(iii)$ let $R=\|\cdot\|_{\calB^\star}$, so $R^\star=\iota_{B_\calB}$
(Lemma~\ref{lem:pullback_regularity}(e)). Then \eqref{eq:certificate_FY_equality} reduces to
$\|f\|_\calB\le1$ and $\langle f,\mu\rangle_{\calB,\calB^\star}=R(\mu)$, and
$s(\mu)=\sup_{f\in\ran\Phi^\star,\,\|f\|_\calB\le1}\langle f,\mu\rangle_{\calB,\calB^\star}$,
with $\langle f,\mu\rangle_{\calB,\calB^\star}\le\|f\|_\calB\,\|\mu\|_{\calB^\star}\le R(\mu)$ on the feasible
set.

\medskip\noindent\emph{Part $(ii)$: the rescaling argument.} Let
$f^\star\in\overline{\ran\Phi^\star}\cap\partial R(\mu)$, so $\|f^\star\|_\calB\le1$ and
$\langle f^\star,\mu\rangle_{\calB,\calB^\star}=R(\mu)$. Fix $\epsilon\in(0,1)$. As $f^\star$ lies in the
$\|\cdot\|_\calB$-closure of $\ran\Phi^\star$, there is $g\in\ran\Phi^\star$ with
$\|g-f^\star\|_\calB\le\epsilon$, whence $\|g\|_\calB\le\|f^\star\|_\calB+\epsilon\le1+\epsilon$ and
\[
    \langle g,\mu\rangle_{\calB,\calB^\star}
    =\langle f^\star,\mu\rangle_{\calB,\calB^\star}+\langle g-f^\star,\mu\rangle_{\calB,\calB^\star}
    \ge R(\mu)-\|g-f^\star\|_\calB\,\|\mu\|_{\calB^\star}\ge(1-\epsilon)R(\mu).
\]
Because $\ran\Phi^\star$ is a linear subspace, $\tilde g\coloneqq g/(1+\epsilon)$ again lies in
$\ran\Phi^\star$, with $\|\tilde g\|_\calB\le1$, so $\tilde g$ is admissible for $s(\mu)$ and
\[
    \langle\tilde g,\mu\rangle_{\calB,\calB^\star}
    =\frac{\langle g,\mu\rangle_{\calB,\calB^\star}}{1+\epsilon}\ge\frac{1-\epsilon}{1+\epsilon}\,R(\mu).
\]
Thus $s(\mu)\ge\tfrac{1-\epsilon}{1+\epsilon}R(\mu)$ for every $\epsilon\in(0,1)$, and letting
$\epsilon\downarrow0$ gives $s(\mu)\ge R(\mu)$.

\medskip\noindent\emph{Part $(iii)$.} Assume $\ran\Phi^\star$ is dense in $\calB$ and fix
$\mu\in\dom R$. By the definition of the dual norm,
$R(\mu)=\|\mu\|_{\calB^\star}=\sup_{\|f\|_\calB\le1}\langle f,\mu\rangle_{\calB,\calB^\star}$, so for
$\epsilon\in(0,1)$ there is $f_\epsilon\in\calB$ with $\|f_\epsilon\|_\calB\le1$ and
$\langle f_\epsilon,\mu\rangle_{\calB,\calB^\star}\ge(1-\epsilon)R(\mu)$. Running the argument of Part~(ii)
with $f_\epsilon$ in place of $f^\star$ (pick $g\in\ran\Phi^\star$ with
$\|g-f_\epsilon\|_\calB\le\epsilon$ and set $\tilde g=g/(1+\epsilon)$) yields an admissible $\tilde g$ with
$\langle\tilde g,\mu\rangle_{\calB,\calB^\star}\ge\tfrac{1-2\epsilon}{1+\epsilon}R(\mu)$, so
$s(\mu)\ge\tfrac{1-2\epsilon}{1+\epsilon}R(\mu)$; letting $\epsilon\downarrow0$ gives $s(\mu)=R(\mu)$ for
every $\mu\in\dom R$.
\end{proof}

\section{Some data-fidelity terms and Fenchel--Young losses}
\label{sec:datafidelity_FYL_app}

\subsection{Least squares (Gaussian noise)}\label{sec:LS_gaussian}
Let $y \in \mathcal{H}$ be the observed data. We consider the standard quadratic loss corresponding to additive Gaussian noise. This loss is finite-valued, $\calC^1$, strictly convex and coercive, so it lies in $\calF_0$. 

\noindent
\textbf{Primal:} $F(h) = \frac{1}{2} \|h - y\|_{\mathcal{H}}^2$; \textbf{Gradient:} $\nabla F(h) = h - y$; 
\textbf{Dual:} $F^\star(w) = \frac{1}{2} \|w\|_{\mathcal{H}}^2 + \langle w, y \rangle_{\mathcal{H}}$; 
\textbf{Dual gradient:} $\nabla F^\star(w) = w + y$; 
\textbf{Fenchel--Young loss:}
    \begin{equation}
    \notag
        L_F(h \parallel w) = \frac{1}{2} \| h - \nabla F^\star(w) \|_{\mathcal{H}}^2 = \frac{1}{2} \| h - (w + y) \|_{\mathcal{H}}^2.
    \end{equation}
Minimizing the Fenchel--Young loss over the dual variable $w$ thus targets the residual $h - y$.

\subsection{Logistic loss (binary classification)}
\label{sec:datafidelity_FYL_app:logistic}
\begin{subequations}
Let $y \in \{-1, 1\}^n$ be binary labels and $h \in \mathbb{R}^n$ be the logits. 

\noindent
\textbf{Primal:} $F(h) = \sum_{i=1}^n \log(1 + \exp(-y_i h_i))$; 
\textbf{Gradient:} $\nabla F(h)_i = -y_i \sigma(-y_i h_i)$, where $\sigma(t) = (1+e^{-t})^{-1}$ is the sigmoid function; 
\textbf{Dual:} The conjugate corresponds to the \emph{negative} binary entropy (the convex object a conjugate must be; the binary entropy itself is concave). For $w \in \mathbb{R}^n$ such that $-y_i w_i \in [0, 1]$ for all $i$, with the convention $0\log 0=0$ (so the boundary values are finite, e.g.\ $F^\star(0)=0$):
    \begin{equation}
    \notag
        F^\star(w) = \sum_{i=1}^n \left[ (-y_i w_i) \log(-y_i w_i) + (1 + y_i w_i) \log(1 + y_i w_i) \right]\,,\qquad\text{otherwise, }F^\star(w) = +\infty\,;
    \end{equation}
\textbf{Dual gradient:} $\nabla F^\star(w)_i = -y_i \log\left(\frac{-y_i w_i}{1 + y_i w_i}\right)$ (well defined for $-y_i w_i \in (0,1)$); 
\textbf{Fenchel--Young loss:}
    \begin{equation}
    \notag
        L_F(h \parallel w) = \sum_{i=1}^n \left( \log(1 + e^{-y_i h_i}) + (-y_i w_i) \log(-y_i w_i) + (1 + y_i w_i) \log(1 + y_i w_i) - w_i h_i \right).
    \end{equation}

    This loss $F$ is finite-valued, $\calC^1$ and strictly convex but not coercive, since $F\to0$ along correctly classified directions, so it lies outside $\calF_0$; because $\dom F=\calH$, the Fenchel--Rockafellar qualification still gives strong duality and dual attainment. The Fenchel--Young error bounds of Theorem~\ref{thm:FY_discrepancy_error_bounds} hold as stated, since they need only a proper, l.s.c.\ convex $F$ and a source-condition certificate (here $\dom F=\calH$, so $\objmeas(\mu^\star)<+\infty$ is automatic); leaving $\calF_0$ costs only the weak-coercivity bound~\eqref{eq:hyp_coercivite_faible}. Since $F$ is $\calC^1$ and strictly convex, the Mirror-Alignment characterization and the uniqueness of the aligned certificate still hold.
\end{subequations}

\subsection{Poisson loss (exponential link)}
\begin{subequations}
Let $y \in \mathbb{R}_+^n$ be count data. 

\noindent
\textbf{Primal:} $F(h) = \sum_{i=1}^n (\exp(h_i) - y_i h_i)$; 
\textbf{Gradient:} $\nabla F(h)_i = \exp(h_i) - y_i$; 
\textbf{Dual:} Let $w \in \mathbb{R}^n$. The domain is restricted to $\{w \mid w + y \ge 0\}$ (element-wise),
    \begin{equation}
    \notag
        F^\star(w) = \sum_{i=1}^n \left[ (w_i + y_i) \log(w_i + y_i) - (w_i + y_i) \right];
    \end{equation}
\textbf{Dual gradient:} $\nabla F^\star(w)_i = \log(w_i + y_i)$; 
\textbf{Fenchel--Young loss:}
    \begin{equation}
    \notag
        L_F(h \parallel w) = \sum_{i=1}^n \left( e^{h_i} + (w_i+y_i)\log\left(\frac{w_i+y_i}{e^{h_i}}\right) - (w_i+y_i) \right).
    \end{equation}
Equivalently, $L_F(h\parallel w)=\sum_i \KL(w_i+y_i\,\|\,e^{h_i})$ with the dual-shifted observation $w+y$ as the first argument and the predicted intensity $e^{h}$ as the second; the order matters because $\KL$ is asymmetric.

This formulation is standard for Poisson regression where the intensity parameter is modeled as $\nu = \exp(h)$ to enforce positivity (we write $\nu$ for the intensity to avoid collision with the regularization parameter $\lambda$).  This loss $F$ is finite-valued, $\calC^1$ and strictly convex, and it lies in $\calF_0$ provided every count is strictly positive ($y>0$); if some $y_i=0$ the coercivity bound~\eqref{eq:hyp_coercivite_faible} fails as $h_i\to-\infty$, so $F$ leaves $\calF_0$; the loss remains $\calC^1$ and strictly convex, so the Mirror-Alignment characterization and the uniqueness of the aligned certificate are unaffected, while the Fenchel--Young error bounds of Theorem~\ref{thm:FY_discrepancy_error_bounds} still hold given a source-condition certificate ($\dom F=\calH$, so $\objmeas(\mu^\star)<+\infty$).
\end{subequations}

\subsection{Kullback--Leibler divergence (linear link)}
\begin{subequations}
Here we assume the operator $\Phi$ maps to a positive domain, $h \in \mathbb{R}_+^n$, representing density or intensity directly. 

\noindent
\textbf{Primal:} $F(h) = \sum_{i=1}^n \left( h_i \log\left(\frac{h_i}{y_i}\right) - h_i + y_i \right)$; 
\textbf{Gradient:} $\nabla F(h)_i = \log(h_i) - \log(y_i)$; 
\textbf{Dual:} $F^\star(w) = \sum_{i=1}^n y_i (e^{w_i} - 1)$; 
\textbf{Dual gradient:} $\nabla F^\star(w)_i = y_i e^{w_i}$; 
\textbf{Fenchel--Young loss:}
    \begin{equation}
    \notag
        L_F(h \parallel w) = \sum_{i=1}^n \left( h_i \log\left(\frac{h_i}{y_i e^{w_i}}\right) - h_i + y_i e^{w_i} \right) = \KL(h \parallel y e^w).
    \end{equation}

    This loss $F$ is the Csisz{\'a}r I-divergence. Because $\dom F=\mathbb{R}_+^n\subsetneq\calH$, this fidelity lies outside $\calF_0$. The Fenchel--Young error bounds of Theorem~\ref{thm:FY_discrepancy_error_bounds} hold given a source-condition certificate, provided $\Phi\mu^\star\in\dom F$, that is $\objmeas(\mu^\star)<+\infty$: on the positive cone $\{\Phi\mu>0\}$ the loss $F$ is finite, $\calC^1$ and strictly convex, so the Mirror-Alignment and uniqueness statements also apply there.
\end{subequations}

\subsection{Huber loss (robust regression)}
\label{sec:datafidelity_FYL_app:huber}
\begin{subequations}
The Huber loss is standard in robust statistics for data with outliers. Let $y \in \mathbb{R}^n$ be the observed data. It behaves like a squared error for small residuals and like an absolute error for large residuals. Let $\delta > 0$ be the robustness threshold.

\noindent
\textbf{Primal:} The function is defined component-wise as $F(h) = \sum_{i=1}^n \rho_\delta(h_i - y_i)$, where:
    \begin{equation}
    \notag
        \rho_\delta(r) = \begin{cases}
            \frac{1}{2}r^2 & \text{if } |r| \le \delta, \\
            \delta(|r| - \frac{1}{2}\delta) & \text{otherwise;}
        \end{cases}
    \end{equation}
\textbf{Gradient:} The gradient acts as a ``soft'' saturation function (censoring large residuals):
    \begin{equation}
    \notag
        [\nabla F(h)]_i = \begin{cases}
            h_i - y_i & \text{if } |h_i - y_i| \le \delta, \\
            \delta \cdot \operatorname{sign}(h_i - y_i) & \text{otherwise;}
        \end{cases}
    \end{equation}
\textbf{Dual:} The convex conjugate is the squared norm restricted to a box constraint, so robust regression is dual to constrained least squares:
    \begin{equation}
    \notag
        F^\star(w) = \begin{cases}
            \frac{1}{2}\|w\|_\calH^2 + \langle w, y \rangle_\calH & \text{if } \|w\|_\infty \le \delta, \\
            +\infty & \text{otherwise;}
        \end{cases}
    \end{equation}
\textbf{Dual subdifferential:} For interior $w$ (i.e., $\|w\|_\infty<\delta$), $\nabla F^\star(w)$ is single-valued with $[\nabla F^\star(w)]_i = w_i + y_i$. For boundary $w$ with some $|w_i|=\delta$, $\partial F^\star(w)$ is set-valued: $\partial F^\star(w) = y + w + N_{B_\infty^\delta(0)}(w)$, and any selection in this set serves as a Fenchel--Young dual gradient;
\textbf{Fenchel--Young loss:}
    The loss $L_F(h \parallel w)$ penalizes the discrepancy between the residual and the saturated dual variable. For any $w$ such that $\|w\|_\infty \le \delta$:
    \begin{equation}
    \notag
        L_F(h \parallel w) = \sum_{i=1}^n \left( \rho_\delta(h_i - y_i) + \frac{1}{2}w_i^2 + w_i y_i - w_i h_i \right).
    \end{equation}
    If we choose the optimal dual $w = \nabla F(h)$ (the saturated residual), the loss vanishes.

     This loss $F$ is $\calC^1$ and coercive but not strictly convex, since it is affine in the large-residual regime, so it lies outside $\calF_0$. The gap decomposition, strong duality and the Fenchel--Young error bounds of Theorem~\ref{thm:FY_discrepancy_error_bounds} still hold ($\dom F=\calH$); only the optimal prediction $\Phi\mu^\star$ and the Mirror-Alignment identity~\eqref{eq:mirror_alignment_dual_certificate} are pinned down up to the flat directions of $F$.
\end{subequations}

\subsection{Squared hinge loss (L2-SVM)}
\label{sec:datafidelity_FYL_app:hinge}
\begin{subequations}
This loss function is a differentiable alternative to the standard hinge loss, often used to enforce classification margins while maintaining a smooth objective for optimization. Let $y \in \{-1, 1\}^n$ be binary labels.

\noindent\textbf{Primal:} $F(h) = \frac{1}{2} \sum_{i=1}^n \max(0, 1 - y_i h_i)^2$; 
\textbf{Gradient:} The gradient is:
    \begin{equation}
    \notag
        [\nabla F(h)]_i = -y_i \max(0, 1 - y_i h_i);
    \end{equation}
\textbf{Dual:} The convex conjugate involves a quadratic term restricted to a specific sign constraint. For the dual variable $w \in \mathbb{R}^n$, the domain is restricted to $K_y\coloneqq \{w \mid -y_i w_i \ge 0 \text{ for all } i\}$ (i.e., $w$ and $y$ must have opposite signs),
    \begin{equation}
    \notag
        F^\star(w) = \begin{cases}
            \sum_{i=1}^n \left( \frac{1}{2} w_i^2 + y_i w_i \right) & \text{if } -y_i w_i \ge 0 \text{ for all } i, \\
            +\infty & \text{otherwise;}
        \end{cases}
    \end{equation}
\textbf{Dual subdifferential:} For $w$ in the relative interior of $K_y$ (i.e., $-y_iw_i>0$ for every $i$), $\nabla F^\star$ is single-valued with $[\nabla F^\star(w)]_i = w_i + y_i$. At boundary points where some $w_i=0$, $\partial F^\star(w)=y+w+N_{K_y}(w)$ is set-valued; the formula below selects one element of this set.
\textbf{Fenchel--Young loss:}
    The loss $L_F(h \parallel w)$ measures the gap between the margin error and the dual certificate. For feasible $w$:
    \begin{equation}
    \notag
        L_F(h \parallel w) = \sum_{i=1}^n \left( \frac{1}{2}\max(0, 1 - y_i h_i)^2 + \frac{1}{2}w_i^2 + y_i w_i - w_i h_i \right).
    \end{equation}
    It measures how well the dual variable $w$ certifies the classification margin of the prediction~$h$.

     This loss $F$ is $\calC^1$ but neither coercive nor strictly convex, since it is flat along correctly classified directions, so it lies outside $\calF_0$. The Fenchel--Young error bounds of Theorem~\ref{thm:FY_discrepancy_error_bounds} still hold ($\dom F=\calH$), as they do not rely on the aligned base point; only the optimal prediction $\Phi\mu^\star$ and the Mirror-Alignment identity~\eqref{eq:mirror_alignment_dual_certificate} are pinned down up to those flat directions.
\end{subequations}

\end{document}